\documentclass [twoside,reqno, 11pt] {amsart}
\usepackage{tikz-cd}
\usepackage{amsfonts}

\calclayout

\usepackage[utf8]{inputenc}
\usepackage[T1]{fontenc}

\usepackage{color}
\usepackage{graphicx}
\usepackage{amssymb}
\usepackage{amsmath}
\usepackage{amsthm}
\usepackage{mathtools}
\usepackage{mathrsfs}
\usepackage{hyperref}
\usepackage{tikz}
\usepackage{xfrac}
\usepackage[super]{nth}

\mathtoolsset{showonlyrefs}

\newcommand{\C}{\mathbb C}

\newcommand{\R}{\mathbb R}
\newcommand{\N}{\mathbb N}
\newcommand{\Z}{\mathbb Z}

\newcommand{\de}{\, \mathrm{d}}

\newcommand{\del}{\partial}

\newcommand{\pardiff}[2]{\frac{\partial #1}{\partial #2}}

\newcommand{\norm}[1]{\left\Vert #1 \right\Vert}
\newcommand{\abs}[1]{\left| #1 \right|}
\newcommand{\dual}[2]{\langle #1 , #2 \rangle}

\newcommand{\ceil}[1]{\left\lceil #1 \right\rceil}

\newcommand{\ot}{\leftarrow}
\newcommand{\I}{\mathrm{i}}
\newcommand{\Dir}{\mathrm{Dir}}

\DeclareMathOperator{\re}{Re}
\DeclareMathOperator{\im}{Im}

\DeclareMathOperator{\intr}{int}

\DeclareMathOperator{\dist}{dist}

\DeclareMathOperator{\id}{id}

\DeclareMathOperator{\tr}{tr}

\DeclareMathOperator{\End}{End}

\DeclareMathOperator{\DN}{DN}

\newtheorem{thm}{Theorem}[section]
\newtheorem{prop}[thm]{Proposition}
\newtheorem{lem}[thm]{Lemma}
\newtheorem{coro}[thm]{Corollary}

\newtheorem{thm*}{Theorem}
\newtheorem{prop*}[thm*]{Proposition}
\newtheorem{lem*}[thm*]{Lemma}
\newtheorem{coro*}[thm*]{Corollary}
\newtheorem{conj*}[thm*]{Conjecture}

\theoremstyle{definition}

\newtheorem{defi}[thm]{Definition}

\title{Gaussian beams and inverse problems for connections at high fixed frequency}
\author{Simon St-Amant}

\address
       {S. St-Amant, Department of Pure Mathematics and Mathematical Statistics\\ 
       University of Cambridge, Cambridge CB3 0WB, UK}
\email{sas242@cam.ac.uk}

\begin{document}

\begin{abstract}
We consider the inverse problem of recovering a connection on a complex vector bundle over a compact smooth Riemannian manifold with boundary from a Dirichlet-to-Neumann (DN) map at a high fixed frequency. We construct Gaussian beams using the language of jet bundles and show that their value at the boundary can be recovered from those DN maps. This allows us to show injectivity up to gauge on manifolds whose non-abelian X-ray transform is injective. We also study DN maps with a cubic nonlinearity and show how to recover the broken non-abelian X-ray transform from them. This transform maps a connection to its parallel transport along broken geodesics with endpoints on the boundary. We show that the broken non-abelian X-ray transform is always injective up to gauge equivalence, regardless of the manifold's geometry.
\end{abstract}

\maketitle

\section{Introduction}

\subsection{Setting}

Let $(M, g)$ be an $m$-dimensional compact oriented smooth Riemannian manifold with smooth boundary $\del M$ and let $\pi : E \to M$ be a smooth complex vector bundle of rank $n$ over $M$ equipped with a fibre metric $\dual{\cdot}{\cdot}_E$. We denote by $\Gamma(E)$ or $\Gamma(\pi)$ the space of smooth sections of $E$. Let $\nabla : \Gamma(E) \to \Gamma(T^*M \otimes E)$ be a smooth connection on $E$ that is compatible with the fibre metric, and let $\Delta = \nabla^* \nabla$ be the connection Laplacian, where $\nabla^* : \Gamma(T^*M \otimes E) \to \Gamma(E)$ is the formal adjoint of $\nabla$ with respect to the $L^2$ inner products on $E$ and $T^*M \otimes E$. The Dirichlet-to-Neumann map $\DN : \Gamma(E\vert_{\del M}) \to \Gamma(E\vert_{\del M})$ is the map that sends $f \in \Gamma(E\vert_{\del M})$ to $(\nabla_\nu u_f) \vert_{\del M}$, where $\nu$ is the outward normal to $\del M$ and $u_f$ is the unique solution to
\begin{equation}
\begin{cases}
	\Delta u = 0, \\
	u \vert_{\del M} = f.
\end{cases}
\end{equation}
The Calderón problem for connections asks whether one can recover the connection $\nabla$ from the knowledge of its associated Dirichlet-to-Neumann map. There is a natural gauge which prevents the unique recovery of $\nabla$. Indeed, if $\varphi \in \Gamma(\mathrm{Aut}(E))$, where $\mathrm{Aut}(E)$ is the bundle of endomorphisms of $E$ that preserve the fibre metric, and $\varphi \vert_{\del M} = \id$, then the pullback connection $\varphi^* \nabla = \varphi^{-1} \nabla \varphi$ is also compatible with the fibre metric and defines the same Dirichlet-to-Neumann map. We say that two compatible connections $\nabla_1$ and $\nabla_2$ are gauge equivalent if there exists a gauge $\varphi \in \Gamma(\mathrm{Aut}(E))$ such that $\nabla_2 = \varphi^* \nabla_1$.

In the case where $M$ is a surface, the connection can always be recovered up to gauge, as shown in \cite{systems2d}. If $\dim M \geq 3$ and all the quantities involved are real-analytic, then the result also holds \cite{gabdurakhmanov2023calderons}. In the smooth case, it is shown in \cite{limiting} that one can recover $\nabla$ up to gauge equivalence when $E$ is an admissible Hermitian line bundle over a conformally transversally anisotropic (CTA) manifold with a transversal manifold that is simple. A CTA manifold of dimension $m$ is a manifold that can be conformally embedded into a product manifold $(\R \times M_0, e \oplus g_0)$ for some transversal manifold $(M_0, g_0)$ of dimension $m-1$, thus having a distinguished Euclidean direction. This result was further improved in \cite{cekic} for the case where the ray transforms for functions and one-forms are injective on the transversal manifold. In \cite{cekic_ray_transform}, a similar result is obtained in the case where $(M, g)$ is isometrically contained in the interior of $(\R^2 \times M_0, c(e \oplus g_0))$ for some conformal factor $c$, without any assumption on the transversal manifold $(M_0, g_0)$. And in \cite{cekicYM}, it is shown that a unitary connection $\nabla$ on a Hermitian vector bundle over any compact smooth manifold with boundary can be recovered from the knowledge of the DN map if it also solves the Yang-Mills equations. There are also results on the hyperbolic Calderón problem for connections, where the underlying equation is a wave equation, see \cite{kop}. For a survey on the Calderón problem in the Euclidean setting, see \cite{uhlmannsurvey}.

\subsection{Linear problem}

We will investigate a version of the Calderón problem with a frequency $\lambda$. For $\lambda \geq 0$, we define the Dirichlet-to-Neumann map $\DN_\lambda : \Gamma(E\vert_{\del M}) \to \Gamma(E\vert_{\del M})$ as the map sending $f \in \Gamma(E\vert_{\del M})$ to $(\nabla_\nu u_f) \vert_{\del M}$, where $u_f$ is now solution to
\begin{equation}
\begin{cases}
	(\Delta - \lambda^2) u = 0, \\
	u \vert_{\del M} = f.
\end{cases}
\end{equation}
As long as $\lambda^2$ is not a Dirichlet eigenvalue of $\Delta$, this problem has a unique solution. We are interested in showing that given two compatible connections $\nabla_1$ and $\nabla_2$, there is $\lambda_0 = \lambda_0(\nabla_1, \nabla_2)$ such that $\nabla_1$ and $\nabla_2$ must be gauge equivalent if their Dirichlet-to-Neumann maps $\DN_\lambda$ agree for some $\lambda > \lambda_0$.

The frequency parameter $\lambda$  was recently introduced in \cite{uhlmann2021anisotropic} to study the anisotropic Calderón problem on smooth simply connected manifolds of dimension $3$ with smooth strictly convex boundaries and non-positive sectional curvatures. They show that for any two compactly supported smooth potentials $V_1$ and $V_2$, there exists $\lambda_0 = \lambda_0(V_1, V_2) > 0$ such that if the DN maps associated with the equations $(\Delta_g + V_i - \lambda^2)u = 0$ agree for some admissible $\lambda > \lambda_0$, then $V_1 = V_2$. Importantly, the frequency $\lambda_0$ is large but fixed for any pair of potentials. This result was extended in \cite{ma2023anisotropic} to merely differentiable potentials on compact non-trapping manifolds of dimension $m \geq 2$ for which the geodesic ray transform is stably invertible and continuous. A similar problem is studied in \cite{katya} for the equation $(\Delta_g - \lambda^2)u + V_i u^3 = 0$ instead. There, they show similar results at high but fixed frequency, but with significantly milder assumptions on the geometry of the manifold. To our knowledge, this is the first text that considers the Calderón problem for connections at high but fixed frequency.

When the bundle is trivial, that is, $E = M \times \C^n$, and the fibre metric is the standard Hermitian inner product, any compatible connection $\nabla$ on $E$ can be expressed as $\nabla = d + A$ for a smooth $\mathfrak{u}(n)$-valued one-form $A \in \Omega^1(M, \mathfrak{u}(n))$. Here, $d$ is the exterior derivative acting componentwise, and $\mathfrak{u}(n)$ is the set of skew-Hermitian matrices. We call $A$ a Hermitian connection. The connection Laplacian $\Delta_A$ can then be written as
\begin{equation}
	\Delta_A u = \Delta_g u - 2(A, du) + (d^*A)u - (A, Au)
\end{equation}
where $\Delta_g$ is the positive-definite Laplace-Beltrami operator acting componentwise. The gauge group then corresponds to smooth maps $\varphi : M \to \mathrm{U}(n)$ such that $\varphi \vert_{\del M} = I$ and they act from the right on a connection $A$ by
\begin{equation}\label{eq:gauge}
	A \triangleleft \varphi = \varphi^{-1}d\varphi + \varphi^{-1} A \varphi.
\end{equation}

Given two Hermitian connections $A_1$ and $A_2$, and $\delta > 0$, we let $J \subset [1, \infty)$, $J = J(M, g, A_1, A_2, \delta)$, be a set of frequencies of Lebesgue measure $\abs{J} < \delta$ such that for some $C > 0$, the resolvent estimate
\begin{equation}\label{eq:resolvent_intro}
	\norm{(\Delta_{A_{\ell}, \Dir} - \lambda^2)^{-1}}_{\mathcal{L}(L^2(M), L^2(M))} \leq C \lambda^{m+1}
\end{equation}
holds for all $\lambda \in [1, \infty) \setminus J$, $\ell = 1, 2$. Here, $\Delta_{A_\ell, \Dir}$ is the Dirichlet Laplacian corresponding to $A_\ell$, that is, it acts as $\Delta_{A_\ell}$ but its domain is restricted to $H_0^1(M) \cap H^2(M)$. The set $J$ is constructed as in \cite[Theorem 2.1]{katya} and contains the Dirichlet eigenvalues of both $\Delta_{A_1}$ and $\Delta_{A_2}$. See Section \ref{sec:solvability} for more details.

If $M$ is contractible, then there is no loss in considering trivial bundles as every bundle over $M$ is trivialisable. Our main result concerns manifolds that fall into one of two categories:
\begin{enumerate}
	\item[(i)] $m = 2$ and $M$ is simple, that is, $M$ has a strictly convex boundary, is nontrapping, and has no conjugate points.
	\item[(ii)] $m \geq 3$, $\del M$ is strictly convex, $M$ admits a strictly convex function, and geodesics on $M$ do not self-intersect on the boundary.
\end{enumerate}

Note that the condition in (ii) asking that geodesics do not self-intersect on the boundary is technical, and is satisfied in all known examples where $M$ has a strictly convex boundary and admits a strictly convex function. For example, if $M$ is simple and admits a strictly convex function, then (ii) holds. At the moment, the question of whether all simple manifolds admit a strictly convex function is open.

\begin{thm}\label{thm:calderon}
Let $(M, g)$ be as above and suppose that {\normalfont (i)} or {\normalfont (ii)} holds. Let $A_1$ and $A_2$ be smooth Hermitian connections that agree on a neighbourhood $U \subset M$ of the boundary $\del M$. For $\delta > 0$, let $J$ be as in \eqref{eq:resolvent_intro} with $\abs{J} < \delta$. Unless $A_1$ and $A_2$ are gauge equivalent, there exists a constant $\lambda_0 = \lambda_0(M, g, U, A_1, A_2, \delta)$ such that $\DN_\lambda^{A_1} \neq \DN_\lambda^{A_2}$ for all $\lambda \in (\lambda_0, \infty)\setminus J$.
\end{thm}

This result can be equivalently rephrased as the following. Given two connections $A_1$ and $A_2$ that agree on $U$, there exists $\lambda_0$ such that if $\DN_\lambda^{A_1} = \DN_\lambda^{A_2}$ for some $\lambda \in (\lambda_0, \infty) \setminus J$, then $A_1$ and $A_2$ are gauge equivalent. Both formulations are equivalent but we chose the one in the Theorem because we think it is clearer.

We require that the connections agree on a neighbourhood $U$ of the boundary to avoid dealing with short geodesics that make small angles with the boundary as our analysis is not suitable for such geodesics. Note that the case where $M$ satisfies (i) follows from \cite[Theorem 1]{systems2d} but we include it here for completeness since our approach is different.

The proof of Theorem \ref{thm:calderon} hinges on the control of the boundary values of Gaussian beams, which are approximate solutions to the equation $(\Delta - \lambda^2)u = 0$. A Gaussian beam concentrates along a geodesic and is constructed by recursively solving ODEs along that geodesic. We explicitly construct Gaussian beams on a vector bundle in Section \ref{sec:GB}. Our construction is different from other Gaussian beam constructions because we take care in prescribing the values of the beams on the boundary, and we view our solutions globally as sections of a jet bundle associated with the geodesic. We highlight the main points of our construction below.

Let $\gamma: [0, \tau] \to M$ be a geodesic such that $\gamma(0) \in \del M$, $\gamma(\tau) \in \del M$, $\gamma(t) \in M^{\intr}$ for $t \in (0, \tau)$, and both $\dot{\gamma}(0)$ and $\dot{\gamma}(\tau)$ are transversal to the tangent space of the boundary $T \del M$. We call such a geodesic nontangential. To construct a Gaussian beam along $\gamma$, we first construct an auxiliary manifold $N_\gamma$ that is locally isometric to a tubular neighbourhood of $\gamma$ through a map $\pi_\gamma : N_\gamma \to M$ and such that the curve $\gamma$ in $N_\gamma$ does not self-intersect. We then look for compactly supported functions on $N_\gamma$ of the form
\begin{equation}\label{eq:intro_u}
	u = e^{\I\lambda \phi}(a_0 + \lambda^{-1} a_1 + \dots + \lambda^{-K} a_K)
\end{equation}
that solve $(\Delta - \lambda^2)u = 0$ up to high order along $\gamma$ with $a_j$ having prescribed derivatives at $\gamma(0)$. This is done using the language of jet bundles. The Gaussian beam on $M$ is then given by
\begin{equation}
	((\pi_\gamma)_*u)(x) := \sum_{q \in \pi_\gamma^{-1}(x)} u(q).
\end{equation}

Our main technical result is Theorem \ref{thm:GBtraces}. A major component of its proof relies on constructing remainders $r$ for Gaussian beams $u$ such that $(\Delta - \lambda^2)(u + r) = 0$ exactly. Importantly, we show that we can choose $r$ such that we can bound its $H^{2k}(M)$ norm with the $H^{2k-2}(M)$ norm of the error $(\Delta - \lambda^2)u$, and $r$ vanishes on the boundary. Therefore, the boundary values of a Gaussian beam are the same as the boundary values of an actual solution nearby, and that solution gets closer as we increase the order $K$ of the Gaussian beam. See Section \ref{sec:solvability} for more details. From Theorem \ref{thm:GBtraces}, one can show the following.

\begin{thm}\label{thm:GB_trace_single_geodesic}
Let $(M, g)$ be as above, let $E = M \times \C^n$, let $\gamma : [0, \tau] \to M$ be a nontangential geodesic with $\gamma(0) \neq \gamma(\tau)$, and fix $K \geq 0$. Let $A_1$ and $A_2$ be smooth Hermitian connections on $E$. For $\delta > 0$, let $J$ be as in \eqref{eq:resolvent_intro} with $\abs{J} < \delta$. Let
\begin{equation}
	u_\ell = e^{\I\lambda\phi}(a_0^{(\ell)} + \lambda^{-1} a_1^{(\ell)} + \dots + \lambda^{-K} a_K^{(\ell)})
\end{equation}
be Gaussian beams on $N_\gamma$ with respect to $A_\ell$, $\ell = 1, 2$. Suppose that $a_j^{(1)}$ and $a_j^{(2)}$ agree up to order $K$ at $\gamma(0)$ for $0 \leq j \leq K$. Unless $a_j^{(1)}$ and $a_j^{(2)}$ also agree up to order $K$ at $\gamma(\tau)$ for $0 \leq j \leq K$, there exists a constant $\lambda_0 = \lambda_0(M, g, \gamma, A_1, A_2, K, \delta)$ such that $\DN_\lambda^{A_1} \neq \DN_\lambda^{A_2}$ for all $\lambda \in (\lambda_0, \infty) \setminus J$.
\end{thm}

Let us now elaborate on how Theorem \ref{thm:calderon} follows from Theorem \ref{thm:GB_trace_single_geodesic}. On the geodesic $\gamma$, the coefficient $a_0$ in \eqref{eq:intro_u} associated to a Hermitian connection $A$ is equal to
\begin{equation}
	a_0(\gamma(t)) = \exp \left(\frac{1}{2} \int_0^t \tr(H(s)) \de s\right) P^A_{\gamma[0, t]} a_0(\gamma(0))
\end{equation}
where $H(s)$ is a matrix determined by the geometry of $(M, g)$ independently of the connection $A$ and $P^A_{\gamma[0,t]}$ is the parallel transport map induced by $A$, defined as follows.

For a smooth curve $\gamma : [0, \tau] \to M$, consider $U : [0,\tau] \to \C^{n \times n}$ the unique solution to
\begin{equation}
\begin{cases}
	\dot{U}(t) + A_{\gamma(t)}(\dot{\gamma}(t)) U(t) = 0, \\
	U(0) = I.
\end{cases}
\end{equation}
The parallel transport map of $A$ along $\gamma$ from time $0$ up to time $t$, denoted $P^A_{\gamma[0, t]} : \C^n \to \C^n$, is defined as $P^A_{\gamma[0, t]} = U(t)$. Since $A$ takes values in $\mathfrak{u}(n)$, $P^A_{\gamma[0,t]}$ takes values in the set of unitary matrices $\mathrm{U}(n)$. We will use the notation $P^A_{\gamma}$ to mean parallel transport along the whole curve $\gamma$ from time $0$ to $\tau$. Note that if $A_2 = A_1 \triangleleft \varphi$ as in \eqref{eq:gauge}, then $P^{A_2}_\gamma = \varphi^{-1}(\gamma(T)) P^{A_1}_\gamma \varphi(\gamma(0))$.

Hence, if $A_1$ and $A_2$ are Hermitian connections on $E = M \times \C^n$ with $a_0^{(1)}(\gamma(\tau)) = a_0^{(2)}(\gamma(\tau))$ whenever $a_0^{(1)}(\gamma(0)) = a_0^{(2)}(\gamma(0))$, then $P^{A_1}_\gamma = P^{A_2}_\gamma$. The choice of $\lambda_0$ in Theorem \ref{thm:GB_trace_single_geodesic} depends on the geodesic $\gamma$. However, by appropriately controlling the quantities on which $\lambda_0$ depends, see Theorem \ref{thm:GBtraces}, we can choose $\lambda_0$ uniformly across all nontangential geodesics that are not entirely contained within $U$. Hence, if $A_1$ agrees with $A_2$ on a neighbourhood of the boundary, unless $P^{A_1}_\gamma = P^{A_2}_\gamma$ for all nontangential geodesics $\gamma$, then there is $\lambda_0$ such that $\DN_\lambda^{A_1} \neq \DN_\lambda^{A_2}$ for all $\lambda \in (\lambda_0, \infty) \setminus J$. The data of the map $\gamma \mapsto P^A_{\gamma}$ over nontangential geodesics $\gamma$ is called the non-abelian X-ray transform of the connection $A$. If $M$ is a simple surface \cite{paternain2022nonabelian} or, for $\dim M \geq 3$, if $M$ has a strictly convex boundary and admits a smooth strictly convex function \cite{matrixweights}, the non-abelian X-ray transform is injective up to gauge equivalence. Stability estimates and statistical consistency results related to both cases can be found in \cite{MNP} and \cite{bohrstability}. Therefore, if $(M, g)$ satisfies (i) or (ii), then $P^{A_1}_\gamma = P^{A_2}_\gamma$ for all nontangential geodesics $\gamma$ if and only if the connections $A_1$ and $A_2$ are gauge equivalent.

\subsection{Nonlinear problem}

We will also consider a nonlinear version of the problem above. For a Hermitian connection $A$ on the trivial bundle $E = M \times \C^n$, we let $\Lambda^A_\lambda : \mathrm{dom}(\Lambda_\lambda^A) \subset C^\infty(\del M;\C^n) \to C^\infty(\del M;\C^n)$ be the map given by $\Lambda_\lambda^A f = \nabla_\nu u_f \vert_{\del M}$ where $u_f$ now solves
\begin{equation}\label{eq:nonlinear_equations}
\begin{cases}
	(\Delta_A - \lambda^2)u + \abs{u}^2 u = 0, \\
	u\vert_{\del M} = f.
\end{cases}
\end{equation}
By the same arguments as in the proof of \cite[Proposition 2.1]{IPelliptic}, see also \cite[Theorem B.1]{ku_nonlinear_magnetic}, as long as $\lambda^2$ is not a Dirichlet eigenvalue of $\Delta_A$, for any $s > 1$, $s \not\in \Z$, there are $\delta > 0$ and $C > 0$ such that for all $f \in C^\infty(\del M, \C^n)$ with $\norm{f}_{C^s(\del M)} < \delta$, problem \eqref{eq:nonlinear_equations} has a unique solution $u_f$ satisfying
\begin{equation}
	\norm{u_f}_{C^s(M)} \leq C \norm{f}_{C^s(\del M)}.
\end{equation}
Here, $C^s(\del M)$ is the Hölder space $C^{k, \alpha}(\del M)$ with $s = k + \alpha$, $k \in \Z$, $0 < \alpha < 1$. If $A_1$ and $A_2$ are gauge equivalent, then $\Lambda_\lambda^{A_1} = \Lambda_\lambda^{A_2}$. We show the following.

\begin{thm}\label{thm:calderon_cubic}
	Let $(M, g)$ be a compact manifold with smooth boundary that satisfies condition (H). Let $A_1$ and $A_2$ be smooth Hermitian connections on the trivial bundle $E = M \times \C^n$. For $\delta > 0$, let $J$ be as in \eqref{eq:resolvent_intro} with $\abs{J} < \delta$. Unless $A_1$ and $A_2$ are gauge equivalent, there exists a constant $\lambda_0 = \lambda_0(M, g, A_1, A_2, \delta)$ such that $\Lambda^{A_1}_\lambda \neq \Lambda^{A_2}_\lambda$ for all $\lambda \in (\lambda_0, \infty) \setminus J$.
\end{thm}

The condition (H) in Theorem \ref{thm:calderon_cubic} is technical, so we omit its precise statement from the Introduction. It can be found in Section \ref{sec:technical_conditions}. In effect, it states that at any given point $x \in M^{\intr}$, we can find sufficiently many pairs of nontangential geodesics that intersect only at $x$ at an angle bounded away from $0$. Moreover, these geodesics have a bounded length, they intersect the boundary at an angle bounded away from $0$, and they do not get arbitrarily close to the boundary away from their endpoints. Among other examples, condition (H) holds for simple manifolds.

Comparing Theorems \ref{thm:calderon} and \ref{thm:calderon_cubic}, we see that the geometric requirements in Theorem \ref{thm:calderon_cubic} are weaker. This is one of many examples that shows how nonlinear terms actually help when solving inverse problems. Indeed, the presence of a nonlinear term is often used as a tool to help solve inverse problems related to elliptic or hyperbolic equations, as seen in, for example, \cite{IPelliptic, IPsemilinear, isakovsylvester, sun}, and \cite{klu, luw, chen2022detection}.

The proof of Theorem \ref{thm:calderon_cubic} relies on controlling the boundary values of Gaussian beams and deriving an integral identity coming from the third-order linearisation of $\Lambda^A_\lambda$ to recover the broken non-abelian X-ray transform of a $A$. This transform, denoted $S^A$, maps broken geodesics with endpoints on the boundary to the parallel transport along them. For us, a broken geodesic is a curve composed of two nontangential geodesics that intersect transversally at a single point inside $M^{\intr}$. See Section \ref{sec:broken} for more details. Our main Theorem regarding the broken non-abelian X-ray transform is Theorem \ref{thm:injectivity_broken}, which implies the following.

\begin{thm}\label{thm:injectivity_broken_intro}
	Let $A_1$ and $A_2$ be smooth Hermitian connections on $E = M \times \C^n$. The broken non-abelian X-ray transforms of $A_1$ and $A_2$ agree over all broken geodesics if and only if $A_1$ and $A_2$ are gauge equivalent.
\end{thm}

A similar X-ray transform along broken light rays in Minkowski space was used to study inverse problems related to the Yang-Mills-Higgs equations, see \cite{chen2021inverse}, \cite{chen2022detection} and \cite{chen2022retrieving}. For stability estimates and statistical consistency results related to that broken light ray transform, see \cite{simon}.

\subsection{Plan of the paper} The paper is organised as follows. We develop the theory of Gaussian beams on vector bundle using jet bundles in Section \ref{sec:GB} with relevant estimates in Section \ref{sec:GBestimates} and solvability results in Section \ref{sec:solvability}. We show Theorems \ref{thm:calderon} and \ref{thm:GB_trace_single_geodesic} in Section \ref{sec:traces} by studying the traces of Gaussian beams related to connections whose $\DN_\lambda$ maps agree. We study the broken non-abelian X-ray transform in Section \ref{sec:broken} and show Theorem \ref{thm:injectivity_broken_intro}. Finally, we show Theorem \ref{thm:calderon_cubic} in Section \ref{sec:calderon_cubic}.

\subsection*{Acknowledgements} The author would like to thank Lauri Oksanen for suggesting the nonlinear problem studied here. The author would also like to thank Katya Krupchyk, Mikko Salo, Mihajlo Cekić, and Gabriel Paternain for their helpful comments and discussions. The author was partially supported by the Natural Sciences and Engineering Research Council of Canada.

\section{Gaussian beams on a vector bundle} \label{sec:GB}

\subsection{Localisation}\label{sec:localisation}

 Let $\gamma: [0, \tau] \to M$ be a geodesic such that $\gamma(0) \in \del M$, $\gamma(\tau) \in \del M$, $\gamma(t) \in M^{\intr}$ for $t \in (0, \tau)$, and both $\dot{\gamma}(0)$ and $\dot{\gamma}(\tau)$ are transversal to the tangent space of the boundary $T \del M$. We call such a geodesic nontangential, and we denote by $\mathcal{G}_T$ the set of all nontangential geodesics of length $\tau < T$. We are interested in constructing approximate solutions to the equation $(\Delta - \lambda^2) u = 0$ that concentrate along $\gamma$ and become increasingly precise as $\lambda$ grows. We wish to do so in a global setting so that we can control our solutions uniformly with respect to $E, M, g, \nabla$, and $\tau$. However, as $\gamma$ might self-intersect, we cannot work directly over $M$.

Let $N$ be a closed extension of $M$ and assume $E$ is the restriction of a vector bundle on $N$ with a connection that restricts to $\nabla$ on $M$. We will not distinguish between $E$ or $\nabla$ and their extensions to make notation less cumbersome. Let $0 < \varepsilon < \mathrm{conv}(N)/2$ where $\mathrm{conv}(N)$ is the convexity radius of $N$, that is, the unique positive constant such that every ball of radius $r < \mathrm{conv}(N)$ is geodesically convex. For $(x, v) \in SN$, define
\begin{equation}
	U(x, v, \delta) = \{(t, y) : t \in (-\delta, \varepsilon + \delta), y \in T_{\gamma_{x,v}(t)}N,  y \perp \dot{\gamma}_{x,v}(t), \abs{y} < \delta\}.
\end{equation}
By \cite[Section 6]{ma2023anisotropic}, there is $\delta > 0$ such that for all $(x, v) \in SN$, the map $F_{x, v} : U(x, v, 2\delta) \to N$, $(t, y) \mapsto \exp_{\gamma_{x,v}(t)}(y)$ is a diffeomorphism onto its image. We fix $\delta$ as such and without loss of generality, $\delta < \varepsilon/2$.

Let $J = \ceil{\tau/\varepsilon}$ and for $0 \leq j < J$, consider the open sets
\begin{equation}
	V_j = F_{\gamma(j\varepsilon), \dot{\gamma}(j\varepsilon)}(U(\gamma(j\varepsilon), \dot{\gamma}(j\varepsilon), \delta)) \subset N.
\end{equation}
From the open sets $V_j$, we construct a new manifold
\begin{equation}
	N_\gamma = (\bigsqcup_{j=0}^{J-1} V_j) / \sim
\end{equation}
where the equivalence relation identifies $V_j$ and $V_{j+1}$ on their intersection in $N$. As $\delta$ is fixed, $N_\gamma$ is a smooth manifold and there is a natural projection $\pi_{\gamma} : N_\gamma \to N$. We equip $N_\gamma$ with the metric $\pi_\gamma^* g$ (which we also denote by $g$), making $\pi_\gamma$ a local isometry. Fermi coordinates yield global coordinates $(t,y)$ on $N_\gamma$ with $-\delta < t < T + \delta$ and $\abs{y} < \delta$. The curve $t \mapsto (t,0)$ is mapped to $\gamma$ by $\pi_\gamma$, and, as $\pi_\gamma$ is an isometry, it is a geodesic in $N_\gamma$. We denote this curve by $\gamma$ as well. By construction, even if $\gamma$ self-intersects as a curve in $N$, it does not self-intersect as a curve in $N_\gamma$.

There is also a smooth bundle structure on $N_\gamma$ given by the pullback bundle $(E', \pi', N_\gamma)$, also written $E' = \pi_\gamma^* E$. There is a natural map $\hat{\pi}_\gamma : E' \to E$ covering the map $\pi_\gamma$, that is, $\pi \circ \hat{\pi}_\gamma = \pi_\gamma \circ \pi'$. There is an induced fibre metric on $E'$ as $\dual{e_1}{e_2}_{E'} = \dual{\hat{\pi}_\gamma(e_1)}{\hat{\pi}_\gamma(e_2)}_{E}$. The pullback connection $\nabla' = \pi_\gamma^* \nabla$ satisfies
\begin{equation}
	\hat{\pi}_\gamma (\nabla_X' s) = \nabla_{d\pi_\gamma(X)} (\hat{\pi}_\gamma s)
\end{equation}
for $X \in T_q N_\gamma$, $s \in \Gamma(\pi')$, and it is compatible with the induced fibre metric on $E'$.

For the sake of convenience, we choose a global trivialisation $(s_\alpha)_{\alpha = 1}^n \subset \Gamma(\pi')$ over $N_\gamma$. Let $(e_\alpha)_{\alpha = 1}^n \subset (E')_{(0,0)}$ be such that $\dual{e_\alpha}{e_\beta}_{E'} = \delta_{\alpha\beta}$. For $(t, y) \in N_\gamma$, we define
\begin{equation}
	s_\alpha(t, y) = P^\nabla_{(t,y) \ot (0,0)} e_\alpha
\end{equation}
where the parallel transport is made along the curve $s \mapsto (st, sy)$ for $s \in [0,1]$. As $\nabla$ is compatible with the inner product, $\dual{s_\alpha}{s_\beta}_{E'} = \delta_{\alpha\beta}$. Moreover, we can find constants $C_k = C_k(N, E, \nabla, \tau)$ such that
\begin{equation}
	\abs{\nabla^k s_\alpha(t,y)} \leq C_k
\end{equation}
as $\nabla^k s_\alpha$ can be computed in terms of the connection and its curvature, and both are bounded over $N$. These bounds are also uniform over all $\gamma \in \mathcal{G}_T$. With respect to the trivialisation $(s_\alpha)$, we can write
\begin{equation}
	\nabla(\phi^\alpha s_\alpha) = (d\phi^\alpha + A^{\alpha}_\beta \phi^\beta)s_\alpha
\end{equation}
where $A^{\alpha}_\beta$ are complex-valued one-forms over $N_\gamma$. By compatibility of the connection with the metric, the matrix $A = (A^\alpha_\beta)$ is then a $\mathfrak{u}(n)$-valued one-form. We denote the connection Laplacian on $E'$ as $\Delta_A$. We have the expansion
\begin{equation}
	\Delta_A(\phi^\alpha s_\alpha) = (\Delta \phi^\alpha - 2(A^{\alpha}_\beta, d\phi^\beta) + (d^*A^{\alpha}_\beta) \phi^\beta - (A^\alpha_\delta, A^\delta_\beta \phi^\beta))s_\alpha.
\end{equation}
Here $(\cdot, \cdot)$ is the complex bilinear form induced by the metric $g$ on one-forms, so that in coordinates $(\omega, \eta) = g^{ij} \omega_i \eta_j$. Note that although we write the connection Laplacian on the pullback bundle as $\Delta_A$, it is a well-defined differential operator that is independent of our choice of coordinates.

Let $\Gamma_0(E')$ denote the set of smooth compactly supported sections of $E'$. Consider the push-forward map $(\pi_\gamma)_*: \Gamma_0(E') \to \Gamma(E)$ given by
\begin{equation}
	((\pi_\gamma)_* u)(x) = \sum_{q \in \pi_\gamma^{-1}(x)} \hat{\pi}_\gamma[u(q)].
\end{equation}
As $\pi_\gamma$ is a local isometry, if $u$ is compactly supported, we have $(\pi_\gamma)_*(\Delta_A u) = \Delta[(\pi_\gamma)_* u]$. We can similarly define $(\pi_\gamma)_* \varphi$ for $\varphi \in C_0^\infty(N_\gamma)$ as
\begin{equation}
	((\pi_\gamma)_* \varphi)(x) = \sum_{q \in \pi_\gamma^{-1}(x)} \varphi(q).
\end{equation}
Since $\pi_\gamma$ is a local isometry,
\begin{equation}\label{eq:pushforward_integral}
	\int_{N}(\pi_\gamma)_* \varphi \de V_g = \int_{N_\gamma} \varphi \de V_{\pi_\gamma^* g}.
\end{equation}

\begin{lem}
Let $\varphi \in C_0^\infty(N_\gamma)$. Then, for $1 \leq p \leq \infty$,
\begin{equation}
\norm{(\pi_{\gamma})_* \varphi}_{L^p(N)} \leq C_{p, \tau} \norm{\varphi}_{L^p (N_\gamma)}.
\end{equation}
\end{lem}

\begin{proof}
By our construction of $N_\gamma$, any point $x \in N$ has at most $J = \ceil{\tau/\varepsilon}$ preimages in $N_\gamma$. Therefore, $\norm{(\pi_{\gamma})_* \varphi}_{L^\infty(N)} \leq J \norm{\varphi}_{L^\infty (N_\gamma)}$. And for $1 \leq p < \infty$, equivalence of norms on $\R^J$ yields
\begin{equation}
	\abs{((\pi_\gamma)_* \varphi)(x)}^p = \abs{\sum_{q \in \pi_{\gamma}^{-1}(x)} \varphi(q)}^p \leq C_{p, \tau} \sum_{q \in \pi_\gamma^{-1}(x)} \abs{\varphi(q)}^p = C_{p, \tau} ((\pi_\gamma)_*\abs{\varphi}^p)(x).
\end{equation}
The result then follows from \eqref{eq:pushforward_integral}.
\end{proof}

\begin{lem}\label{lem:pushforward_Hk}
Let $u \in \Gamma_0(\pi_\gamma^* E) \cap W^{k, p}(N_\gamma)$ for some $k \in \Z_{\geq 0}$ and $1 \leq p \leq \infty$. Then, $(\pi_\gamma)_* u \in W^{k, p}(N)$ and
\begin{equation}
	\norm{(\pi_\gamma)_* u}_{W^{k, p}(N)} \leq C_{p, \tau} \norm{u}_{W^{k, p}(N_\gamma)}.
\end{equation}
\end{lem}

\begin{proof}
For $1 \leq p < \infty$, let $V$ be a neighbourhood of $x \in N$ and let $q_1, \dots, q_r$ be all the points such that $\pi_\gamma(q_i) = x$. Let $U_1, \dots, U_r$ be disjoint open sets in $N_\gamma$ such that $q_i \in U_i$ and $\pi_\gamma \vert_{U_i}$ is an isometry onto its image. Let $u_i$ be a section that vanishes outside $U_i$ and is identical to $u$ on a neighbourhood $K_i$ of $q_i$ whose closure is contained in $U_i$. Then, by definition of the pullback connection and by using that $\pi_\gamma$ is a local isometry,
\begin{align}
	\abs{\nabla^k [(\pi_\gamma)_* u] (x)}^p_{(T^*N)^{\otimes k} \otimes E} &\leq C_{p, r} \sum_{i=1}^r \abs{\nabla^k(\hat{\pi}_\gamma u_i)(x)}^p_{(T^*N)^{\otimes k} \otimes E} \\
	&= C_{p, r} \sum_{i = 1}^r \abs{\hat{\pi}_\gamma([(\nabla')^k \circ (d\pi_\gamma)_{q_i}^{-1}]u(q_i))}^p_{(T^*N)^{\otimes k} \otimes E} \\
	&= C_{p, r} \sum_{i=1}^r \abs{(\nabla')^k u(q_i)}^p_{(T^*N_\gamma)^{\otimes k} \otimes E'} \\
	&= C_{p, r} (\pi_\gamma)_*(\abs{(\nabla')^k u}^p_{(T^*N_\gamma)^{\otimes k} \otimes E'})(x)
\end{align}
The result then follows from \eqref{eq:pushforward_integral} and the fact that $r$ can be bounded in terms of $\tau$. For $p = \infty$, the estimate $\norm{(\pi_\gamma)_* u}_{W^{k, \infty}(N)} \leq J \norm{u}_{W^{k, \infty}(N_\gamma)}$ holds since any point $x \in N$ has at most $J$ preimages in $N_\gamma$.
\end{proof}

\subsection{Equations for Gaussian beams}\label{sec:eqsbeams}

As stated above, we are interested in constructing approximate solutions to the equation $(\Delta - \lambda^2)u = 0$ that concentrate along a nontangential geodesic $\gamma$. To do so, we will construct solutions on $N_\gamma$ first and push them forward to $N$ by the map $(\pi_\gamma)_*$. As $\pi_\gamma$ is a local isometry, if $u$ is compactly supported in $N_\gamma$, then
\begin{equation}
	(\Delta - \lambda^2)[(\pi_\gamma)_* u] = (\pi_\gamma)_*[(\Delta_A - \lambda^2)u].
\end{equation}
Hence, by Lemma \ref{lem:pushforward_Hk}, if $\norm{(\Delta_A - \lambda^2)u}_{H^k(N_\gamma)} = O(\lambda^{R})$,  then $\norm{(\Delta - \lambda^2)[(\pi_\gamma)_* u]}_{H^k(N)} = O_\tau(\lambda^{R})$.

We will look for solutions of $(\Delta_A - \lambda^2)u = 0$ on $N_\gamma$ of the form $u = e^{\I\lambda \phi} a$, where $\phi : N_\gamma \to \C$ is a complex phase function and $a \in \Gamma_0(\pi')$. If we expand $a$ as
\begin{equation}
	a = a_0 + \lambda^{-1} a_1 + \dots + \lambda^{-k}a_k,
\end{equation}
and we gather terms of the same order in $\lambda$ in the expansion of $(\Delta_A - \lambda^2)e^{\I\lambda \phi}a = 0$, we get the equations
\begin{align}
	(d\phi, d\phi) - 1 &= 0, \\
	(\Delta_g \phi)a_0 - 2(d\phi, d_A a_0) &= 0, \\
	(\Delta_g \phi)a_j - 2(d\phi, d_A a_j) - i \Delta_A a_{j-1} &= 0, \quad 1 \leq j \leq K.
	\end{align}
The first equation is called the eikonal equation, while the others are transport equations. Instead of solving these equations on the whole of $N_\gamma$, we will solve them to high order along $\gamma$. To do so, we will work within the framework of jet bundles.

\subsection{Jet bundles}

\subsubsection{Background}

Let $(E, \pi, M)$ be a vector bundle that we will denote by its projection map $\pi$. Let $(x^i, u^\alpha)$ and $(y^j, v^\beta)$ be coordinates for $\pi$ around some point $p \in M$, and let $\phi, \psi \in \Gamma(\pi)$ be such that $\phi(p) = \psi(p)$. Then, one can show that
\begin{equation}
	\pardiff{^{\abs{I}}}{x^I} (u^\alpha \circ \phi) \vert_p = \pardiff{^{\abs{I}}}{x^I} (u^\alpha \circ \psi) \vert_p
\end{equation}
for all multi-indices $I$ with $1 \leq \abs{I} \leq k$ if and only if
\begin{equation}
	\pardiff{^{\abs{J}}}{y^J} (v^\beta \circ \phi) \vert_p = \pardiff{^{\abs{J}}}{y^J} (v^\beta \circ \psi) \vert_p
\end{equation}
for all multi-indices $J$ with $1 \leq \abs{J} \leq k$. Therefore, we can define an equivalence relation between local sections of $\pi$ around $p$ so that $\phi \sim \psi$ if $\phi(p) = \psi(p)$, and, in a chart, all the derivatives of $\phi$ and $\psi$ agree up to order $k$. We denote the equivalence class of $\phi$ at $p$ as $\mathrm{j}^k_p \phi$, and let $\mathrm{J}^k\pi$ denote the set of all such equivalence classes over all $p \in M$. We call it the $k$-th jet manifold of $\pi$. It admits many different fibre bundle structures. Firstly, we can define the vector bundle $(\mathrm{J}^k\pi, \pi_k, M)$ with $\pi_k(\mathrm{j}^k_p \phi) = p$. For $0 \leq \ell < k$, we can also define the bundle $(\mathrm{J}^k\pi, \pi_{k, \ell}, \mathrm{J}^\ell \pi)$ with $\pi_{k,\ell}(\mathrm{j}^k_p \phi) = \mathrm{j}^\ell_p \phi$. For the sake of completeness, we let $(\mathrm{J}^0\pi, \pi_0, M) = (E, \pi, M)$ so that $\mathrm{j}^0_p\phi = \phi(p)$. There is a natural map $\mathrm{j}^k : \Gamma(\pi) \to \Gamma(\pi_k)$ given by
\begin{equation}
	\mathrm{j}^k \phi (p) = \mathrm{j}^k_p \phi.
\end{equation}
The section $\mathrm{j}^k \phi$ is called the $k$-th prolongation of $\phi$. For more background on jet bundles, see \cite{jetbundles}.

\subsubsection{Setting}

In what follows, we consider the bundle $(E, \pi, N_\gamma) = (E', \pi', N_\gamma)$ as above, and we will drop the apostrophe from our notation to make it less cumbersome. Let $\iota: \gamma \hookrightarrow N_\gamma$ be the natural inclusion map. We can construct jet bundles on $\gamma$ in two different ways. First, we can restrict the bundle $\pi$ over $N_\gamma$ to $\gamma$ to get $(\iota^* E, \iota^* \pi, \gamma)$. We denote $\tilde{\pi} = \iota^* \pi$ to avoid confusion with other bundles below. The jet manifold of $\tilde{\pi}$ is then $\mathrm{J}^k(\tilde{\pi})$ and we will denote the different bundles associated with it as $\tilde{\pi}_k$ and $\tilde{\pi}_{k, \ell}$. We denote the $k$-th prolongation map with respect to $\tilde{\pi}_k$ as $\tilde{j}^k : \Gamma(\tilde{\pi}) \to \Gamma(\tilde{\pi}_k)$. Second, if we consider the manifold
\begin{equation}
	\iota^*(\mathrm{J}^k \pi) = \{\mathrm{j}^k_p \phi \in \mathrm{J}^k \pi : p \in \gamma\},
\end{equation}
we can start with the bundle $(\mathrm{J}^k \pi, \pi_k, N_\gamma)$ and restrict it to $\gamma$ to form the bundle $(\iota^*(\mathrm{J}^k \pi), \iota^* \pi_k, \gamma)$. This yields similarly the bundles $(\iota^*(\mathrm{J}^k \pi), \iota^* \pi_{k, \ell}, \iota^*(\mathrm{J}^\ell \pi))$ from the bundles $(\mathrm{J}^k\pi, \pi_{k, \ell}, \mathrm{J}^\ell \pi)$. The $k$-th prolongation map $\mathrm{j}^k$ on $\pi$ induces the map $\iota^* \mathrm{j}^k : \Gamma(\pi) \to \Gamma(\iota^*(\pi_k))$.

Note that the bundles $\tilde{\pi}_k$ and $\iota^* \pi_k$ are quite different. If $\phi \in \Gamma(\pi)$, then both bundles capture derivatives of $\phi$ at points on $\gamma$. However, $\tilde{\pi}_k$ only captures the derivatives of $\phi$ along $\gamma$, while $\iota^* \pi_k$ captures the derivatives of $\phi$ in all directions. For instance, the fibre of $\tilde{\pi}_1$ has dimension $n$, while the fibre of $\iota^* \pi_1$ has dimension $mn$.

Let $\rho^k : \iota^*(\mathrm{J}^k \pi) \to \mathrm{J}^k \tilde{\pi}$ be the unique map such that for any $\phi \in \Gamma(\pi)$ with $\iota^* \mathrm{j}^k \phi(p) = \mathrm{j}_p^k \phi$, we have
\begin{equation}\label{eq:rho}
	\rho^k(\iota^* \mathrm{j}^k \phi(p)) = \tilde{\mathrm{j}}^k (\iota^* \phi(p)).
\end{equation}
This map is well defined since if $p = \gamma(t)$ and $\del_t = \dot{\gamma}(t)$, then the derivatives $\del_t, \del_t^2, \dots, \del_t^k$ can all be uniquely determined from the knowledge of all the derivatives of order at most $k$ at $p$. The map $\rho^k$ essentially forgets all the derivatives that are not in the direction of $\gamma$.

When using Fermi coordinates $(t,y) = (t, y_1, \dots, y_{m-1})$ and the trivialisation $(s_\alpha)$, we can write a section $\psi$ of $\iota^*\pi_k$ as
\begin{equation}
	\psi = (\psi_J^\alpha)_{0 \leq \abs{J} \leq k}.
\end{equation}
In those coordinates, a multi-index $J$ will always denote a multi-index over all the variables, that is, $J \in \N^m$. On the other hand, a multi-index $I$ will always denote a multi-index running over only the $y$ variables. Associated to such an $I\in \N^{m-1}$ is the differential operator $\del_y^I$. For $\ell \in \N$, we denote by $(\ell, I)$ the multi-index in $\N^m$ whose first component is $\ell$ and whose other components are given by $I$.

A differential operator $Q$ on $\pi$ of order $0 \leq \ell \leq k$ naturally extends to a map $Q: \mathrm{J}^k \pi \to \mathrm{J}^{k-\ell} \pi$. Indeed, if $\phi, \psi \in \Gamma(\pi)$ are such that $\mathrm{j}^k_p \phi = \mathrm{j}^k_p \psi$, then $\mathrm{j}^{k-\ell}_p(Q\phi) = \mathrm{j}^{k-\ell}_p(Q\psi)$. This shows that $Q$ is well-defined on $\mathrm{J}^k \pi$. In particular, $C^\infty(N_\gamma)$ induces endomorphisms of $\mathrm{J}^k \pi$ as $f \in C^\infty(N_\gamma)$ can be seen as a differential operator of order $0$. Moreover, it restricts to a map $Q : \iota^*(\mathrm{J}^k \pi) \to \iota^*(\mathrm{J}^{k-\ell} \pi)$ and, if $\phi \in \Gamma(\pi)$,
\begin{equation}\label{eq:Q_commutes_iota}
	Q(\iota^*(\mathrm{j}^k \phi)) = \iota^*(\mathrm{j}^{k-\ell}(Q\phi))
\end{equation}

If $\phi \in \Gamma(\pi)$, then $\iota^*(\mathrm{j}^k \phi) \in \Gamma(\iota^* \pi_k)$. However, not all sections of $\iota^* \pi_k$ arise this way. We say that a section $\psi \in \Gamma(\iota^* \pi_k)$ is prolongable along $\gamma$ if there exists $\phi \in \Gamma(\pi)$ with $\psi = \iota^*(\mathrm{j}^k \phi)$. We call $\phi$ the prolongation of $\psi$, and we write $\psi \in \Gamma_\gamma(\iota^* \pi_k)$ if it is prolongable. We start by characterising the prolongable sections.

\begin{lem}\label{lem:prolongable}
Let $\psi \in \Gamma(\iota^* \pi_k)$, then $\psi$ is prolongable along $\gamma$ if and only if
\begin{equation}\label{eq:prolongable}
	\rho^{k - \ell}(Q\psi) = \tilde{j}^{k - \ell}(\tilde{\pi}_{k-\ell, 0} \circ \rho^{k - \ell}(Q\psi))
\end{equation}
for all differential operators $Q$ of order $0 \leq \ell \leq k$ on $\pi$.
\end{lem}

\begin{proof}
Suppose first that $\psi \in \Gamma_{\gamma}(\iota^* \pi_k)$, that is, there is $\phi \in \Gamma(\pi)$ with $\psi =\iota^* (\mathrm{j}^k \phi)$. Let $Q$ be a differential operator of order $0 \leq \ell \leq k$. Then, if $\tilde{\phi} = \iota^* (Q\phi) \in \Gamma(\tilde{\pi})$, we have $\rho^{k-\ell}(Q\psi) \in \Gamma(\tilde{\pi}_{k-\ell})$ and, by \eqref{eq:rho} and \eqref{eq:Q_commutes_iota}, $\rho^{k-\ell}(Q\psi) = \tilde{j}^{k-\ell}(\tilde{\phi})$. Therefore, $\tilde{\phi}$ is the $(k-\ell)$-th prolongation of $\rho^{k-\ell}(Q\psi)$ along $\gamma$ and by \cite[Lemma 6.2.16]{jetbundles}, $\rho^{k-\ell}(Q\psi)$ must satisfy \eqref{eq:prolongable}.

It remains to show that if \eqref{eq:prolongable} holds for all differential operators $Q$ of order $0 \leq \ell \leq k$, then $\psi$ is prolongable along $\gamma$. For this, we work in Fermi coordinates $(t, y)$. Let $(s_\alpha)$ be a trivialisation of $E$ on $N_\gamma$. Take $\psi = (\psi_J^\alpha) \in \Gamma(\iota^* \pi_k)$ such that \eqref{eq:prolongable} holds for all differential operators $Q$ of order $0 \leq \ell \leq k$. In particular, it must hold for $Q = \del_y^I$ with $0 \leq \abs{I} \leq k$. In coordinates, \eqref{eq:prolongable} is then equivalent to
\begin{equation}
	\psi_{(\ell, I)}^\alpha = \del_t^\ell \psi_{(0, I)}^\alpha
\end{equation}
for all $0 \leq \ell + \abs{I} \leq k$. Therefore, if
\begin{equation}
	\phi(t,y) = \sum_{\abs{I} \leq k} \frac{\psi_{(0,I)}^\alpha(t)}{I!} y^I s_\alpha
\end{equation}
then it is easy to check that $\psi = \iota^*(\mathrm{j}^k \phi)$.
\end{proof}

\begin{lem}\label{lem:prolongableFermi}
In Fermi coordinates, let $\psi = (\psi_J^\alpha)_{0 \leq \abs{J} \leq k} \in \Gamma(\iota^* \pi_k)$. Then, $\psi$ is prolongable if and only if
\begin{equation}
	\psi^\alpha_{(\ell, I)}(t) = \del_t^\ell \psi^\alpha_{(0,I)}
\end{equation}
for all $0 \leq \ell + \abs{I} \leq k$.
\end{lem}

\subsection{Gaussian beams as sections of a jet bundle}

\subsubsection{Jet of a Gaussian beam}\label{sec:jet_beam}

Let $\pi_k$ denote the jet bundle associated with the bundle $(E, \pi, N_\gamma)$ as above, and let $\pi_k^\C$ denote the jet bundle associated with the trivial bundle $N_\gamma \times \C$. We say that a $(K+2)$-tuple $(\Phi; \mathsf{a}_0, \dots, \mathsf{a}_K)$ with $\Phi \in \Gamma_\gamma(\iota^*\pi_{3K + 2}^\C)$ and $\mathsf{a}_j \in \Gamma_\gamma(\iota^*\pi_{3K - 2j})$ for $0 \leq j \leq K$ is the jet of a Gaussian beam of order $K$ if, for any choice of prolongations $\phi \in \Gamma(\pi^\C)$ and $a_j \in \Gamma(\pi)$, we have
\begin{align}
	\iota^* \mathrm{j}^{3K+2}(\abs{d\phi}_g^2 - 1) &= 0 \\
	\iota^* \mathrm{j}^{3K}((\Delta \phi)a_0 - 2(d\phi, d_A a_0)) &= 0 \\
	\iota^* \mathrm{j}^{3K - 2j}((\Delta \phi)a_j - 2(d\phi, d_A a_j) - i \Delta_A a_{j-1}) &= 0, \quad 1 \leq j \leq K.
\end{align}
Note that the difference in the orders of the jet bundles is due to the term $\Delta_A a_{j-1}$ in the equation for $a_j$. Hence, to solve $a_j$ up to order $k$, we need to know $a_{j-1}$ to order $k+2$. Similarly, to solve for $a_0$ up to order $3K$, we need to know $\phi$ up to order $3K + 2$. Note that if these equations hold for one set of prolongations, then they hold for any other by the definition of the equivalence classes for the jet bundles. In other words, the resulting Gaussian beam depends on the choice of coordinates and trivialisation of the bundle, but the jet of the Gaussian beam does not.

We would like to control the value of the amplitudes and their derivatives at $\gamma(0)$ along a hypersurface $Y$, that is, we would like to fix the values of $\rho_Y^{3K-2j}(\mathsf{a}_j(\gamma(0)))$ and understand how they affect the Gaussian beams. We denote the set of initial conditions at $\gamma(0)$ along $Y$ as
\begin{equation}
\mathcal{A}^K_{Y, \gamma(0)} = (\mathrm{J}^{3K} \pi_Y)_{\gamma(0)} \times (\mathrm{J}^{3K - 2} \pi_Y)_{\gamma(0)} \times \dots \times (\mathrm{J}^K \pi_Y)_{\gamma(0)}.
\end{equation}
Let $(e_\alpha)$ be an orthonormal basis of $E_{\gamma(0)}$ and let $(s_\alpha)$ be the global trivialisation obtained from $(e_\alpha)$ as described in Section \ref{sec:localisation}. We write $\mathsf{e} = (\mathsf{e}_0, \dots, \mathsf{e}_K) \in \mathcal{A}^K_{Y, \gamma(0)}(C)$ if, with respect to these coordinates and this trivialisation, $\abs{\mathsf{e}_{j, I}^\alpha} \leq C$ for all $\abs{I} \leq 3K - 2j$ and all $1 \leq \alpha \leq n$. Note that this is independent of the initial choice of orthonormal basis as any other choice is related to it by a constant unitary transformation. When working on $N_\gamma$, we will choose $Y = Y_0 := \{t = 0\}$, and when working on $M$, we will choose $Y = \del M$. We will show in Lemma \ref{lem:initial_conditions} that the choice of hypersurface $Y$ is always equivalent to $Y_0$ up to a constant as long as $Y$ is transversal to $\gamma$ at $\gamma(0)$.

\subsubsection{Phase function}

In what follows, we only need to consider the trivial vector bundle $(N_\gamma \times \C, \pi^\C, N_\gamma)$. We will work in Fermi coordinates. The eikonal equation $\abs{d\phi}_g^2 - 1 = 0$ on $N_\gamma$ can be expanded to
\begin{equation}
	g^{ij} (\del_i \phi)(\del_j \phi) - 1 = 0
\end{equation}
and we can consider this equation from the point of view of jet bundles. We can see $\phi$ in two ways. On one hand, $\mathrm{j}^k\phi \in \Gamma(\pi_k^\C)$, while on the other, $\phi \in C^\infty(N_\gamma)$ and hence acts on $\Gamma(\pi_k)$. We can also see the metric components $g^{ij}$ as acting on $\Gamma(\pi_k^\C)$. Our goal is to find $\Phi \in \Gamma_\gamma(\iota^* \pi_k^\C)$ such that for any prolongation $\phi \in \Gamma(\pi^\C)$, $\iota^* \mathrm{j}^k(\abs{d\phi}_g^2 - 1) = 0$. By Lemma \ref{lem:prolongableFermi}, it suffices to solve for $\Phi_{(0,I)}$ for $0 \leq \abs{I} \leq k$, and then we set
\begin{equation}
	\Phi_{(\ell, I)} = \del_t^\ell \Phi_{(0, I)}
\end{equation}
for all $0 \leq \ell + \abs{I} \leq k$. The equation $\iota^* \mathrm{j}^k(\abs{d\phi}_g^2 - 1) = 0$ is equivalent to $\del_y^I(g^{ij}(\del_i \phi)(\del_j \phi) - 1) = 0$ on $\gamma$ for all $\abs{I} \leq k$.

\begin{prop}\label{prop:phase}
For any $k \in \N$, there exists $\Phi \in \Gamma_{\gamma}(\iota^* \pi_k^\C)$ such that for any $\phi: N_\gamma \to \C$ with $\iota^*\mathrm{j}^k \phi = \Phi$, we have $\iota^*\mathrm{j}^k(\abs{d\phi}^2 - 1) = 0$. Moreover, there is a positive constant $C = C(N, g, \tau, k)$ such that, in Fermi coordinates, we can choose $\phi$ with
\begin{equation}\label{eq:phase}
	\phi(t,y) = t + H(t) y \cdot y + O(\abs{y}^3)
\end{equation}
where $H(t)$ is a complex symmetric $(m-1) \times (m-1)$ matrix with $\im(H(t)) \geq C^{-1} I$, and
\begin{equation}\label{eq:bounded_phase}
	\norm{\phi}_{C^k(N_\gamma)} \leq C.
\end{equation}
\end{prop}

\begin{proof}
Let the indices $i$ and $j$ run over the variables $t, y_1, \dots, y_{m-1}$ and let the multi-indices $I, \alpha, \beta, \epsilon$ correspond only to the variables $y_1, \dots, y_{m-1}$. We can compute
\begin{equation}\label{eq:phase_taylor}
	\del^I_y(g^{ij} (\del_i \phi)(\del_j \phi) - 1) = \sum_{\alpha + \beta + \epsilon = I} \frac{I!}{\alpha!\beta!\epsilon!} (\del^\alpha_y g^{ij})(\del^\beta_y \del_i \phi)(\del^\epsilon_y \del_j \phi)
\end{equation}
We choose $\phi$ such that $\phi(t) = t$ on $\gamma$ and $\del_y^{I} \phi = 0$ for $\abs{I} = 1$. Then, $\del^I_y(g^{ij} (\del_i \phi)(\del_j \phi) - 1) = 0$ for $\abs{I} \leq 1$, which is equivalent to $\iota^*\mathrm{j}^1(\abs{d\phi}_g^2 - 1) = 0$. Moreover, as we need $\phi$ to come from a section of $\pi$, this fixes $\Phi_{(\ell, 0)} = \del_t^\ell \Phi = 0$ for $\ell \leq k$, and $\Phi_{(\ell, I)} = \del_t^\ell \del^I_y \phi$ for $\ell \leq k-1$, $\abs{I} = 1$.

For $\abs{I} = 2$ with $\beta + \epsilon = I$, the only non-zero terms in \eqref{eq:phase_taylor} are given by
\begin{equation}
	\del^I_y(g^{ij} (\del_i \phi)(\del_j \phi) - 1) = 2\del_t (\del^I_y \phi) + 2 \delta^{ij} (\del^{\beta}\del_i \phi)(\del^\epsilon \del_j \phi) + 2\del^I_y g^{tt}
\end{equation}
Writing $\phi(t, y) = t + \frac{1}{2} H(t) y \cdot y$, this equation vanishes for all $\abs{I} = 2$ if and only if the Riccati equation
\begin{equation}
	\dot{H}(t) + H(t)^2 = F(t)
\end{equation}
holds. Here, $F(t)$ is the unique matrix such that $g^{tt} = 1 - \frac{1}{2} F(t) y \cdot y + O(\abs{y}^3)$. If we choose a complex symmetric matrix $H_0$ with a positive-definite imaginary part, there is a unique smooth solution to the Riccati equation such that $H(0) = H_0$ and $\im(H(t))$ is positive-definite \cite[Lemma 2.56]{inverseboundaryspectral}. We choose $H_0 = \I (\kappa^* g)_{\gamma(0)}$ where $\kappa : T_{\gamma(0)} M \to (\dot{\gamma}(0))^\perp$ is the orthogonal projection on $(\dot{\gamma}(0))^\perp$. This yields $\iota^* \mathrm{j}^2(\abs{d\phi}_g^2 - 1) = 0$ and fixes $\Phi_{(\ell, I)} = \del^\ell_t \del^I_y \phi$ for $\ell \leq k - 2$, $\abs{I} = 2$. As in \cite[Lemma 6.3]{ma2023anisotropic}, we can find a constant $c = c(N, g, \tau) > 0$ such that
\begin{equation}
	\im(\nabla^2 \phi \vert_{(\dot{\gamma}(t))^\perp}) = \im (H(t)) \geq c I.
\end{equation}

Now suppose that we have found $\phi$ such that $\iota^* \mathrm{j}^r(\abs{d\phi}_g^2 - 1) = 0$ for some $r \geq 2$, with $\Phi_{(\ell, I)} = \del_t^\ell \del^I_y\phi$ for all $\ell \leq k - \abs{I}$ with $\abs{I} \leq r$. Then, for $\abs{I} = r + 1$, we have
\begin{equation}
	\del^I_y(g^{ij} (\del_i \phi)(\del_j \phi) - 1) = 2 \del_t(\del_y^I \phi) + 2 \sum_{\beta + \epsilon = I, \abs{\beta} = r} \frac{I!}{\beta!} \delta^{ij} (\del^\beta_y \del_i \phi)(\del^\epsilon_y \del_j \phi) + F_I(t)
\end{equation}
where $F_I(t)$ is uniquely determined by the derivatives of $\phi$ of order at most $r$ and the derivatives of $g^{ij}$ of order at most $r+1$. By combining all the equations $\del^I_y(g^{ij} (\del_i \phi)(\del_j \phi) - 1) = 0$ for $\abs{I} = r+1$, we get a system of nonhomogeneous ODEs. Given the initial conditions $\del_y^I \phi(0) = 0$ for $\abs{I} = r + 1$, it admits a unique solution $\del^I_y \phi(t)$. We set $\Phi_{(0, I)} = \del^I_y \phi(t)$ and this fixes $\Phi_{(\ell, I)} = \del_t^\ell \del_y^I \phi$ for $\ell \leq k - r - 1$, $\abs{I} = r + 1$. Therefore, working inductively, we can find $\phi \in \Gamma(\iota^* \pi_k)$ as required.

By energy estimates such as \cite[Section 1.5]{taylor}, we can find $C = C(M, g, k, \tau) > 0$ such that
\begin{equation}
	\abs{\Phi_J(t)} \leq C
\end{equation}
for all $0 \leq t \leq \tau$, $0 \leq \abs{J} \leq k$. Therefore, if we let
\begin{equation}\label{eq:prolongation_phase}
	\phi = \sum_{\abs{I} \leq k} \frac{\Phi_{(0, I)}(t)}{I!} y^I,
\end{equation}
then $\iota^* \mathrm{j}^k(\abs{d\phi}^2 - 1) = 0$ and $\norm{\phi}_{C^k((0,T) \times B(\delta))} \leq Ce^{(m-1)\delta}$.
\end{proof}

\subsubsection{Amplitudes}

We proceed similarly for the amplitudes $\mathsf{a}_j$, but now consider the bundle $(E, \pi, N_\gamma)$ instead of the trivial bundle $N_\gamma \times \C$.

\begin{prop}\label{prop:amplitudes}
Let $K \in \Z_{\geq 0}$, and let $\mathsf{e} \in \mathcal{A}^{K}_{Y_0, \gamma(0)}(C_0)$ for $0 \leq j \leq K$. Then, there are sections $\mathsf{a}_j \in \Gamma_\gamma(\pi_{3K - 2j})$ such that for any $a_j \in \Gamma(\pi)$ with $\iota^* \mathrm{j}^{3K - 2j} a_j = \mathsf{a}_j$, we have
\begin{align}
	\iota^* \mathrm{j}^{3K}((\Delta \phi)a_0 - 2(d\phi, d_A a_0)) &= 0, \\
	\iota^* \mathrm{j}^{3K - 2j}((\Delta \phi)a_j - 2(d\phi, d_A a_j) - i \Delta_A a_{j-1}) &= 0, \quad 1 \leq j \leq K, \\
	\rho^{3K - 2j}_{Y_0}[(\iota^* \mathrm{j}^{3K - 2j}a_j)(\gamma(0))] &= \mathsf{e}_j,
\end{align}
where $\phi$ is the phase function from Proposition \ref{prop:phase}. On $\gamma$, the amplitude $a_0$ is given by
\begin{equation}\label{eq:a0_parallel_transport}
	a_0(\gamma(t)) = \exp\left(\frac{1}{2} \int_0^t \tr H(s) \de s\right)P^A_{\gamma[0, t]} \mathsf{e}_{0, 0}.
\end{equation}
Moreover, there is a positive constant $C = C(N, g, \tau, E, \nabla, K, C_0)$ such that we can choose the amplitudes $a_j$ with
\begin{equation}\label{eq:bounded_amplitudes}
	\norm{a_j}_{C^{3K-2j}(N_\gamma)} \leq C.
\end{equation}
\end{prop}

\begin{proof}
We will detail the construction for $\mathsf{a}_0$, which we denote by $\mathsf{a}$. The construction for the other amplitudes will be similar. We wish to solve the equation
\begin{equation}
	(\Delta \phi)a - 2(d \phi, d_A a) = 0
\end{equation}
to order $3K$ along $\gamma$, that is, find $\mathsf{a} \in \Gamma_\gamma(\iota^* \pi_k)$ such that there is a prolongation $a \in \Gamma(\pi)$ with $\iota^* \mathrm{j}^{3K}((\Delta \phi)a - 2(d \phi, d_A a)) = 0$. Let $(t, y)$ be Fermi coordinates and let $(s_\alpha)_{\alpha=1}^n$ be the smooth orthonormal frame for the vector bundle on $N_\gamma$ from Section \ref{sec:localisation}. We let $i$, $j$ denote indices running from $1$ to $m$, and $\alpha$, $\beta$ denote indices running from $1$ to $n$. In coordinates, $\mathsf{a} = (\mathsf{a}^\alpha_J)_{0 \leq \abs{J} \leq 3K}$, and
\begin{equation}
	(\Delta \phi)a - 2(d \phi, d_A a) = [(\Delta \phi)a^\alpha - 2g^{ij}(\del_i \phi)(\del_j a^\alpha + A^\alpha_{\beta j} a^\beta)]s_\alpha.
\end{equation}
Here, the functions $A^\alpha_{\beta j}$ come from the connection one-form $A$ associated to $\nabla$. Similarly to our approach for the phase, we only need to compute the coefficients $\mathsf{a}^\alpha_{(0, I)}$ from the equation $\del_y^I((\Delta \phi)a - 2(d \phi, d_A a)) = 0$, and we then set
\begin{equation}
	\mathsf{a}^\alpha_{(\ell, I)} = \del_t^\ell \mathsf{a}^\alpha_{(0,I)}
\end{equation}
for $0 \leq \ell + \abs{I} \leq 3K$.

As $\Delta \phi = -\tr H(t)$ on $\gamma$, the equation for $\abs{I} = 0$ yields
\begin{equation}
	\del_t a^\alpha + A^\alpha_{\beta t} a^\beta + \tr H(t) a^\alpha = 0.
\end{equation}
If we let $f^\alpha(t) = \exp(-\frac{1}{2} \int_{0}^{t} \tr H(s) \de s)a^\alpha(t)$, we see that $f^\alpha$ solves
\begin{equation}
	\del_t f^\alpha + A^\alpha_{\beta t} f^\beta = 0,
\end{equation}
which are precisely the equations of parallel transport along $\gamma$. Therefore, if $a(0) = a_0$, then
\begin{equation}
	a(t) = \exp\left(\frac{1}{2} \int_{0}^t \tr H(s) \de s\right) P^A_{\gamma[0, t]} a_0.
\end{equation}
As we want $a$ to come from a section of $\pi$, this fixes $\mathsf{a}^\alpha_{(\ell, 0)}$ for $\ell \leq 3K$.

Now suppose that we have found $a$ such that $\iota^* \mathrm{j}^r((\Delta \phi)a - 2(d \phi, d_A a)) = 0$ for some $r \geq 0$, with $\mathsf{a}^\alpha_{(\ell, I)} = \del_t^\ell \del_y^I a^\alpha$ for all $\ell \leq 3K - \abs{I}$ with $\abs{I} \leq r$. Then, for $\abs{I} = r + 1$, we have
\begin{align}
	\del_y^I((\Delta \phi)a^\alpha - 2g^{ij}(\del_i \phi)(\del_j a^\alpha + A^\alpha_{j\beta} a^\beta)) &= (\Delta \phi) (\del_y^I a^\alpha) - 2 \del_t(\del_y^I a^\alpha) - 2A^\alpha_{\beta t}(\del_y^I a^\beta)\\ &\quad - 2\sum_{I_1 + I_2 = I, \abs{K} = r} (\del_y^{I_1} \del_i \phi)(\del_y^{I_2} \del_j a^\alpha) + F_I(t)
\end{align}
where $F_I$ is uniquely determined by the derivatives of $a^\alpha$ of order at most $r$, the derivatives of the metric of order at most $r+1$, the derivatives of $\phi$ of order at most $r+2$, and the derivatives of the connection $A$ of order at most $r+1$. By combining all the equations $\del_y^I((\Delta \phi)a - 2(d \phi, d_A a)) = 0$ for $\abs{I} = r + 1$, we get a system of nonhomogeneous ODEs. Given initial conditions for $\del_y^I a^\alpha(0)$, it admits a unique solution $\del_y^I a^\alpha(t)$. We set $\mathsf{a}^\alpha_{(0,I)}(t) = \del_y^I a^\alpha(t)$ and this fixes $\mathsf{a}_{(\ell, I)} = \del_t^\ell \del_y^I a^\alpha$ for $\ell \leq 3K - r - 1$, $\abs{I} = r+1$. Working inductively, we can find $\mathsf{a} \in \Gamma(\iota^* \pi_{3K})$ as required.

If $Y_0$ is the hypersurface given by $\{t = 0\}$, then the choice of initial conditions $\del^I_y a^\alpha(0)$ for $0 \leq \abs{I} \leq 3K$ is equivalent to a choice of jet $\mathsf{e}_0 \in (\mathrm{J}^{3K} \pi_{Y_0})_{\gamma(0)}$. Hence, if $\mathsf{e} \in \mathcal{A}^{K}(C_0)$, then $\abs{\mathsf{a}^\alpha_{(0,I)}(0)} \leq C_0$ for all $\abs{I} \leq 3K$ and, by energy estimates \cite[Section 1.5]{taylor}, we can find $C' = C'(N, g, E, \nabla, \tau, C_0, K) > 0$ such that
\begin{equation}
	\abs{\mathsf{a}_J^\alpha(t)} \leq C'
\end{equation}
for all $-\delta < t < \tau + \delta$, $0 \leq \abs{J} \leq 3K$. If we let
\begin{equation}
	a = \sum_{\abs{I} \leq k} \frac{\mathsf{a}^\alpha_{(0,I)}(t)}{I!} y^I s_\alpha(t, y),
\end{equation}
then $\iota^* \mathrm{j}^{3K}((\Delta \phi)a - 2(d\phi, d_A a_0)) = 0$ and there is a positive constant $C = C(N, g, E, \nabla, \tau, C_0, K)$ such that
\begin{equation}
	\norm{a}_{C^k(N_\gamma)} \leq C.
\end{equation}
\end{proof}

\subsubsection{Initial conditions and boundary values}

When pushing forward our Gaussian beam to $M$, we want to control the derivatives of our solution in terms of derivatives on the boundary, which does not necessarily correspond to $Y_0$. This is equivalent to controlling the derivatives of our solution on $N_\gamma$ at $\gamma(0)$ along the hypersurface $\pi_\gamma^{-1}(\del M)$. We show the following.

\begin{lem}\label{lem:initial_conditions}
Let $K \in \Z_{\geq 0}$. Let $Y$ be a smooth embedded hypersurface in $N_\gamma$ that intersects $\gamma$ at $\gamma(0)$ at an angle $\theta_Y \in (0, \frac{\pi}{2}]$, and let $Y_0$ be the smooth hypersurface $\{t = 0\}$. For any $C_0 > 0$, there exists $C_0' = C_0'(N, g, \theta_Y, Rm_Y, K, C_0)$ such that for any choice of initial conditions $\mathsf{f} \in \mathcal{A}^{K}_{Y, \gamma(0)}(C_0)$ along $Y$, there are initial conditions $\mathsf{e} \in \mathcal{A}^{K}_{Y_0, \gamma(0)}(C_0')$ along $Y_0$ and a jet $(\Phi; \mathsf{a}_0, \dots, \mathsf{a}_K)$ of a Gaussian beam of order $K$ such that
\begin{align}
	\rho^{3K - 2j}_Y(\mathsf{a}_j(\gamma(0)) &= \mathsf{f}_j, \\
	\rho^{3K - 2j}_{Y_0}(\mathsf{a}_j(\gamma(0)) &= \mathsf{e}_j,
\end{align}
for $0 \leq j \leq K$. Here, $Rm_Y$ is the Riemannian curvature tensor of $Y$.
\end{lem}

\begin{proof}
Let us start with the initial conditions for $a_0$, that is, we want to show that we can pick $a_0$ such that $\rho_{Y}^{3K}[(\iota^* \mathrm{j}^{3K} a_0)(\gamma(0))] = \mathsf{f}_0$, and $\iota^* \mathrm{j}^{3K}((\Delta \phi)a_0 - 2(d\phi, d_A a_0)) = 0$. Let $(r, x) = (r, x_1, \dots, x_{m-1})$ denote Fermi coordinates for the hypersurface $Y$ around $\gamma(0)$, defined in the same way as boundary normal coordinates. Without loss, we can choose $r$ so that $g(\del_r, \del_t) > 0$ at $\gamma(0)$. By our Gaussian beam construction, we can choose $a_0$ such that $a_0^\alpha(\gamma(0)) = \mathsf{f}_{0, 0}^\alpha$. Set $\mathsf{e}_{0, 0}^\alpha = \mathsf{f}_{0, 0}^\alpha$. As the beam is constructed from solving ODEs along $\gamma$, this fixes $\del_t a_0^\alpha(\gamma(0))$. At $\gamma(0)$, we can write 
\begin{align}
	\del_t &= (\cos \theta)\del_r + w \label{eq:del_t}\\
	\del_{y^i} &= \alpha^i \del_r + w^i \label{eq:del_yi}
\end{align}
for some $w, w^1, \dots, w^{m-1} \in \mathrm{span}(\del_{x_1}, \dots, \del_{x_{m-1}})$. Here, $\theta \in [0, \frac{\pi}{2})$ is the angle between $\dot{\gamma}(0) = \del_t$ and $\del_r$, and $\alpha^i = g(\del_r, \del_{y^i})$. Note that $\theta_Y = \frac{\pi}{2} - \theta$. Equations \eqref{eq:del_t} and \eqref{eq:del_yi} define a linear system from which we can see there exist initial conditions $\mathsf{e}_{0, I}$ such that $\del_{x}^I a_0^\alpha(\gamma(0)) = \mathsf{f}_{0, I}^\alpha$ if and only if $\del_{y}^I a_0^\alpha(\gamma(0)) = \mathsf{e}_{0, I}^\alpha$ for all $\abs{I} = 1$. We want to show that $\mathsf{e}_{0, I}^\alpha$ is bounded whenever $\mathsf{f}_{0, I}^\alpha$ is bounded. By expanding $\del_r = (\cos \theta)\del_t + \alpha^i \del_{y^i}$ and taking norms, we see that $\abs{\alpha^i} \leq \sin \theta$. By isolating $\del_r$ in \eqref{eq:del_t} and substituting into \eqref{eq:del_yi}, we get
\begin{equation}
	\del_{y^i} = \frac{\alpha^i}{\cos \theta} (\del_t - w) + w^i.
\end{equation}
As $\abs{w} \leq 1$ and $\abs{w^i} \leq 1$, it follows that for any smooth function $f$ defined around $\gamma(0)$,
\begin{equation}\label{eq:hypersurface_change}
	\max_{i = 1, \dots, m-1} \abs{\del_{y^i} f} \leq (\tan \theta)\abs{\del_t f} + (1 + \tan \theta) \max_{i = 1, \dots, m-1} \abs{\del_{x^i} f}.
\end{equation}
Therefore, with $a_0^\alpha(\gamma(0))$ fixed, as $\abs{\del_t a_0^\alpha(\gamma(0))}$ is bounded by Lemma \ref{prop:amplitudes}, we see that there is a constant $C_0 = C_0(C, \theta)$ such that $\abs{\mathsf{e}_{0, I}^\alpha} \leq C_0$ whenever $\abs{\mathsf{f}_{0, I}^\alpha} \leq C$ for $\abs{I} \leq 1$.

Similarly, for any $I$ with $\abs{I} = k$, we can express $\del_y^I f$ as a linear combination of $\del_x^{I'} f$ with $\abs{I'} \leq k$ and $\del_t^{\ell} \del_y^{I''} f$ with $\ell \leq k$ and $\abs{I''} < k$. We can iterate \eqref{eq:hypersurface_change} to get bounds on the coefficients. By expressing commutators in terms of the Levi-Civita connections of $Y$ and $N_\gamma$, we can find a constant $C'$ that depends on $k$, $\theta$ and the derivatives of the components of the Riemannian curvature tensors of $Y$ and $N_\gamma$ at $\gamma(0)$ such that
\begin{equation}\label{eq:hypersurface_higher_derivatives}
	\max_{\abs{I} = k} \abs{\del_{y}^I f} \leq C'\left(\max_{\abs{I} < k, \ell \leq k} \abs{\del_t^\ell \del_y^I f} + \max_{\abs{I} \leq k} \abs{\del_{x}^I f}\right).
\end{equation}
Using \eqref{eq:hypersurface_higher_derivatives} recursively we can find a constant $C_0'$ such that any set of initial conditions $\mathsf{f} \in \mathcal{A}^{K}_{Y, \gamma(0)}(C_0)$ corresponds to a unique set of initial conditions $\mathsf{e} \in \mathcal{A}^{K}_{Y_0, \gamma(0)}(C_0')$.
\end{proof}

\subsection{Gaussian beam estimates}\label{sec:GBestimates}

Let $\phi$ be a phase as in Proposition \ref{prop:amplitudes}. As the constant $C$ in the Proposition is uniform over $\mathcal{G}_T$, we can find constants $c, C > 0$ such that
\begin{equation}
	\im(\phi(t, y)) \geq 2c \abs{y}^2 - C \abs{y}^3
\end{equation}
for all phases functions related to geodesics $\gamma \in \mathcal{G}_T$. Let $0 < \delta' < \delta$ be such that $\im(\phi(t,y)) \geq c \abs{y}^2$ for all $\abs{y} < \delta'$. We fix a cutoff function $\chi_1 \in C_0^\infty(\R)$ with $0 \leq \chi_1 \leq 1$, $\chi_1(x) = 1$ for $\abs{x} \leq 1/2$, and $\chi_1(x) = 0$ for $\abs{x} \geq 1$. Let $\chi_2: (-\delta, \tau + \delta) \to \R$ be the smooth function such that
\begin{equation}
	\chi_2(t) =
	\begin{cases}
		\chi_1(\frac{t}{\delta'}) & -\delta < t \leq 0, \\
		1 & 0 \leq t \leq \tau, \\
		\chi_1(\frac{\tau - t}{\delta'}) & \tau \leq t < \tau + \delta.
	\end{cases}
\end{equation}
Finally, define the cutoff $\chi(t,y) = \chi_1(\frac{\abs{y}}{\delta'}) \chi_2(t)$. Note that $\chi$ is a compactly supported smooth function on $N_\gamma$ with $0 \leq \chi \leq 1$, $\chi(t,y) = 1$ on a $\frac{\delta'}{2}$-neighbourhood of $\gamma$, and $\chi(t,y) = 0$ outside the set $(-\delta', \tau + \delta') \times \{\abs{y} < \delta'\}$. Note also that the derivatives of $\chi$ can be computed in terms of derivatives of $\chi_1$, and hence can be bounded independently from the choice of geodesic $\gamma$. As $\delta'$ can be determined by $N, g,$ and $T$, we omit its dependence in the estimates.

\begin{thm}\label{thm:beams_tube}
Let $K \geq 2$, $C_0 > 0$, $T > 0$, and $p \in [1, \infty)$. There exist constants $c = c(N, g, T)$ and $C = C(N, g, T, E, \nabla, K, C_0, p)$ such that for all $\gamma \in \mathcal{G}_T$ and related initial conditions $\mathsf{e} \in \mathcal{A}_{Y_0, \gamma(0)}^{K}(C_0)$ for $0 \leq j \leq K$, there exists $u \in \Gamma_0(\pi)$ given by
\begin{equation}
	u = \lambda^{\frac{m-1}{2p}}e^{\I\lambda\phi}(a_0 + \lambda^{-1} a_1 + \ldots + \lambda^{-K} a_K) \chi
\end{equation}
with $\im(\phi(t,y)) \geq c \abs{y}^2$ on $\mathrm{supp}(\chi)$ and $\rho^{3K - 2j}_Y(\iota^* \mathrm{j}^{3K - 2j}a_j)(\gamma(0)) = \mathsf{e}_j$ such that
\begin{align}
	\norm{\phi}_{C^{3K+2}(N_\gamma)} &\leq C,\\
	\norm{a_j}_{C^{3K-2j}(N_\gamma)} &\leq C,\\
	\norm{u}_{L^p(N_\gamma)} &\leq C, \label{eq:bounded_Lp}\\
	\norm{(\Delta_A - \lambda^2)u}_{W^{k, p}(N_\gamma)} &\leq C\lambda^{2 + \frac{k-K}{2}}
\end{align}
for $0 \leq k \leq K-2$.
\end{thm}

\begin{proof}
We take $\phi$ as in Proposition \ref{prop:phase} with $k = 3K+2$, $a_j$ as in Proposition \ref{prop:amplitudes}, and $\chi$ as above. As $\im(\phi) \geq c\abs{y}^2$,
\begin{equation}
	\norm{u}_{L^p(N_\gamma)}^p \leq C'\lambda^{\frac{m-1}{2}} \int_{-\delta'}^{\tau + \delta'} \int_{\abs{y} < \delta'} e^{-c\lambda p \abs{y}^2} \de y \de t \leq C (\tau + 2\delta') \int_{\R^{m-1}} e^{-p\abs{y}^2} \de y
\end{equation}
where $C'$ comes from \eqref{eq:bounded_amplitudes} and depends on $p$. This shows \eqref{eq:bounded_Lp}.

Let $f = (\Delta_A - \lambda^2)u$. By our computation in Section \ref{sec:eqsbeams}, letting $a = a_0 + \lambda^{-1} a_1 + \dots + \lambda^{-K} a_K$, we have
\begin{equation}
	f = \lambda^{2 + \frac{m-1}{2p}} e^{\I\lambda \phi}[(\abs{d\phi}_g^2 - 1)a\chi + \lambda^{-1} f_0 + \ldots + \lambda^{-K-1} f_{K} + \lambda^{-K} \Delta_A(a_K \chi)]
\end{equation}
where
\begin{align}
	f_0 &= (\Delta \phi)a_0 \chi - 2(d\phi, d_A(a_0 \chi)), \\
	f_j &= (\Delta \phi)a_j \chi - 2(d\phi, d_A(a_j \chi)) - \I \Delta_A(a_{j-1}\chi), \quad 1 \leq j \leq K.
\end{align}
Note that the derivatives that hit the cutoff function $\chi$ are negligible in powers of $\lambda$ by nonstationary phase, and they can be bounded uniformly as $\delta'$ is fixed. By \eqref{eq:bounded_phase}, \eqref{eq:bounded_amplitudes}, and the boundedness of $a_j$ in $C^{3K - 2j}$, we can find $C = C(N, g, T, E, \nabla, C_0)$ such that
\begin{align}
	\abs{f} &\leq C\lambda^{2 + \frac{m-1}{2p}} e^{-c\lambda \abs{y}^2}(\abs{y}^{3K + 2} + \lambda^{-1} \abs{y}^{3K} + \ldots + \lambda^{-K-1}\abs{y}^{K} + \lambda^{-K}) \\
	&\leq Ce^{-c\lambda \abs{y}^2}\lambda^{2 + \frac{m-1}{2p}}(\abs{y}^{K} + \lambda^{-K}).
\end{align}
Therefore,
\begin{align}
	\norm{f}_{L^p(N_\gamma)}^p &\leq C \lambda^{2p + \frac{m-1}{2}} \int_{-\delta'}^{\tau + \delta'} \int_{\abs{y} < \delta'} e^{-c\lambda \abs{y}^2}(\abs{y}^{pK} + \lambda^{-pK}) \de y \de t \\
	&\leq C\lambda^{2p}(\lambda^{-\frac{pK}{2}} + \lambda^{-pK})
\end{align}
and hence $\norm{f}_{L^p(N_\gamma)} \leq C\lambda^{2 - \frac{K}{2}}$. Proceeding similarly using \eqref{eq:bounded_phase} and \eqref{eq:bounded_amplitudes}, for $1 \leq k \leq K-2$, we have
\begin{equation}
	\abs{\nabla^k f} \leq C e^{-c\lambda \abs{y}^2} \lambda^{2 + \frac{m-1}{2p}}(\abs{y}^{K - k} + \lambda^{k - K})
\end{equation}
and therefore $\norm{f}_{W^{k, p}(N_\gamma)} \leq C \lambda^{2 + \frac{k - K}{2}}$, as claimed. Note that we indeed need $k \leq K-2$ as there are already two derivatives hitting $a_K$ in the definition of $f$ and we only have a $C^K$ bound on $a_K$.
\end{proof}

The Gaussian beams in Theorem \ref{thm:beams_tube} can now be used to generate Gaussian beams on $M$. We will need to control our geodesics in different ways to get uniform estimates. We want to control the length of the geodesics, their angles of intersection with the boundary, and how close they can get to the boundary away from their intersection points.

For a geodesic $\gamma : [0, \tau] \to M$, let
\begin{align}
	\beta_- &= \sup \{\beta \in (0, \tau): \dist(\gamma(t), \del M) \text{ is strictly increasing for } t \in [0, \beta]\}, \\
	\beta_+ &= \inf \{\beta \in (0, \tau): \dist(\gamma(t), \del M) \text{ is strictly decreasing for } t \in [\beta, \tau]\}.
\end{align}
We always have $0 < \beta_- \leq \beta_+ < \tau$. We say that $\gamma$ is $\varepsilon$-separated from the boundary if the distance between $\gamma([\beta_-, \beta_+])$ and $\del M$ is greater or equal to $\varepsilon$.

For $T > 0$, $\theta_0 \in (0, \pi/2)$, and $\varepsilon > 0$, let $\mathcal{G}_{T, \theta_0, \varepsilon}$ denote the set of nontangential geodesics $\gamma : [0, \tau] \to M$ such that $\tau \leq T$, the angles between $\gamma$ and the boundary at $\gamma(0)$ and $\gamma(\tau)$ are both greater or equal to $\theta_0$, $\gamma$ is $\varepsilon$-separated from the boundary, and $\dist(\gamma(0), \gamma(\tau)) \geq 2\varepsilon$. The condition on the angles of intersection is equivalent to the equations
\begin{equation}
	-g(\dot{\gamma}(0), \nu_{\gamma(0)}) \geq \sin \theta_0, \quad g(\dot{\gamma}(\tau), \nu_{\gamma(\tau)}) \geq \sin \theta_0,
\end{equation}
where $\nu_x$ is the outward normal vector at $x \in \del M$. The last condition guarantees that $\gamma$ does not get too close to the boundary away from its endpoints, and those endpoints are sufficiently apart.

\begin{thm}\label{thm:beams_manifold}
Let $K \geq 2$, $C_0 > 0$, $T > 0$, $\theta_0 \in (0, \pi/2)$, $\varepsilon > 0$, and $p \in [1, \infty)$. There exists constants $c = c(M, g, T)$ and  $C = C(M, g, T, \theta_0, \varepsilon, E, \nabla, K, C_0, p)$ such that for all $\gamma \in \mathcal{G}_{T, \theta_0, \varepsilon}$ and all related initial conditions $\mathsf{f} \in \mathcal{A}_{\del M, \gamma(0)}^{K}(C_0)$ for $0 \leq j \leq K$, there exists $u \in \Gamma(\pi)$ such that the following holds.
\begin{enumerate}
\item[\textit{1}.] In Fermi coordinates around an $\varepsilon$-neighbourhood of $\gamma(0)$ or $\gamma(\tau)$, $u$ is given by
\begin{equation}
	u = \lambda^{\frac{m-1}{2p}} e^{\I\lambda \phi} (a_0 + \lambda^{-1} a_1 + \dots + \lambda^{-K} a_K)\chi
\end{equation}
with $\im(\phi(t,y)) \geq c \abs{y}^2$ on the support of $\chi$ and $\rho_{\del M}^{3K - 2j}(\iota^* \mathrm{j}^{3K - 2j} a_j)(\gamma(0)) = \mathsf{f}_j$.
\item[\textit{2}.] The support of $u$ is contained within a tubular neighbourhood of $\gamma$ of radius $\varepsilon$.
\item[\textit{3}.] The support of $u\vert_{\del M}$ has two connected components $V_0$ and $V_\tau$ around $\gamma(0)$ and $\gamma(\tau)$, respectively.
\item[\textit{4}.] $\norm{u}_{L^p(M)} \leq C$.
\item[\textit{5}.] $\norm{(\Delta - \lambda^2)u}_{W^{k, p}(M)} \leq C\lambda^{2 + \frac{k - K}{2}}$ for $0 \leq k \leq K-2$.
\end{enumerate}
\end{thm}

\begin{proof}
We will use the Gaussian beams we constructed in Theorem \ref{thm:beams_tube} and the properties of the map $(\pi_{\gamma})_* : \Gamma_0(E') \to \Gamma(E)$. By compactness of $\del M$ and Lemma \ref{lem:initial_conditions}, there is $C_0'$ independent of $\gamma \in \mathcal{G}_{T, \theta_0, \varepsilon}$ such that initial conditions $\mathsf{f} \in \mathcal{A}_{\del M, \gamma(0)}^{K}(C_0)$ correspond to initial conditions $\tilde{\mathsf{e}} \in \mathcal{A}_{\tilde{Y}_0, \gamma(0)}^{K}(C_0')$ where $\tilde{Y}_0$ is the hypersurface in $N$ corresponding to the hypersurface $Y_0$ in $N_\gamma$, that is,
\begin{equation}
	\tilde{Y}_0 = \exp_{\gamma(0)}(\{v \in T_{\gamma(0)}N : v \perp \dot{\gamma}(0), \abs{v} < \delta'\}).
\end{equation} 
The initial conditions $\tilde{\mathsf{e}} \in \mathcal{A}_{\tilde{Y}_0, \gamma(0)}^{K}(C_0')$ correspond naturally to initial conditions $\mathsf{e} \in \mathcal{A}_{Y_0, \gamma(0)}^{K}(C_0')$. Let $u \in \Gamma_0(E')$ be the Gaussian beam on $N_\gamma$ constructed with the initial conditions $\mathsf{e}$. We are interested in the pushforward $(\pi_\gamma)_* u$. As $\pi_\gamma$ is a local isometry, it commutes with the Laplacian. Therefore, by Lemma \ref{lem:pushforward_Hk}, the estimates \textit{4} and \textit{5} hold for $(\pi_{\gamma})_* u$. Moreover, we can choose $\delta' < \varepsilon$ in our construction so that \textit{2} holds. As $\gamma$ is $\varepsilon$-separated from the boundary and $\dist(\gamma(0), \gamma(\tau)) > 2\varepsilon$, we see that \textit{3} follows from \textit{2}. The local expression in \textit{1} follows from the local expression for $u$ in Theorem \ref{thm:beams_tube} and our matching of initial conditions.
\end{proof}

\subsection{Transport map of Gaussian beams}

Let $\gamma : [0, \tau] \to M$ be a nontangential geodesic. For $\mathsf{f} \in \mathcal{A}_{\del M, \gamma(0)}^K$, let 
\begin{equation}
	u = e^{\I\lambda \phi}(a_0 + \lambda^{-1} a_1 + \dots + \lambda^{-K} a_K)\chi
\end{equation}
be the Gaussian beam of order $K$ along $\gamma$ on $N_\gamma$ with initial conditions $\mathsf{f}$. We can then consider the jet of each $a_j$ at $\gamma(\tau)$. This induces a map $\mathcal{P}^K_\gamma : \mathcal{A}_{\del M, \gamma(0)}^K \to \mathcal{A}_{\del M, \gamma(\tau)}^K$. Explicitly, this map is given by
\begin{equation}
	\mathcal{P}^K_\gamma \mathsf{f} = (\rho^{3K}_{\del M}(\iota^* \mathrm{j}^{3K} a_0)(\gamma(\tau)), \rho^{3K-2}_{\del M}(\iota^* \mathrm{j}^{3K-2} a_1)(\gamma(\tau)), \dots, \rho^{K}_{\del M}(\iota^* \mathrm{j}^{K} a_K)(\gamma(\tau))).
\end{equation}
We call $\mathcal{P}^K_\gamma$ the Gaussian beam transport map.

When $K = 0$, $\mathcal{A}_{\del M, \gamma(0)}^K$ and $\mathcal{A}_{\del M, \gamma(\tau)}^K$ are simply the fibres of $E$ over $\gamma(0)$ and $\gamma(\tau)$ respectively. By \eqref{eq:a0_parallel_transport}, we then have
\begin{equation}
	\mathcal{P}_\gamma^0 \mathsf{f} = \exp\left(\frac{1}{2} \int_0^\tau \tr H(s) \de s\right)P^{\nabla}_\gamma \mathsf{f}.
\end{equation}
Since $H(s)$ is uniquely determined by the geometry of $M$, knowledge of $\mathcal{P}_\gamma^0$ is equivalent to knowledge of the parallel transport along $\gamma$ with respect to the connection $\nabla$. The data consisting of all such parallel transports over nontangential geodesics is also called the non-abelian X-ray transform of $\nabla$.

\subsection{Resolvent estimate and solvability}\label{sec:solvability}

We constructed approximate solutions to $(\Delta - \lambda^2)u = 0$ in the form of Gaussian beams. We now want to guarantee the existence of actual solutions to the equation that are sufficiently close to the Gaussian beams. We follow the approach in \cite{katya}.

Let $\Delta_{\Dir}$ denote the connection Laplacian associated to the Dirichlet boundary condition $u \vert_{\del M} = 0$, that is, $\Delta_{\Dir}$ acts as $\Delta$, but its domain is restricted to $H_0^1(M) \cap H^2(M)$. Following \cite[Chapter 5, \S 1]{taylor} for the case of $\Delta_{\Dir}$, by virtue of the embedding $H^1(M) \hookrightarrow L^2(M)$ being compact, we can see that $\Delta_{\Dir}$ has a compact self-adjoint inverse and hence has a discrete real spectrum $(\lambda_j^2)_{j \in \N}$ accumulating at infinity. The following Weyl's law holds \cite[Theorem 0.11]{ivrii}
\begin{equation}\label{eq:Weyllaw}
	\mathrm{N}(\lambda) := \#\{\lambda_j^2 \in \mathrm{Spec}(\Delta_{\Dir}) : 0 < \lambda_j^2 < \lambda^2\} = c \lambda^m + O(\lambda^{m-1})
\end{equation}
for some constant $c > 0$.

\begin{prop}\label{prop:resolventL2}
Given $\delta > 0$ and $\varepsilon > 0$, there exists $C = C(M, g, E, \nabla, \delta, \varepsilon) > 0$ and a set $J = J(M, g, E, \nabla, \delta, \varepsilon) \subset [1, \infty)$ with $\abs{J} \leq \delta$ such that
\begin{equation}
	\norm{(\Delta_{\Dir} - \lambda^2)^{-1}}_{\mathcal{L}(L^2(M), L^2(M))} \leq C\lambda^{m + \varepsilon},
\end{equation}
for all $\lambda \in [1, \infty) \setminus J$.	
\end{prop}

\begin{proof}
The proof is the same as in \cite[Theorem 2.1]{katya}, but for the connection Laplacian with Dirichlet boundary conditions. By \eqref{eq:Weyllaw}, the semiclassical estimate needed for the proof
\begin{equation}
	\#(\mathrm{Spec}(h^2 \Delta_{\Dir}) \cap [1,2]) = O(h^{-m})
\end{equation}
also holds.
\end{proof}

In what follows, we shall fix $J \subset [1, \infty)$ as in Proposition \ref{prop:resolventL2} with $\varepsilon = 1$. It is also understood that $\lambda \in J$ whenever $\lambda^2 \in \mathrm{Spec}(\Delta_{\Dir})$. The measure of $J$ can be made arbitrarily small at the cost of worse constants.

\begin{prop}\label{prop:resolventHk}
Let $k \in \N$. There exists $C = C(M, g, E, \nabla, \delta, k)$ such that for all $v \in H^{2k-2}(M)$ and for all $\lambda \in [1, \infty) \setminus J$, there exists $u \in H^1_0(M) \cap H^{2k}(M)$ solving $(\Delta - \lambda^2)u = v$ with
\begin{equation}
	\norm{u}_{H^{2k}(M)} \leq C\lambda^{m + 2k + 1}\norm{v}_{H^{2k-2}(M)}.
\end{equation}
\end{prop}

\begin{proof}
For $\lambda^2 \not\in \mathrm{Spec}(\Delta_{\Dir})$, given $v \in L^2(M)$, there exists a unique $u \in H_0^1(M) \cap H^2(M)$ such that $(\Delta - \lambda^2)u = v$. By elliptic regularity \cite[Chapter 5, Theorem 1.3 or Proposition 11.2]{taylor}, as $v \in H^{2k-2}(M)$, we know that $u \in H^{2k}(M)$. By iterating $\Delta u = \lambda^2 u + v$, we see that
\begin{equation}
	\Delta^k u = \lambda^{2k} u + \sum_{j=0}^{k-1} \lambda^{2(k-j-1)}\Delta^{j}v
\end{equation}
Hence, by the elliptic estimate $\norm{u}_{H^{2k}(M)} \leq C(\norm{\Delta^{k} u}_{L^2(M)} + \norm{u}_{L^2(M)})$, we have
\begin{equation}
	\norm{u}_{H^{2k}(M)} \leq C\lambda^{2k} \norm{u}_{L^2(M)} + C\lambda^{2k - 2}\norm{v}_{H^{2k-2}(M)}.
\end{equation}
By Proposition \ref{prop:resolventL2}, $\norm{u}_{L^2(M)} \leq C\lambda^{m+ 1}\norm{v}_{L^2(M)}$ for $\lambda \in [1, \infty) \setminus J$, and the estimate follows.
\end{proof}

The next Corollary follows directly from Proposition \ref{prop:resolventHk}, but we write it down explicitly as it is the solvability result we will need to guarantee the existence of actual solutions next to our Gaussian beam quasimodes.

\begin{coro}\label{coro:remainders}
Let $u$ be such that $(\Delta - \lambda^2)u \in H^{2k-2}(M)$ with $\lambda \in [1, \infty)\setminus J$. There is $r \in H^1_0(M) \cap H^{2k}(M)$ such that $(\Delta - \lambda^2)(u + r) = 0$ and
\begin{equation}
	\norm{r}_{H^{2k}(M)} \leq C \lambda^{m + 2k + 1} \norm{(\Delta - \lambda^2)u}_{H^{2k-2}(M)}.
\end{equation}
\end{coro}

We will also need the following.

\begin{coro}\label{coro:P_lambda}
For $\lambda \in [1, \infty) \setminus J$, let $P_\lambda$ be the operator sending $f \in H^{3/2}(\del M)$ to the unique function $u_f \in H^2(M)$ such that $(\Delta - \lambda^2)u_f = 0$ and $u_f \vert_{\del M} = f$. For any $k \in \N$, $P_\lambda$ extends to a bounded operator from $H^{2k - 1/2}(\del M)$ to $H^{2k}(M)$ such that
\begin{equation}
\norm{P_\lambda f}_{H^{2k}(M)} \leq C\lambda^{m + 2k + 3}\norm{f}_{H^{2k - 1/2}(\del M)}	
\end{equation}
for some $C = C(M, g, E, \nabla, \delta, k)$.
\end{coro}

\begin{proof}
Consider the operator $P_0 : H^{2k - \frac{1}{2}}(\del M) \to H^{2k}(M)$ defined similarly as $P_\lambda$ but with $\lambda = 0$. By the same arguments as in \cite[Chapter 5, Proposition 1.7]{taylor}, $P_0$ is a well-defined bounded operator. Write $u_f = P_0 f + v$. Then, $v$ must solve
\begin{equation}
\begin{cases}
	(\Delta - \lambda^2)v = - (\Delta - \lambda^2)P_0 f \\
	v \vert_{\del M} = 0	
\end{cases}
\end{equation}
By Proposition \ref{prop:resolventHk}, there is $v \in H_0^1(M) \cap H^{2k}(M)$ which solves this problem and
\begin{align*}
	\norm{v}_{H^{2k}(M)} &\leq C\lambda^{m + 2k + 1} \norm{(\Delta - \lambda^2)P_0 f}_{H^{2k - 2}(M)} \\
	&\leq C\lambda^{m + 2k + 1}(\norm{P_0 f}_{H^{2k}(M)} + \lambda^2 \norm{P_0 f}_{H^{2k - 2}(M)}) \\
	&\leq C\lambda^{m + 2k + 3} \norm{P_0 f}_{H^{2k}(M)}.
\end{align*}
Therefore, since $P_0$ is bounded,
\begin{equation}
	\norm{u_f}_{H^{2k}(M)} \leq \norm{P_0 f}_{H^{2k}(M)} +  \norm{v}_{H^{2k}(M)} \leq C\lambda^{m + 2k + 3} \norm{f}_{H^{2k- 1/2}(\del M)}.
\end{equation}
\end{proof}

\section{Calderón problem at high but fixed frequency}\label{sec:traces}

\subsection{Setting}

On the trivial bundle $E = M \times \C^n$, we have globally defined sections given by the $n$ different vectors in the canonical basis of $\C^n$. We will express all our Gaussian beams with respect to that basis. The Gaussian beams from Theorem \ref{thm:beams_manifold} for $p = 2$ can then be described as follows. Let $(x_1, \dots, x_{m-1}, r)$ be boundary normal coordinates around $\gamma(0)$ and let $(t, y_1, \dots, y_{m-1})$ be Fermi coordinates along $\gamma$, based at $\gamma(0)$. A Gaussian beam along $\gamma \in \mathcal{G}_{T, \theta_0, \varepsilon}$ can then be expressed locally near $\gamma(0)$ as
\begin{equation}
u = \lambda^{\frac{m-1}{4}}	e^{\I\lambda \phi}(a_0 + \lambda^{-1} a_1 + \dots + \lambda^{-K} a_K)\chi
\end{equation}
with $\im(\phi(t,y)) \geq c \abs{y}^2$ and $a_j : M \to \C^n$. A similar expansion also holds in a neighbourhood of $\gamma(\tau)$. The initial conditions $\mathsf{f} \in \mathcal{A}^K_{\del M, \gamma(0)}(C_0)$ are simply assigning the values of the derivatives of each $a_j$ in normal coordinates along the boundary at $\gamma(0)$ with
\begin{equation}
	\abs{(\del^I_x a_j)(\gamma(0))} \leq C_0
\end{equation}
for all $\abs{I} \leq 3K - 2j$, $j \leq K$.

We will be interested in showing that if two connections have the same Dirichlet-to-Neumann maps for some $\lambda$ large enough, then the traces of their Gaussian beams must agree. We show the following, which directly implies Theorem \ref{thm:GB_trace_single_geodesic}.

\begin{thm}\label{thm:GBtraces}
Let $K \in \Z_{\geq 0}$, $T > 0$, $\theta_0 \in (0, \pi/2)$, $\varepsilon > 0$. Let $A_1$ and $A_2$ be smooth Hermitian connections, and let $\mathcal{P}^{A_1, K}$ and $\mathcal{P}^{A_2, K}$ be the transport maps of their Gaussian beams of order $K$. For $\delta > 0$, let $J$ as in \eqref{eq:resolvent_intro} with $\abs{J} < \delta$. Unless 
\begin{equation}
	\mathcal{P}^{A_1, K}_\gamma = \mathcal{P}^{A_2, K}_\gamma
\end{equation}
for all $\gamma \in \mathcal{G}_{T, \theta_0, \varepsilon}$, there exists a constant $\lambda_0 = \lambda_0(M, g, T, \theta_0, \varepsilon, A_1, A_2, K, C_0, \delta)$ such that $\DN_\lambda^{A_1} \neq \DN_\lambda^{A_2}$ for all $\lambda \in (\lambda_0, \infty) \setminus J$.
\end{thm}

\subsection{Local estimates and stationary phase}

We record some useful estimates that will be used throughout. We denote by $D = (\del_{x_1}, \dots, \del_{x_m})$ the gradient operator (without a factor of $-\I$) on $\R^{m}$, and let $D^\alpha = \prod_{i = 1}^{m} \del_{x_i}^{\alpha_i}$.

\begin{lem}\label{lem:local_estimate}
Let $\Phi : \R^{m} \to \C$ and $a: \R^{m} \to \C$ be smooth functions. Suppose that $\Phi(0) = 0$, $\Phi'(0) = 0$, and $\im \Phi(y) \geq c \abs{y}^2$. If $D^\alpha a(0) = 0$ for all $\abs{\alpha} < N$, $\abs{D^\beta a(y)} \leq C$ for $\abs{\beta} = N$, and $\norm{\Phi}_{C^N(\R^{m})} \leq C$, then
\begin{equation}
	\norm{\lambda^{\frac{m}{4}} e^{\I\lambda \Phi} a}_{H^k(\R^{m})} \leq C' \lambda^{k - \frac{N}{2}}
\end{equation}
where $C'$ depends on $c$, $C$, $k$ and $N$.
\end{lem}

\begin{proof}
We start by computing $\norm{\lambda^{\frac{m}{4}} e^{\I\lambda \Phi}a}_{L^2}$. By Taylor's formula,
\begin{equation}
	a(y) = \sum_{\abs{\alpha} < N} \frac{D^\alpha a(0)}{\alpha!} y^\alpha + \sum_{\abs{\beta} = N} R_\beta(y) y^\beta
\end{equation}
with
\begin{equation}
	R_\beta(y) = \frac{N}{\beta!} \int_0^1 (1-t)^{N-1} D^\beta a(ty) \de t.
\end{equation}
Hence, since the derivatives of $a$ vanish up to order $N-1$,
\begin{align}
	\norm{\lambda^{\frac{m}{4}} e^{\I\lambda \Phi}a}_{L^2}^2 &= \lambda^{\frac{m}{2}}\int_{\R^{m}} e^{-2 \lambda \im \Phi} \abs{\sum_{\abs{\beta} = N} R_\beta(y) y^\beta}^2 \de y \\
	&\leq C' \lambda^{\frac{m}{2}} \int_{\R^{m}} e^{-c\lambda \abs{y}^2/2} \abs{y}^{2N} \de y.
\end{align}
The change of variable $u = \sqrt{c\lambda} y$ yields
\begin{equation}
	\lambda^{\frac{m}{2}} \int_{\R^{m}} e^{-c\lambda \abs{y}^2/2} \abs{y}^{2N} \de y = C' \lambda^{-N} \int_{\R^{m}} e^{-\abs{u}^2/2} \abs{u}^{2N} \de u \leq C' \lambda^{-\abs{N}},
\end{equation}
and hence $\norm{\lambda^{\frac{m}{4}} e^{\I\lambda \Phi} a}_{L^2} \leq C'\lambda^{-\frac{N}{2}}$.

Now let $\gamma$ be a multi-index such that $\abs{\gamma} \leq k$. Then,
\begin{equation}
	D^\gamma(e^{\I\lambda \Phi} a) = \sum_{\alpha + \beta = \gamma} c_{\alpha, \beta} (\I\lambda)^{\abs{\alpha}} (D^\alpha \Phi) (D^\beta a) e^{\I\lambda \Phi},
\end{equation}
and since $\norm{\Phi}_{C^N(\R^{m})} \leq C$,
\begin{equation}
	\norm{\lambda^{\frac{m}{4}} D^\gamma(e^{i \lambda \Phi} a)}_{L^2}^2 \leq C' \sum_{\alpha + \beta = \gamma} \lambda^{2\abs{\alpha}} \norm{\lambda^{\frac{m}{4}} e^{i \lambda \Phi} (D^\beta a)}_{L^2}^2.
\end{equation}
Note that $D^\alpha (D^\beta a)(0) = 0$ for all $\abs{\alpha} < N - \beta$ and $\abs{D^\alpha(D^\beta a)(y)} \leq C$ for $\abs{\alpha} = N - \abs{\beta}$. Hence, by applying what we proved above for $D^\beta a$ instead of $a$ and $N - \abs{\beta}$ instead of $N$, we get $\norm{\lambda^{\frac{m}{4}} e^{i \lambda \Phi} (D^\beta a)}_{L^2} \leq C' \lambda^{\frac{-N + \abs{\beta}}{2}}$. Therefore,
\begin{equation}
	\norm{\lambda^{\frac{m}{4}} D^\gamma(e^{i \lambda \Phi} a)}_{L^2}^2 \leq C'\sum_{\alpha + \beta = \gamma} \lambda^{2\abs{\alpha}} \lambda^{-N + \abs{\beta}} \leq C' \lambda^{2\abs{\gamma} - N},
\end{equation}
and the result follows.
\end{proof}

We will need to use the following version of the stationary phase Theorem from Hörmander. 

\begin{thm}[{\cite[Theorem 7.7.5]{hormander1}}]\label{thm:stationaryphase}
Let $U \subset \R^{m}$ be a neighbourhood of $0$ and let $\Phi : U \to \C$ be a phase function with $\im \Phi \geq 0$ and a single nondegenerate critical point at $0$, that is, $\Phi'(0) = 0$ and $\det(\Phi''(0)) \neq 0$. Without loss, suppose as well that $\Phi(0) = 0$. Let $u \in C^\infty_0(U)$. Then,
\begin{equation}\label{eq:stationary_phase}
	\abs{\lambda^{\frac{m}{2}} \int_{U} e^{\I\lambda \Phi} u \de x - \det(\Phi''(0)/(2\pi i))^{-\frac{1}{2}} \sum_{j < N} \lambda^{-j} L_j u} \leq C\lambda^{-N} \sum_{\abs{\alpha} \leq 2N + m + 1} \sup_{x \in U} \abs{D^\alpha u(x)}
\end{equation}
where $L_j$ are differential operators of order $2j$. The operator $L_j$ has the form
\begin{equation}
	L_j u = \sum_{\mu = 0}^{2j} \frac{2^{-(j+\mu)} \I^{-j}}{\mu! (j+\mu)!} Q^{j+\mu}(g_0^{\mu} u) \vert_{x = 0}
\end{equation}
where $Q = \dual{-(\Phi''(0))^{-1}D}{D}$ and $g_0$ is a function vanishing to third order at $0$. The constant $C$ can be chosen uniformly over all $\Phi$ in a $C^{3N + 1}$-ball.
\end{thm}

Note that the estimate in \cite{hormander1} is actually written as
\begin{equation}\label{eq:stationary_phase_original}
	\abs{\lambda^{\frac{m}{2}} \int_{U} e^{\I\lambda \Phi} u \de x - \det(\Phi''(0)/(2\pi i))^{-\frac{1}{2}} \sum_{j < k} \lambda^{-j} L_j u} \leq C\lambda^{\frac{m}{2}-k} \sum_{\abs{\alpha} \leq 2k} \sup_{x \in U} \abs{D^\alpha u(x)}.
\end{equation}
If we choose $N$ such that $2N + m \leq 2k \leq 2N + m + 1$ and send the terms with $j \geq N$ into the error term, we see that we can derive \eqref{eq:stationary_phase} from \eqref{eq:stationary_phase_original} since $L_j$ is a differential operator of order at most $2k$.

\subsection{Powers of the Laplacian and polynomials}\label{sec:polynomials}

Let $B$ be a positive-definite $m \times m$ matrix and consider the differential operator $Q = \dual{BD}{D} = \sum_{i, j=1}^m B_{ij} \del_{x_i} \del_{x_j}$. Let $(v_j)_{j=1}^m$ be an orthonormal basis of eigenvectors of $B$ with corresponding eigenvalues $\omega_j$. We can then rewrite $Q$ as
\begin{equation}
	Q = \sum_{j=1}^m (\omega_j \del_{v_j})^2.
\end{equation}
For a multi-index $\alpha \in (\Z_{\geq 0})^m$, let us denote $D^\alpha_\omega = \prod_{j=1}^m (\omega_j \del_{v_j})^{\alpha_j}$. We associate to $v_j$ the variable $z_j = \frac{1}{\omega_j} \dual{v_j}{x} = \frac{1}{\omega_j} (v_{j1} x_1 + \dots + v_{jm} x_m)$. We see that $\omega_j\del_{v_j}(z_k) = \dual{v_j}{v_k} = \delta_{jk}$ and, more generally,
\begin{equation}
	D^{\alpha}_\omega z^\beta = \frac{\beta!}{(\beta - \alpha)!} z^{\beta - \alpha}
\end{equation}
where $z^\beta = \prod_{j = 1}^m z_j^{\beta_j}$ for $\beta \in (\Z_{\geq 0})^m$. In the variables $z_j$, $D^\alpha_\omega$ therefore acts as $D_z^\alpha$ and the operator $Q$ corresponds to $\Delta = \Delta_z$.

Let $\mathcal{H}^k$ denote the vector space of homogeneous polynomials of degree $k$ over $\R^m$ with coefficients in $\C$. Consider the matrices
\begin{equation}
	Z = \begin{pmatrix}
		\omega_1 v_1 & \dots &  \omega_m v_m
	\end{pmatrix}, \quad
	Z^{-1} = \begin{pmatrix}
		\frac{1}{\omega_1} v_1^T \\
		\vdots \\
		\frac{1}{\omega_m} v_m^T
	\end{pmatrix},
\end{equation}
so that $z = Z^{-1} x$. These matrices induce isomorphisms $Z^* : \mathcal{H}^k \to \mathcal{H}^k$ via
\begin{equation}
	(Z^*p)(z) = p(Zz) = p(x).
\end{equation}
The inverse of $Z^*$ is simply $(Z^{-1})^*$ which is defined similarly as $((Z^{-1})^* p)(x) = p(Z^{-1} x) = p(z)$. By writing $p(x) = (Z^* p)(Z^{-1} x)$ and using the chain rule, we can find a constant $C_k > 0$ such that
\begin{equation}\label{eq:z_to_x}
	\sum_{\abs{\alpha} = k} \abs{D^\alpha_x p(x)} \leq \frac{C_k}{\omega_{\min}^k} \sum_{\abs{\alpha} = k} \abs{D^\alpha_z (Z^* p)(z)}.
\end{equation}
Similarly, by writing $(Z^* p)(z) = p(Zz)$, there is a constant $C_k > 0$ such that
\begin{equation}\label{eq:x_to_z}
	\sum_{\abs{\alpha} = k} \abs{D^\alpha_z (Z^* p)(z)} \leq C_k\omega^k_{\max} \sum_{\abs{\alpha} = k} \abs{D^\alpha_x p(x)}.
\end{equation}

If we consider $Q^k : \mathcal{H}^{k-r} \times \mathcal{H}^{k+r} \to \C$ as the sesquilinear form $(p, q) \mapsto Q^k(\overline{p}q)$ for $\abs{r} \leq k$, we have
\begin{equation}
	Q^k(\overline{p} q) = \Delta^k((Z^* \overline{p})(Z^* q)).
\end{equation} 
In matrix notation, $Q^k = (Z^*)^T \Delta^k Z^*$. When $r = 0$, this form defines an inner product.

\begin{lem}\label{lem:Deltak}
The map $(p, q) \mapsto \Delta^k(\overline{p}q)$ is an inner product on $\mathcal{H}^k$.	
\end{lem}

\begin{proof}
As $\Delta^k$ is a differential operator of order $2k$ without any non-principal part and $\overline{p}q$ is a homogeneous polynomial of degree $2k$, $\Delta^k(\overline{p}q)$ indeed takes scalar values. It suffices to show that the map is positive-definite, that is, $\Delta^k(\abs{p}^2) \geq 0$ with equality if and only if $p = 0$. From \cite[Lemma 1.1]{polyharmonic}, we can write
\begin{equation}
\Delta^k(\overline{p} q) = \sum_{j + \ell + \abs{\alpha} = k} \frac{2^{\abs{\alpha}}k!}{j!\ell!\alpha!} (D^\alpha \Delta^j \overline{p})(D^\alpha \Delta^\ell q).	
\end{equation}
Therefore,
\begin{equation}\label{eq:Deltak}
\Delta^k(\abs{p}^2) = \sum_{2j + \abs{\alpha} = k} \frac{2^{\abs{\alpha}} k!}{(j!)^2 \alpha!} \abs{D^\alpha \Delta^j p}^2.
\end{equation}
Indeed, as $D^{\alpha} \Delta^j$ and $D^{\alpha} \Delta^\ell$ are differential operators of order $2j + \abs{\alpha}$ and $2\ell + \abs{\alpha}$ respectively, and $(2j + \abs{\alpha}) + (2\ell + \abs{\alpha}) = 2k$, we see that $(D^\alpha \Delta^j \overline{p})(D^\alpha \Delta^\ell p)$ is nonzero only if $2j + \abs{\alpha} = 2\ell + \abs{\alpha} = k$, which implies $j = \ell$. Moreover, $D^\alpha \Delta^j \overline{p} = \overline{D^\alpha \Delta^j p}$. This shows $\Delta^k(\abs{p}^2) \geq 0$, and we see that if equality holds, then $\abs{D^\alpha p} = 0$ for all $\abs{\alpha} = k$, which implies $p = 0$ as $p$ is a homogeneous polynomial of order $k$.
\end{proof}

In what follows, we will need to consider sesquilinear forms on homogeneous polynomials. For $k \geq 0$, consider the spaces
\begin{align}
	&\mathcal{H}_{\oplus}^{2k} = \mathcal{H}^{2k} \oplus \mathcal{H}^{2k-2} \oplus \dots \oplus \mathcal{H}^0, \\
	&\mathcal{H}_{\oplus}^{2k+1} = \mathcal{H}^{2k+1} \oplus \mathcal{H}^{2k-1} \oplus \dots \oplus \mathcal{H}^1.
\end{align}
Let $c_r = \frac{2^{-r}}{r!}$ and let $\zeta_{2k} : \mathcal{H}_{\oplus}^{2k} \times \mathcal{H}_{\oplus}^{2k} \to \C$ be the sesquilinear form given by
\begin{equation}
	\zeta_{2k}(p, q) = \zeta_{2k}((p_0, \dots, p_k), (q_0, \dots, q_k)) = \sum_{\substack{i + j + r = 2k \\ 0 \leq i, j \leq k}} c_r \Delta^r(\overline{p_i} q_j)
\end{equation}
where $p_i$ and $q_j$ are homogeneous polynomials of order $2k - 2i$ and $2k - 2j$, respectively. Define $\zeta_{2k+1} : \mathcal{H}^{2k+1}_\oplus \times \mathcal{H}^{2k+1}_\oplus \to \C$ analogously. We show the following.

\begin{lem}\label{lem:zeta_inner_product}
The sesquilinear forms $\zeta_{2k}$ and $\zeta_{2k+1}$ are inner products on $\mathcal{H}_{\oplus}^{2k}$ and $\mathcal{H}_{\oplus}^{2k+1}$, respectively.
\end{lem}

Before we prove Lemma \ref{lem:zeta_inner_product}, we will need the following identity.

\begin{prop}\label{prop:zsquared_identity}
The following identity holds
\begin{equation}
	\Delta^k(\abs{z}^2 p q) = 2k(m + 2k -2)\Delta^{k-1}(p q)	
\end{equation}
for all $p \in \mathcal{H}^{k+r-1}$ and $q \in \mathcal{H}^{k-r-1}$ with $\abs{r} \leq k - 1$.
\end{prop}

\begin{proof}
It suffices to show $\Delta^k(\abs{z}^2 p) = 2k(m + 2k - 2)\Delta^{k-1}p$ for all homogeneous polynomials $p \in \mathcal{H}^{2k-2}$. We have
\begin{equation}
	\Delta^k(z_i^2 p) = \sum_{j + \ell + \abs{\alpha} = k} \frac{2^{\abs{\alpha}}k!}{j! \ell! \alpha!} (D^\alpha \Delta^j z_i^2)(D^\alpha \Delta^\ell p).
\end{equation}
The summand is nonzero only when $j = 1$, $\abs{\alpha} = 0$ and $\ell = k-1$ or $j = 0$, $\alpha = 2e_i$, $\ell = k-2$. Therefore,
\begin{equation}
	\Delta^k(z_i^2 p) = 2k \Delta^{k-1} p + 4k(k-1) D^{2e_i} \Delta^{k-2} p.
\end{equation}
Summing from $i = 1$ to $m$ yields the identity.
\end{proof}

\begin{proof}[Proof of Lemma \ref{lem:zeta_inner_product}]
We will show the result for $\zeta_{2k}$. The proof for $\zeta_{2k+1}$ is analogous. Let $f_k = 2k(m + 2k - 2)$. By Proposition \ref{prop:zsquared_identity},
\begin{equation}
	\zeta_{2k}(p,q) = \sum_{\substack{i + j + r = 2k \\ 0 \leq i, j \leq k}} \frac{c_r}{f_{2k} f_{2k-1} \dots f_{r+1}} \Delta^{2k}(\abs{z}^{2(i + j)} p_i q_j).
\end{equation}
Let us rewrite this expression in matrix notation. Seeing $\abs{z}^{2j} : \mathcal{H}^{2k - 2j} \to \mathcal{H}^{2k}$ as the injection that maps $p \in \mathcal{H}^{2k - 2j}$ to $\abs{z}^{2j}p \in \mathcal{H}^{2k}$ and seeing $\Delta^{2k}$ as a map from $\mathcal{H}^{2k}$ to $(\mathcal{H}^{2k})^*$, we can rewrite $\zeta_{2k}(p,q)$ as
\begin{align}
\begin{pmatrix}
	p_0 \\
	p_1 \\
	\vdots \\
	p_k
\end{pmatrix}^\dagger &
\begin{pmatrix}
	I & 0 & \dots & 0 \\
	0 & \abs{z}^2 & \dots & 0 \\
	\vdots & \vdots & \ddots & \vdots \\
	0 & 0 & \dots & \abs{z}^{2k}
\end{pmatrix}^T 
\begin{pmatrix}
	c_{2k}I & \frac{c_{2k-1}}{f_{2k}} I & \dots & \frac{c_k}{f_{2k} \dots f_{k+1}} I \\
	\frac{c_{2k-1}}{f_{2k}} I & \frac{c_{2k-2}}{f_{2k} f_{2k-1}} I & \dots & \frac{c_{k-1}}{f_{2k} \dots f_{k}} I \\
	\vdots & \vdots & \ddots & \vdots \\
	\frac{c_k}{f_{2k} \dots f_{k+1}} I & \frac{c_{k-1}}{f_{2k} \dots f_{k}} I & \dots & \frac{c_0}{f_{2k} \dots f_1} I
\end{pmatrix} \\
&\begin{pmatrix}
	\Delta^{2k} & 0 & \dots & 0 \\
	0 & \Delta^{2k} & \dots & 0 \\
	\vdots & \vdots & \ddots & \vdots \\
	0 & 0 & \dots & \Delta^{2k}
\end{pmatrix}
\begin{pmatrix}
	I & 0 & \dots & 0 \\
	0 & \abs{z}^2 & \dots & 0 \\
	\vdots & \vdots & \ddots & \vdots \\
	0 & 0 & \dots & \abs{z}^{2k}
\end{pmatrix} 
\begin{pmatrix}
	q_0 \\
	q_1 \\
	\vdots \\
	q_k
\end{pmatrix},
\end{align}
or in a more compact notation as $\zeta_{2k}(p, q) = p^\dagger \Upsilon^T(C_{2k} \otimes I)(I \otimes \Delta^{2k}) \Upsilon q$. Here, $\otimes$ is the Kronecker product of matrices. As $\Upsilon$ is one-to-one, it suffices to show $(C_{2k} \otimes I)(I \otimes \Delta^{2k})$ is positive-definite. The set of eigenvalues of $I \otimes \Delta^{2k}$ agrees with the set of eigenvalues of $\Delta^{2k}$. By Lemma \ref{lem:Deltak}, $\Delta^{2k}$ is positive-definite and hence $I \otimes \Delta^{2k}$ is positive-definite as well. Since $C_{2k} \otimes I$ and $I \otimes \Delta^{2k}$ commute, their product is positive-definite if $C_{2k} \otimes I$ is positive-definite, which is equivalent to $C_{2k}$ being positive-definite.

We have
\begin{equation}
C_{2k} = \frac{1}{f_{2k} \cdots f_1}
\begin{pmatrix}
	c_{2k} (f_{2k} \ldots f_1) & c_{2k-1} (f_{2k-1} \ldots f_1) & \dots & c_k(f_k \ldots f_1)\\
	c_{2k-1} (f_{2k-1} \ldots f_1) & c_{2k-2} (f_{2k-2} \ldots f_1) & \dots & c_{k-1}(f_{k-1} \ldots f_1)\\
	\vdots & \vdots & \ddots & \vdots \\
	c_k(f_k \ldots f_1) & c_{k-1}(f_{k-1} \ldots f_1) & \dots & c_0
\end{pmatrix}
\end{equation}
Let $X_{2k}$ be the matrix so that $C_{2k} = \frac{1}{f_{2k} \dots f_1} X_{2k}$. The coefficients of $B_{2k}$ are given by
\begin{equation}
	c_r(f_r \ldots f_1) = \frac{2^{-r}}{r!} \prod_{j=1}^r 2j(m + 2j -2) = \prod_{j=1}^r (m + 2j -2) = \frac{(m + 2r -2)!!}{(m-2)!!}
\end{equation}
where $!!$ is the double factorial, defined recursively as $n!! = n (n-2)!!$ with $0!! = 1$ and $1!! = 1$. In the notation of Proposition \ref{prop:bigmatrix} below, $X_{2k} = \frac{X(m-2, k)}{(m-2)!!}$ and hence $X_{2k}$ is positive-definite. Therefore, $C_{2k}$ is also positive-definite, and the result follows.
\end{proof}

\begin{prop}\label{prop:bigmatrix}
For $n \geq 0$, $k \geq 0$, let $X(n, k)$ be the symmetric $(k + 1) \times (k + 1)$ matrix
\begin{equation}
X(n, k) = 
\begin{pmatrix}
	(n + 4k)!! & (n + 4k - 2)!! & \dots & (n + 2k)!! \\
	(n + 4k - 2)!! & (n + 4k - 4)!! & \dots & (n + 2k - 2)!! \\
	\vdots & \vdots & \ddots & \vdots \\
	(n + 2k)!! & (n + 2k - 2)!! & \dots & n!!
\end{pmatrix}.
\end{equation}
Then,
\begin{equation}\label{eq:det_Mnk}
	\det X(n,k) = 2^{\frac{k(k+1)}{2}} \prod_{j=0}^k j! (n + 2j)!!
\end{equation}
and $X(n, k)$ is positive-definite.
\end{prop}

\begin{proof}
Let $p(n,k) = \det X(n, k)$. For a $(k+1) \times (k+1)$ matrix $X$, let $X^j_i$ be the submatrix of $X$ that results from removing the $i$-th row and the $j$-th column. Similarly, let $X^{j_1, j_2}_{i_1, i_2}$ be the submatrix of $X$ that results from removing the rows $i_1$ and $i_2$, as well as the columns $j_1$ and $j_2$. The Desnanot–Jacobi identity \cite[Theorem 3.12]{bressoud} reads
\begin{equation}
	(\det X)(\det X_{1,k+1}^{1,k+1}) = \det(X^1_1)\det(X^{k+1}_{k+1}) - \det(X^{k+1}_1)\det(X^{1}_{k+1}).
\end{equation}
Applying this identity to $X(n, k)$ yields the recursion relation
\begin{equation}
	p(n, k) p(n + 4, k -2) = p(n, k - 1) p(n + 4, k - 1) - p(n + 2, k - 1)^2
\end{equation}
with initial conditions $p(n, 0) = n!!$ and $p(n, 1) = 2 n!!(n+2)!!$. One can check that \eqref{eq:det_Mnk} satisfies the recursion relation. The determinants are all positive and hence $X(n, k)$ is positive-definite by Sylvester's criterion since the $j \times j$ leading principal minor of $X(n, k)$ is $X(n + 4(k + 1 - j), k)$ for $1 \leq j \leq k + 1$.
\end{proof}

\begin{lem}\label{lem:dual_estimate}
Let $p = (p_0, \dots, p_k) \in \mathcal{H}^{2k}_{\oplus}$, and let $C, C' > 0$. Suppose that $\abs{\zeta_{2k}(p, q)} \leq C\lambda^{-1}$ for all $q = (q_0, \dots, q_k) \in \mathcal{H}^{2k}_{\oplus}$ with $\abs{D^\alpha q_j (0)} \leq C'$ for all $\abs{\alpha} = 2k-2j$. Then, there is a constant $C'' > 0$ depending only on $m$ and $k$ such that
\begin{equation}
	\abs{D^\alpha p_j(0)} \leq C'' \frac{C}{C'} \lambda^{-1}
\end{equation}
for all $\abs{\alpha} = 2k - 2j$.
\end{lem}

\begin{proof}
If $\dual{\cdot}{\cdot}$ is an inner product on a vector space $V$ and $\abs{\dual{v}{w}} \leq C$ for all $w \in V$ with $\norm{w} \leq C'$, then
\begin{equation}
C' \norm{v} = \abs{\dual{v}{C'\frac{v}{\norm{v}}}} \leq C
\end{equation}
and therefore $\norm{v} \leq \frac{C}{C'}$. The result follows by considering $\zeta_{2k}$ as an inner product on $\mathcal{H}^{2k}_{\oplus}$ and by seeing $\sqrt{\zeta_{2k}(p, p)}$ and $\sum_{j=0}^k\sum_{\abs{\alpha} = 2k-2j} \abs{D^\alpha p_j (0)}$ as equivalent norms on a finite-dimensional vector space.
\end{proof}

\subsection{Proof of Theorem \ref{thm:GBtraces}}
Let $A_1$ and $A_2$ be two smooth Hermitian connections. Let us assume that their respective Dirichlet-to-Neumann maps agree, that is, $\DN_\lambda^{A_1} = \DN_\lambda^{A_2}$. By \cite[Lemma 2.3]{cekicYM}, we can assume through an action of the gauge that $A_1(\del_r) = A_2(\del_r) = 0$ on a small neighbourhood of the boundary (here $\del_r$ is the vector field from boundary normal coordinates that agrees with $\del_\nu$ on the boundary). Therefore, for any $f : M \to \C^n$, we have
\begin{equation}\label{eq:gauge_del_nu}
	(d_{A_1} f)(\nu)\vert_{\del M} = (d_{A_2} f)(\nu)\vert_{\del M} = \del_\nu f \vert_{\del M}.
\end{equation}

Let $\gamma \in \mathcal{G}_{T, \theta_0, \varepsilon}$. We let $u_1$ and $u_2$ be Gaussian beam quasimode of order $K$ along $\gamma$ with respect to $A_1$ and $A_2$, respectively. For fixed $\delta > 0$ and $\varepsilon = 1$, let $J_\ell$ be as in Proposition \ref{prop:resolventL2} for $A_\ell$ with measure $\abs{J_\ell} < \delta/2$. We fix $J = J_1 \cup J_2$. By Corollary \ref{coro:remainders}, there are $r_1$ and $r_2$ such that $(\Delta_{A_\ell} - \lambda^2)(u_\ell + r_\ell) = 0$, $r_\ell \vert_{\del M} = 0$, and
\begin{equation}\label{eq:remainder}
	\norm{r_\ell}_{H^{2k}(M)} \leq C \lambda^{m + 2k + 1} \norm{(\Delta_{A_\ell} - \lambda^2)u_\ell}_{H^{2k-2}(M)}
\end{equation}
for all $\lambda \in [1,\infty)\setminus J$. Let $f_\ell = (u_\ell + r_\ell)\vert_{\del M} = u_\ell \vert_{\del M}$, so that $\DN_\lambda^{A_j} f_\ell = \del_\nu (u_\ell + r_\ell) \vert_{\del M}$ by \eqref{eq:gauge_del_nu}.

Near $\gamma(0)$, we write $u_\ell = \lambda^{\frac{m-1}{4}}e^{\I\lambda \phi}a^{(\ell)} \chi$ with $a^{(\ell)} = a^{(\ell)}_0 + \lambda a^{(\ell)}_1 + \dots + \lambda^{-K} a^{(\ell)}_K$ for $\ell = 1, 2$, and where $\chi$ is a fixed cut-off function coming from the construction on $N_\gamma$. Since $\mathsf{f}_1 = \mathsf{f}_2$,
\begin{equation}\label{eq:initial_conditions_match}
	(\del^I_x a^{(1)}_j)(\gamma(0)) = (\del^I_x a^{(2)}_j)(\gamma(0))
\end{equation}
for all $\abs{I} \leq 3K - 2j$, $j \leq K$.

 Let $v$ be another Gaussian beam quasimode of order $K$ along $\gamma$ with respect to $A_2$. Contrary to $u_1$ and $u_2$, we will control the initial conditions of $v$ at $\gamma(\tau)$ rather than at $\gamma(0)$. Near $\gamma(\tau)$, we write $v = \lambda^{\frac{m-1}{4}} e^{\I\lambda \phi} b \chi$ with $b = b_0 + \lambda b_1 + \dots + \lambda^{-K} a_K$. We denote the initial conditions of $v$ by $\mathsf{h} \in \mathcal{A}_{\del M, \gamma(\tau)}^K(C_0)$. Controlling the initial conditions $\mathsf{h}$ amounts to choosing arbitrary values for the derivatives in normal coordinates around $\gamma(\tau)$ of the amplitudes $b_j$ at $\gamma(\tau)$ such that
 \begin{equation}
 	\abs{(\del_x^I b_j)(\gamma(\tau))} \leq C_0
 \end{equation}
 for all $\abs{I} \leq 3K - 2j$, $j \leq K$. As for the other Gaussian beams, we can find $r$ with $r \vert_{\del M} = 0$ such that $(\Delta_{A_2} - \lambda^2)(v + r) = 0$ and $r$ satisfies an estimate similar to \eqref{eq:remainder}. We let $h = (v + r)\vert_{\del M} = v \vert_{\del M}$.

As $\DN_\lambda^{A_1} = \DN_\lambda^{A_2}$ and the operators are self-adjoint, we have
\begin{equation}\label{eq:magic_formula}
	0 = \dual{(\DN_\lambda^{A_\ell} - \DN_\lambda^{A_2})f_\ell}{h}_{L^2(\del M)} = \dual{\del_\nu(u_\ell + r_\ell)}{h}_{L^2(\del M)} - \dual{f_\ell}{\del_\nu(v + r)}_{L^2(\del M)}.
\end{equation}
By the local expression for $v$, the trace theorem and the bounds on the derivatives of the phase and the amplitudes, we have
\begin{equation}\label{eq:bounds_v}
	\max \{\norm{h}_{L^2(\del M)}, \norm{\del_\nu v}_{L^2(\del M)}\} \leq C\norm{v}_{H^2(M)} \leq C\lambda^2.
\end{equation}
By \eqref{eq:remainder} and the estimate \textit{5} in Theorem \ref{thm:beams_manifold}, we have
\begin{equation}
	\norm{\del_\nu r}_{L^2(\del M)} \leq C \norm{r}_{H^2(M)} \leq C \lambda^{m+3}\norm{(\Delta_{A_2} - \lambda^2)v}_{L^2(M)} \leq C \lambda^{m + 5 - \frac{K}{2}}.
\end{equation}
Similar expressions also hold for $u_\ell$, $f_\ell$, and $r_\ell$. By Cauchy-Schwarz, \eqref{eq:magic_formula} then yields
\begin{equation}
	\dual{\del_\nu u_\ell}{h}_{L^2(\del M)} - \dual{f_\ell}{\del_\nu v}_{L^2(\del M)} = O(\lambda^{-R})
\end{equation}
with $R = \frac{K}{2} - m - 7$. By choosing $K$ large enough, we can always pick $R$ to be positive. Writing $u = u_1 - u_2$ and $f = f_1 - f_2$, we can subtract the equations for $\ell = 1$ and $\ell = 2$ to get
\begin{equation}\label{eq:magic_equation_bigO}
	\dual{\del_\nu u}{h}_{L^2(\del M)} - \dual{f}{\del_\nu v}_{L^2(\del M)} = O(\lambda^{-R}).
\end{equation}
Here and in what follows, the implicit constants in the big O notations will be uniform with respect to $C_0$ and $\gamma \in \mathcal{G}_{T, \theta_0, \varepsilon}$.

By Theorem \ref{thm:beams_manifold}, both integrals in \eqref{eq:magic_equation_bigO} are supported on small disjoint neighbourhoods $V_0$ and $V_\tau$. Let $U_0$ be a small neighbourhood of $V_0$ in $M$. On $U_0$, we can expand $u = \lambda^{\frac{m-1}{4}} e^{\I\lambda \phi}a\chi$ with $a = a^{(1)} - a^{(2)}$. Let $a_j = a_j^{(1)} - a_j^{(2)}$. By \eqref{eq:initial_conditions_match} and Lemma \ref{lem:local_estimate}, we see that
\begin{equation}
	\norm{f}_{L^2(V_0)} \leq C\lambda^{-\frac{3K}{2}}, \quad \norm{\del_\nu u}_{L^2(V_0)} \leq C\lambda^{2-\frac{3K}{2}}.
\end{equation}
By combining these estimates with \eqref{eq:bounds_v} we see that the contribution on $V_0$ of both integrals in \eqref{eq:magic_equation_bigO} is at most $O(\lambda^{4 - \frac{3K}{2}})$ and can therefore be included in $O(\lambda^{-R})$. Similarly, $\norm{u_1 - u_2}_{H^k(V_0)} \leq C \lambda^{k - R}$. The only meaningful contributions of the integrals in \eqref{eq:magic_equation_bigO} are therefore on a small neighbourhood $V_\tau$ of $\gamma(\tau)$, that is,
\begin{equation}\label{eq:magic_Vtau}
	\dual{\del_\nu u}{h}_{L^2(V_\tau)} - \dual{f}{\del_\nu v}_{L^2(V_\tau)} = O(\lambda^{-R}).
\end{equation}

Let $U_\tau$ be a small neighbourhood of $V_\tau$ in $M$. As we did on $U_0$, we can expand $u = \lambda^{\frac{m-1}{2}} e^{\I\lambda \phi} a \chi$ on $U_\tau$ with $a = a^{(1)} - a^{(2)}$. Denoting $\dual{\cdot}{\cdot} = \dual{\cdot}{\cdot}_{L^2(V_\tau)}$, equation \eqref{eq:magic_Vtau} then becomes
\begin{align}
	0 &= \lambda^{\frac{m-1}{2}}(\dual{\del_\nu(e^{\I\lambda \phi} a\chi)}{e^{\I\lambda \phi}b\chi} - \dual{e^{\I\lambda \phi} a\chi}{\del_\nu(e^{\I\lambda \phi} b\chi)}) + O(\lambda^{-R}) \\
	&= \lambda^{\frac{m-1}{2}}(\dual{2\I\lambda \re(\del_\nu \phi) e^{-2\lambda \im \phi} a\chi^2}{b} + \dual{e^{-2\lambda\im \phi}(\del_\nu a)\chi^2}{b} - \dual{e^{-2\lambda \im \phi}a\chi^2}{\del_\nu b}) + O(\lambda^{-R})
\end{align}
since the terms involving derivatives of $\chi$ can be absorbed in $O(\lambda^{-R})$ because $\chi = 1$ on a neighbourhood of $\gamma(\tau)$. By expanding even further, grouping the terms with the same order in $\lambda$, and discarding the terms of order smaller than $O(\lambda^{-R})$, we get
\begin{align}\label{eq:main_identity_traces}
	&\lambda^{\frac{m-1}{2}}\sum_{\ell \leq R} \lambda^{-\ell + 1} \sum_{j = 0}^{\ell} \dual{2\I\re(\del_\nu \phi)e^{-2\lambda \im \phi}a_j\chi^2}{b_{\ell - j}} \\
	&\quad + \lambda^{\frac{m-1}{2}}\sum_{\ell \leq R-1} \lambda^{-\ell} \sum_{j = 0}^{\ell} \left[\dual{e^{-2\lambda \im \phi} (\del_\nu a_j) \chi^2}{b_{\ell-j}} - \dual{e^{-2\lambda \im \phi} a_j \chi^2}{\del_\nu b_{\ell - j}}\right] = O(\lambda^{-R}).
\end{align}
Each term can be seen as a stationary phase integral on $\R^{m-1}$ with phase $\Phi = 2\I \im \phi$. Using boundary normal coordinates $(x_1, \dots, x_{m-1}, r)$ around $\gamma(\tau) = (0, \dots, 0)$ on $U_\tau$, we have $\del M = \{r = 0\}$ and $\de V_{\del M} = \rho \de x_1 \dots \de x_{m-1}$ with $\rho = 1 + O(\abs{x}^2)$. This yields
\begin{equation}
	\lambda^{\frac{m-1}{2}} \dual{2\I\re(\del_\nu \phi)e^{-2\lambda \im \phi}a_j\chi^2}{b_{\ell - j}} = \lambda^{\frac{m-1}{2}} \int_{\R^{m-1}} e^{-2\lambda \im \phi} \dual{a_j}{\tilde{b}_{\ell - j}} \de x
\end{equation}
where $\tilde{b}_{\ell - j} = (-2\I \re(\del_\nu \phi) \chi^2 \rho) b_{\ell - j}$. By Theorem \ref{thm:beams_manifold}, we can write $\phi$ in Fermi coordinates $(t, y_1, \dots, y_{m-1})$ around $\gamma(\tau) = (\tau, 0)$ as
\begin{equation}
	\phi(t,y) = t + \frac{1}{2} H(t)y \cdot y + O(\abs{y}^3)
\end{equation}
with $\im H(t)$ positive-definite and $\im \phi > c \abs{y}^2$.

\begin{lem}\label{lem:hessian}
Let $\phi$ be as above, and let $\cos \theta = g_{\gamma(\tau)}(\del_t, \del_\nu)$. Then,
\begin{equation}
	(D^2_{x} \im \phi \vert_{\del M})\vert_{\gamma(\tau)} = \mathrm{Hess}_{\del M}(\im \phi \vert_{\del M}) \vert_{\gamma(\tau)} \geq c \cos^2 \theta I.
\end{equation}
\end{lem}

\begin{proof}
In Fermi coordinates, we have
\begin{equation}
	\mathrm{Hess}_M(\im \phi)\vert_{\gamma(\tau)} = (D^2_{(t,y)} \im \phi)\vert_{\gamma(\tau)} = 
	\begin{pmatrix}
		\im H(\tau) & 0 \\
		0 & 0
	\end{pmatrix}
\end{equation}
where $\mathrm{Hess}_M(\im \phi)$ is the $m \times m$ Hessian matrix of $\im \phi$ in $M$. The Hessian matrices correspond to $(0,2)$-tensor fields and, for any $F \in C^2(M)$, we have
\begin{equation}
\mathrm{Hess}_{\del M}(F\vert_{\del M})(	X, Y) = \mathrm{Hess}_{M}(F)(\tilde{X}, \tilde{Y}) + (\del_{\nu}F) \mathrm{I\!I}(X, Y)
\end{equation}
where $X$ and $Y$ are vector fields on $\del M$, $\tilde{X}$ and $\tilde{Y}$ are extensions of those vector fields to $M$, and $\mathrm{I\!I}$ is the second fundamental form. As $\del_\nu \im \phi \vert_{\gamma(T)} = 0$, $\mathrm{Hess}_{\del M}(\im \phi \vert_{\del M})\vert_{\gamma(\tau)} = (\mathrm{Hess}_{M}(\im \phi)\vert_{\gamma(\tau)})\vert_{T \del M}$. Let $V \subset T_{\gamma(\tau)}M$ be the subspace generated by $\del_{y_1}, \dots, \del_{y_{m-1}}$, or equivalently, $V = (\del_t)^{\perp}$. If $X \in T_{\gamma(\tau)} \del M \subset T_{\gamma(\tau)}M$, then
\begin{equation}
	\abs{\mathrm{proj}_V X}^2 = \abs{X}^2 - \abs{\mathrm{proj}_{\del_t} X}^2.
\end{equation}
On the other hand, $\del_t = \cos \theta \del_\nu + \mathrm{proj}_{T\del M} \del_t$, and so $\abs{\mathrm{proj}_{T\del M} \del_t}^2 = \sin^2 \theta$. As $X \in T\del M$, we have $\abs{\mathrm{proj}_{\del_t} X}^2 \leq \abs{\mathrm{proj}_{T\del M} \del_t}^2$, and so $\abs{\mathrm{proj}_V X}^2 \geq \cos^2\theta \abs{X}^2$. As $\im H(\tau) \geq c I$, it follows that
\begin{equation}
	\mathrm{Hess}_{\del M}(\im \phi \vert_{\del M}) \vert_{\gamma(\tau)} \geq c \cos^2 \theta I.
\end{equation}
\end{proof}

By controlling the initial conditions of each $b_j$ at $\gamma(\tau)$, we now want to show that \eqref{eq:main_identity_traces} either leads to a contradiction when $\lambda$ is large or implies that each $a_j$ must vanish to all orders. By Theorem \ref{thm:stationaryphase}, omitting the cutoffs $\chi$ because they don't contribute to the asymptotics, we can expand \eqref{eq:main_identity_traces} as
\begin{align}\label{eq:expanded_stationary}
	&\pi^{\frac{m-1}{2}}\det((D^2_{x} \im \phi \vert_{\del M})\vert_{\gamma(\tau)})^{-\frac{1}{2}} \left(\sum_{\ell = 0}^R \lambda^{-\ell + 1} \sum_{j = 0}^\ell  \sum_{r = 0}^{R - \ell} \lambda^{-r} L_r(\dual{a_j}{\tilde{b}_{\ell - j}}) \right. \\
	& + \left.\quad \sum_{\ell = 0}^{R-1} \lambda^{-\ell} \sum_{j = 0}^\ell  \sum_{r = 0}^{R - \ell} \lambda^{-r} [L_r(\dual{\del_\nu a_j}{\rho b_{\ell - j}}) - L_r(\dual{a_j}{\rho \del_\nu b_{\ell - j}})] \right) = O(\lambda^{-R}).
\end{align}
Indeed, by Theorem \ref{thm:beams_tube}, the $C^{3K - 2\ell}$ norm of $\dual{a_j}{\tilde{b}_{\ell - j}}$ is uniformly bounded for all $0 \leq j \leq \ell$. Therefore, by taking $N$ such that $R - \ell + 1 \leq N \leq R - \ell + 2$, we have $\lambda^{-\ell + 1}\lambda^{-N} \leq \lambda^{-R}$ and
\begin{equation}
	2N + m \leq 2R - 2\ell + 4 + m = K - 2\ell - 10 \leq 3K - 2\ell
\end{equation}
so that the error term in \eqref{eq:stationary_phase} is uniformly bounded. A similar argument shows that the error terms for the second sum are also uniformly bounded. By Lemma \ref{lem:hessian} and Theorem \ref{thm:beams_tube}, $\det((D^2_x \im \phi \vert_{\del M}) \vert_{\gamma(\tau)})$ is uniformly bounded from below by $(c\sin^2 \theta_0)^{m-1}$ and from above by a constant since $\norm{\phi}_{C^2} \leq C$. Therefore, collecting terms of the same order in \eqref{eq:expanded_stationary} yields
\begin{equation}\label{eq:workhorse}
	\sum_{\ell=0}^{R} \lambda^{-\ell} \left( \sum_{i + j + r = \ell} L_r(\dual{a_i}{\tilde{b}_{j}}) + \sum_{i + j + r = \ell-1} [L_r(\dual{\del_\nu a_i}{\rho b_j}) - L_r(\dual{a_i}{\rho \del_\nu b_j})] \right) = O(\lambda^{-R-1}).
\end{equation}
We use the following lemma to deal with the terms in the sum.

\begin{lem}\label{lem:workhorse}
Suppose $a$ vanishes to order $k_1$ and $b$ vanishes to order $k_2$. Let $p = \sum_{\abs{\alpha} = k_1} \frac{(D^\alpha_x a)(0)}{\alpha!} x^\alpha$ and $q = \sum_{\abs{\alpha} = k_2} \frac{(D^\alpha_x b)(0)}{\alpha!} x^\alpha$. If $k_1 + k_2 = 2r$, then
\begin{equation}
	L_r(\dual{a}{b}) = \frac{2^{-r}}{r!} Q^r(\dual{p}{q})
\end{equation}
where $Q = \dual{BD}{D}$ with $B = [2(D^2_x \im \phi \vert_{\del M})\vert_{\gamma(\tau)}]^{-1}$. If $k_1 + k_2 > 2r$, then	$L_r(\dual{a}{b}) = 0$. 
\end{lem}

\begin{proof}
By Theorem \ref{thm:stationaryphase} with $\Phi = 2\I \im \phi$,
\begin{equation}
	L_r(\dual{a}{b}) = \sum_{\mu = 0}^{2r} \frac{2^{-(r+\mu)} \I^{\mu}}{\mu! (r + \mu)!} Q^{r + \mu}(g_0^{\mu} \dual{a}{b}) \vert_{x = 0}
\end{equation}
with $Q$ as above and $g_0$ vanishing to third order at $0$. The operator $Q^{r + \mu}$ is of order $2r + 2\mu$ while $g_0^{\mu}\dual{a}{b}$ vanishes to order $k_1 + k_2 + 3\mu$. Hence, if $k_1 + k_2 = 2r$, then $Q^{r+\mu}(g_0^\mu \dual{a}{b})$ vanishes unless $\mu = 0$. When $\mu = 0$, the terms in the Leibniz rule's expansion of $Q^r(\dual{a}{b})$ vanish unless exactly $k_1$ derivatives hit $a$ and $k_2$ derivatives hit $b$. It follows that $Q^r(\dual{a}{b}) = Q^r(\dual{p}{q})$. If $k_1 + k_2 > 2r$, all the terms vanish.
\end{proof}

We also need the following proposition, whose proof easily follows from the reverse triangle inequality.

\begin{prop}\label{prop:reverse_triangle}
Let $c_\ell \in \C$ be such that $\abs{c_\ell} \leq C'$ and suppose $\abs{\sum_{k = 0}^R c_\ell \lambda^{-\ell}} \leq C\lambda^{-R-1}$ for some constants $C, C' > 0$. Then,
\begin{equation}
	\abs{\sum_{\ell=0}^{k} c_\ell \lambda^{-\ell}} \leq (C + RC') \lambda^{-k-1}
\end{equation}
for all $0 \leq k \leq R$.
\end{prop}

Using Proposition \ref{prop:reverse_triangle}, Lemma \ref{lem:workhorse} and Lemma \ref{lem:zeta_inner_product}, we will now show by induction that equation \eqref{eq:workhorse} either leads to a contradiction when $\lambda$ is large, or forces the coefficients $a_i$ to vanish to all orders.

We first consider $a_0$. By Theorem \ref{thm:beams_tube}, every term in \eqref{eq:workhorse} can be uniformly bounded. Hence, Proposition \ref{prop:reverse_triangle} applied to \eqref{eq:workhorse} with $k = 0$ yields
\begin{equation}\label{eq:a0}
	L_0(\dual{a_0}{b_0}) = O(\lambda^{-1}).
\end{equation}
Near $\gamma(\tau) = 0$, we have $\chi^2 \equiv 1$, $\re(\del_\nu \phi) = \cos\theta + O(\abs{x})$, and $\rho = 1 + O(\abs{x}^2)$. Hence, $|\tilde{b}_0(0)|$ is bounded from below by $\frac{2\abs{b_0(0)}}{\cos \theta} \geq \frac{2\abs{b_0(0)}}{\sin \theta_0}$. Since \eqref{eq:a0} holds for all $b_0$ with $\abs{b_0(0)} \leq C_0$, we get $\abs{a_0(0)} \leq C \lambda^{-1}$. Moreover, the constant is uniform with respect to $C_0$ and $\gamma \in \mathcal{G}_{T, \theta_0, \varepsilon}$. Therefore, we have
\begin{equation}
	S_0 := \sup_{\substack{\gamma \in \mathcal{G} \\ \mathsf{f} \in \mathcal{A}(C_0)}} \abs{a_0(\gamma(\tau))} \leq C \lambda^{-1}
\end{equation}
where the supremum runs over all geodesics $\gamma \in \mathcal{G}_{T, \theta_0, \varepsilon}$ and initial conditions $\mathsf{f} \in \mathcal{A}_{\del M, \gamma(0)}^K(C_0)$ based at $\gamma(0)$. As $A_1$ and $A_2$ are fixed connections and the differential equations defining $a_0$ do not depend on $\lambda$, $S_0$ is independent of $\lambda$. If $S_0 \neq 0$, then taking $\lambda > C/S_0$ yields a contradiction. Therefore, if $\lambda > C/S_0$, then $\DN_\lambda^{A_1}$ and $\DN_\lambda^{A_2}$ cannot coincide. And if $S_0 = 0$, then $a_0(\gamma(\tau)) = 0$ for all $\gamma \in \mathcal{G}_{T, \theta_0, \varepsilon}$ and $\mathsf{f} \in \mathcal{A}^{K}_{\del M, \gamma(0)}$ by linearity. This is equivalent to $\mathcal{P}_{\gamma}^{A_1, 0} = \mathcal{P}_\gamma^{A_2, 0}$ for all $\gamma \in \mathcal{G}_{T, \theta_0, \varepsilon}$.

Now for the inductive step, let us suppose that $a_i$ vanishes to order $(2k - 2i)_+$ at $\gamma(\tau)$, where $(x)_+ := \max\{x, 0\}$. We choose initial conditions $\mathsf{h} \in \mathcal{A}_{\del M, \gamma(\tau)}^K(C_0)$ such that $b_j$ vanishes to order $(2k - 2j)_+$ at $\gamma(\tau)$. Proposition \ref{prop:reverse_triangle} applied to \eqref{eq:workhorse} gives
\begin{equation}\label{eq:workhorse_2k}
	\sum_{\ell=0}^{2k} \lambda^{-\ell} \left( \sum_{i + j + r = \ell} L_r(\dual{a_i}{\tilde{b}_{j}}) + \sum_{i + j + r = \ell-1} [L_r(\dual{\del_\nu a_i}{\rho b_j}) - L_r(\dual{a_i}{\rho \del_\nu b_j})] \right) = O(\lambda^{-2k-1}).
\end{equation}
Let us look at the second inner sum first. By boundary determination with a known metric \cite[Theorem 3.4]{cekicYM} and the normalisation \eqref{eq:gauge_del_nu}, $A_1$ and $A_2$ must agree to all orders on $\del M$. Since $a^{(1)}$ and $a^{(2)}$ are solutions to linear differential equations that differ only through $A_1$ and $A_2$, $a = a^{(1)} - a^{(2)}$ satisfies the equations for the amplitude of a Gaussian beam for $A_2$ (or equivalently, $A_1$) at $\gamma(\tau)$. It follows that since $a_i$ vanishes to order $(2k - 2i)_+$, $\del_\nu a_i$ must vanish to order $(2k - 2i - 1)_+$. Similarly, $\del_\nu b_j$ vanishes to order $(2k - 2j - 1)_+$. Therefore, $\dual{\del_\nu a_i}{\rho b_j}$ and $\dual{a_i}{\rho \del_\nu b_j}$ vanish to order at least
\begin{equation}
	(2k - 2i) + (2k - 2j) - 1 = 2(2k - 1 - i - j) + 1 \geq 2(\ell - 1 - i - j) + 1 = 2r + 1 > 2r.
\end{equation}
By Lemma \ref{lem:workhorse}, all the terms in the second inner sum vanish. Let us now look at the first inner sum. The term $\dual{a_i}{\tilde{b}_j}$ vanishes to order
\begin{equation}
	(2k - 2i)_+ + (2k - 2j)_+ \geq 2(2k - i - j) \geq 2(\ell - i - j) = 2r.
\end{equation}
The first inequality is strict if $i > k$ or $j > k$; the second if $\ell < 2k$. By Lemma \ref{lem:workhorse}, $L_r(\dual{a_i}{\tilde{b}_j})$ vanishes in either case. Therefore, \eqref{eq:workhorse_2k} simplifies to
\begin{equation}\label{eq:workhorse_2k_simplified}
	\sum_{\substack{i + j + r = 2k \\ 0 \leq i, j \leq k}} \frac{2^{-r}}{r!} Q^r(\dual{p_i}{q_j}) = O(\lambda^{-1})
\end{equation}
where $p_i$ and $q_j$ are the homogeneous polynomials with coefficients in $\C^n$
\begin{equation}
	p_i = \sum_{\abs{\alpha} = 2k - 2i} \frac{D^\alpha a_i(0)}{\alpha!} x^\alpha, \qquad q_j = \sum_{\abs{\alpha} = 2k - 2j} \frac{D^\alpha \tilde{b}_j(0)}{\alpha!} x^\alpha.
\end{equation}
We let $p_{i, \ell}$ and $q_{j, \ell}$ denote the $\ell$th component of $p_i$ and $q_j$ respectively so that $\dual{p_i}{q_j} = \sum_{\ell = 1}^n \overline{p}_{i, \ell} q_{j, \ell}$. In the notation of Section \ref{sec:polynomials}, we can write $Q^r(\overline{p}_{i, \ell} q_{j, \ell}) = \Delta^r(Z^* \overline{p}_{i, \ell} Z^* q_{j, \ell})$ and rewrite \eqref{eq:workhorse_2k_simplified} as
\begin{equation}\label{eq:zeta_Cn}
	\sum_{\ell = 1}^n \zeta_{2k}((Z^*p_{0, \ell}, \dots, Z^*p_{k, \ell}), (Z^* q_{0, \ell}, \dots, Z^* q_{k, \ell})) = O(\lambda^{-1}).
\end{equation}
Since $D^\alpha \tilde{b}_j(0) = (2\I\cos\theta) D^{\alpha}b_j(0)$ for $\abs{\alpha} = 2k - 2j$, we have $\abs{D^\alpha q_{j, \ell}(0)} \leq 2C_0 \cos\theta$ for all $\abs{\alpha} = 2k - 2j$, $1 \leq \ell \leq n$. By \eqref{eq:x_to_z}, $\abs{D^\alpha (Z^* q_{j, \ell})(0)} \leq CC_0 \cos\theta \omega_{\max}^{2k-2j}$. We can bound $\omega_{\max}$ from below by $c \cos^2 \theta$ through Lemma \ref{lem:hessian}, and we can bound $\cos\theta$ from below by $\sin\theta_0$. By Lemma \ref{lem:zeta_inner_product}, we can see the left-hand side of \eqref{eq:zeta_Cn} as an inner product on $(\mathcal{H}^{2k}_{\oplus})^n$. Hence, by Lemma \ref{lem:dual_estimate}, there is a constant $C$, uniform in $C_0$ and $\gamma \in \mathcal{G}_{T, \theta_0, \varepsilon}$ such that
\begin{equation}
	\abs{D^\alpha (Z^* p_{i, \ell})(0)} \leq C \lambda^{-1}
\end{equation}
for all $\abs{\alpha} = 2k - 2i$, $1 \leq \ell \leq n$. By \eqref{eq:z_to_x}, $\abs{D^\alpha p_{i, \ell}(0)} \leq C\omega_{\min}^{-(2k-2i)} \lambda^{-1}$ for all $\abs{\alpha} = 2k - 2i$, $1 \leq \ell \leq n$. By Lemma \ref{lem:hessian}, $\omega_{\min}$ is bounded from below uniformly by $c \sin^2 \theta_0$. Therefore, as the derivatives of order $2k - 2i$ of $p_i$ agree with those of $a_i$, we have
\begin{equation}
	S_{2k} := \sup_{\substack{\gamma \in \mathcal{G} \\ \mathsf{f} \in \mathcal{A}(C_0)}} \sup_{\abs{\alpha} = 2k - 2i} \abs{\del_x^\alpha a_i(\gamma(\tau))} \leq C\lambda^{-1}
\end{equation}
where, as with $S_0$, the supremum is taken over all geodesics $\gamma \in \mathcal{G}_{T, \theta_0, \varepsilon}$ and initial conditions $\mathsf{f} \in \mathcal{A}^{K}_{\del M, \gamma(0)}(C_0)$. By the same logic as with $S_0$, if $S_{2k} \neq 0$, then taking $\lambda > C/S_{2k}$ leads to a contradiction and hence $\DN_\lambda^{A_1} \neq \DN_\lambda^{A_2}$. And if $S_{2k} = 0$, then $a_i$ vanishes to order $(2k + 1 - 2i)_+$ at $\gamma(\tau)$ for all $\gamma \in \mathcal{G}_{T, \theta_0, \varepsilon}$.

By proceeding similarly with $2k + 1$ instead of $2k$, that is, using Proposition \ref{prop:reverse_triangle} on \eqref{eq:workhorse}, simplifying the sum with the help of Lemma \ref{lem:workhorse} to simplify the sum, expressing it with respect to $\zeta_{2k+1}$, and using an analogous version of Lemma \ref{lem:dual_estimate} for $\zeta_{2k+1}$, we can show that if $a_i$ vanishes to order $(2k + 1 - 2i)_+$, then there is a constant $C > 0$ such that
\begin{equation}
	S_{2k+1} := \sup_{\substack{\gamma \in \mathcal{G} \\ \mathsf{f} \in \mathcal{A}(C_0)}} \sup_{\abs{\alpha} = 2k + 1 - 2i} \abs{\del^\alpha_x a_i (\gamma(\tau))} \leq C\lambda^{-1}.
\end{equation}
As above, either $S^{2k+1} = 0$ or there is $\lambda_0$ such that $\DN^{A_1}_\lambda \neq \DN^{A_2}_\lambda$ whenever $\lambda > \lambda_0$. If $S^{2k+1} = 0$, then $a_i$ vanishes to order $(2k + 2 - 2i)_+$.

By induction, it follows that either there is $\lambda_0$ such that $\DN_\lambda^{A_1} \neq \DN_\lambda^{A_2}$ whenever $\lambda > \lambda_0$, or the coefficients $a_i$ vanish to order $(R + 1 - 2i)_+$ at $\gamma(\tau)$. By the definition of the Gaussian beam transport map, $\mathcal{P}^{A_1, k}_\gamma = \mathcal{P}^{A_2, k}_\gamma$ whenever the $a_i$ vanish at $\gamma(\tau)$ to order $3k + 1 - 2i$ for $0 \leq i \leq k$. Hence, if the $a_i$ vanish to order $(R + 1 - 2i)_+$ at $\gamma(\tau)$ and $R \geq 3k$, then $\mathcal{P}^{A_1, k}_\gamma = \mathcal{P}^{A_2, k}_\gamma$. The result follows by taking $K$ (and hence $R$) large enough.

\subsection{Proof of Theorem \ref{thm:calderon}}

\begin{proof}[Proof of Theorem \ref{thm:calderon}]
Let us first assume that (ii) holds, that is, $m \geq 3$, $\del M$ is strictly convex, $M$ admits a strictly convex function, and geodesics on $M$ do not self-intersect on the boundary. For small $r > 0$, consider the manifolds
\begin{equation}
	M_r = \{x \in M: d_g(x, \del M) \geq r\}
\end{equation} 
Let $U = M \setminus M_\varepsilon^{\intr}$ and choose $\varepsilon$ such that both connections agree on $U$. Since $\del M$ is strictly convex, we can also choose $\varepsilon$ such that the boundaries $\del M_r$, $0 < r \leq \varepsilon$ form a strictly convex foliation of $U$ with related strictly convex function $x \mapsto d_g(x, \del M)$.

Let $(x,v) \in \del_+ SM = \{(x,v) \in TM: x \in \del M, \abs{v}_g = 1, g(v, \nu_x) < 0\}$. By the convex foliation of $U$, the geodesic $\gamma_{x,v} : [0, \tau(x,v)] \to M$ can intersect $\del M_\varepsilon$ at most twice. If it intersects $\del M_\varepsilon$ twice at times $t_1 < t_2$, then by the foliation we see that $d_g(\gamma_{x,v}(t), \del M)$ is strictly increasing on an interval containing $[0, t_1]$ and strictly decreasing on an interval containing $[t_2, \tau(x,v)]$. Hence, $\gamma_{x,v}$ is $\varepsilon$-separated from the boundary since $d_g(\gamma_{x,v}(t_1), \del M) = d_g(\gamma_{x,v}(t_2), \del M) = \varepsilon$. However, if $\gamma_{x,v}$ does not intersect the boundary or intersects it once, then it is not $\varepsilon$-separated but is entirely contained in $U$.

By compactness and the strictly convex foliation, there is $\theta_0$ such that every geodesic $\gamma_{x, v}$ with $(x,v) \in \del_0 SM_{\varepsilon} = \{(x,v) \in T \del M_\varepsilon : \abs{v}_g = 1\}$ intersects $\del M$ at an angle bounded from below by $\theta_0$. Therefore, if $(x,v) \in \del_+ SM$ with $\abs{g_x(\nu_x, v)} < \sin \theta_0$, then $\gamma_{x,v}(t) \in U$ for all $t \in [0, \tau(x,v)]$. Without changing $U$, we can shrink $\varepsilon$ so that $d_g(\gamma_{x,v}(0), \gamma_{x,v}(\tau(x,v))) > 2\varepsilon$ for all $(x,v) \in \del_+ SM$ with $\abs{g_x(\nu_x, v)} \geq \sin \theta_0$ because geodesics do not self-intersect on the boundary and the length of such geodesics is bounded away from $0$.

By \cite[Lemma 2.1]{matrixweights}, we know that $M$ is nontrapping as it admits a strictly convex function. Hence, there is $T > 0$ such that all geodesics in $M$ have length at most $T$. Now consider $\mathcal{G}_{T, \theta_0, \varepsilon}$ with the parameters chosen as above. By our choices, we see that if $\gamma$ is a geodesic that is not in $\mathcal{G}_{T, \theta_0, \varepsilon}$, then it is entirely contained in $U$, where $A_1 = A_2$. Therefore, $P^{A_1}_\gamma = P^{A_2}_\gamma$ for all such geodesics.

By Theorem \ref{thm:GBtraces}, unless $\mathcal{P}^{A_1, 0}_\gamma = \mathcal{P}^{A_2, 0}_\gamma$ (and hence $P^{A_1}_\gamma = P^{A_2}_\gamma$) for all $\gamma \in \mathcal{G}_{T, \theta_0, \varepsilon}$ there exists $\lambda_0 > 0$ such that $\DN_\lambda^{A_1} \neq \DN_\lambda^{A_2}$ for all $\lambda \in (\lambda_0, \infty) \setminus J$. But if $P^{A_1}_\gamma = P^{A_2}_\gamma$ for all $\gamma \in \mathcal{G}_{T, \theta_0, \varepsilon}$, then the non-abelian X-ray transforms of $A_1$ and $A_2$ agree. By \cite[Theorem 1.1]{matrixweights}, the connections must then be gauge equivalent, that is, there is $\varphi : M \to U(n)$ with $\varphi \vert_{\del M} = I$ such that $A_2 = A_1 \triangleleft \varphi$.

The case where (i) holds, that is, $m = 2$ and $M$ is simple, is proven similarly but by using the injectivity result from \cite[Theorem 1.1]{paternain2022nonabelian} instead.
\end{proof}

\section{Broken non-abelian X-ray transform}\label{sec:broken}

\subsection{Broken geodesics and scattering map}

We define what we mean by a broken geodesic and give a way to parametrise them. In what follows, $SM$ is the unit tangent bundle of $M$ with fibre $S_x M = \{v \in T_x M: \abs{v}_g = 1\}$.
\begin{defi}\label{def:brokengeodesic}
We say that a continuous curve $\Gamma:[0, \tau] \to M$ is a broken geodesic if its endpoints lie on $\del M$, its restriction to $(0, \tau)$ is contained in $M^{\intr}$, and there is a time $t^* \in (0,\tau)$ such that $\Gamma\vert_{[0, t^*]}$ and $\Gamma\vert_{[t^*, \tau]}$ are smooth geodesics intersecting transversally at $\Gamma(t^*)$.
\end{defi}

To avoid confusion, we will denote broken geodesics as $\Gamma$ and smooth geodesics as $\gamma$. We can identify a broken geodesic $\Gamma$ with its breakpoint $x = \Gamma(t^*) \in M^{\intr}$ and the two directions $v, w \in S_x M$ such that $\Gamma\vert_{[0, t^*]}(t) = \gamma_{x, v}(t^* - t)$ and $\Gamma\vert_{[t^*, \tau]}(t) = \gamma_{x, w}(t - t^*)$. This leads us to consider the fibred product of bundles
\begin{equation}
	\mathcal{F} = SM \times_M SM = \{((x, v), (y, w)) \in SM \times SM : x = y\}.
\end{equation}
The set $\mathcal{F}$ is a fibre bundle over $M$. The fibre $\mathcal{F}_x$ over $x \in M$ is given by $S_x M \times S_x M$ and hence we write points in $\mathcal{F}$ as $(x, v, w)$.

By the above description, any broken geodesic can be associated with a unique triple $(x, v, w) \in \mathcal{F}$ with $x \in M^{\intr}$. Conversely, a triple $(x, v, w) \in \mathcal{F}$ with $x \in M^{\intr}$ corresponds to a broken geodesic if $\tau(x,v) + \tau(x,w) < \infty$ where $\tau : SM \to [0, \infty]$ is the exit time.

Given a connection $A$ on $E = M \times \C^n$ and a broken geodesic $\Gamma$, we define its scattering data as
\begin{equation}
	S^A_\Gamma = P^A_{\Gamma[t^*, \tau]}P^A_{\Gamma[0, t^*]}.
\end{equation}
The broken non-abelian X-ray transform is the map $A \mapsto S^A$ where $S^A$ assigns broken geodesics $\Gamma$ to their scattering data $S^A_\Gamma$.

\subsection{Geodesic graph structures and injectivity}

For $x \in M^{\intr}$, let $\mathcal{V}(x)$ be a subset of nontangential geodesics that start at $x$, that is, the set of geodesics $\gamma : [0, \tau] \to M$ such that $\gamma(0) = x$, $\gamma([0, \tau)) \subset M^{\intr}$, and $\gamma(\tau) \in \del M$ with $\dot{\gamma}(\tau) \not \in T\del M$. We can identify each such geodesic $\gamma$ with $(x, \dot{\gamma}(0)) \in S_x M$ and therefore see $\mathcal{V}(x)$ as a subset of $S_x M^{\intr}$. Now let
\begin{equation}
	\mathcal{V} = \bigcup_{x \in M^{\intr}} \mathcal{V}(x) \subset SM^{\intr}.
\end{equation}
We define the scattering map $\alpha : \mathcal{V} \to \del M$ as $\alpha(x,v) = \gamma_{x,v}(\tau(x,v))$. For $x \in M^{\intr}$, consider subsets $\mathcal{E}(x) \subset \{((x,v), (x,w)) \in \mathcal{V}(x) \times \mathcal{V}(x) : v \neq \pm w\} \subset \mathcal{F}_x$ and let $\mathcal{E} = \bigcup_{x \in M^{\intr}} \mathcal{E}(x)$. The broken non-abelian X-ray transform of a connection $A$ can be seen as a function on $\mathcal{E} \subset \mathcal{F}$ by setting
\begin{equation}
	S^A(v_x, w_x) = S^A(x, v, w) = P^A_{\gamma_{x, w}} (P^A_{\gamma_{x, v}})^{-1}.
\end{equation}

The pairs $(\mathcal{V}(x), \mathcal{E}(x))$ each have the structure of a directed graph. The graph $(\mathcal{V}(x), \mathcal{E}(x))$ is connected if every distinct vertices $v, w \in \mathcal{V}(x)$ are connected by a path. In other words, there are vertices $v = v_0, v_1, \dots, v_{k-1}, v_k = w$ such that $(v_{j-1}, v_j) \in \mathcal{E}(x)$ for all $1 \leq j \leq k$. We call $(\mathcal{V}, \mathcal{E})$ a geodesic graph structure on $(M, g)$. We say that such a structure is complete if $\alpha(\mathcal{V})$ is dense in $\del M$, and if the following conditions hold for every $x \in M^{\intr}$.
\begin{enumerate}
	\item[(i)] The directed graph $(\mathcal{V}(x), \mathcal{E}(x))$ is connected;
	\item[(ii)] There is an embedded submanifold $U_x \subset SM$ such that $U_x \subset \mathcal{V}$ and $x \in (\pi_{SM}(U_x))^{\intr}$ where $\pi_{SM} : SM \to M$ is the natural projection;
	\item[(iii)] There are directions $v_1, \dots, v_m \in \mathcal{V}(x) \subset S_x M$ that span $T_x M$ such that $\phi_t(x, v_i) \in \mathcal{V}$ for $t$ in a neighbourhood of $0$ where $\phi_t : SM \to SM$ is the geodesic flow.
\end{enumerate}
Note that if $\pi_{SM}(\mathcal{V}^{\intr}) = M^{\intr}$, then (ii) and (iii) hold.

\begin{thm}\label{thm:injectivity_broken}
	Let $A_1$ and $A_2$ be smooth Hermitian connections on $E = M \times \C^n$. Let $(\mathcal{V}, \mathcal{E})$ be a complete geodesic graph structure on $(M, g)$. Then, $S^{A_1} = S^{A_2}$ on $\mathcal{E}$ if and only if $A_1$ and $A_2$ are gauge equivalent.
\end{thm}

\begin{proof}
Consider the map $\tilde{\varphi}: \mathcal{V} \to \mathrm{U}(n)$ given by
\begin{equation}
	\tilde{\varphi}(v_x) = (P^{A_1}_{\gamma_{x,v}})^{-1} P^{A_2}_{\gamma_{x,v}}.
\end{equation}
We claim that $\tilde{\varphi}$ only depends on $x$. If $w_x \in \mathcal{V}$ is such that $(v_x, w_x) \in \mathcal{E}(x)$, then $S^{A_1}(v_x, w_x) = S^{A_2}(v_x, w_x)$ if and only if
\begin{equation}
	P^{A_1}_{\gamma_{x,w}} (P^{A_1}_{\gamma_{x,v}})^{-1} = P^{A_2}_{\gamma_{x,w}} (P^{A_2}_{\gamma_{x,v}})^{-1}
\end{equation}
which can be rearranged as $\tilde{\varphi}(v_x) = \tilde{\varphi}(w_x)$. It follows that $\tilde{\varphi}(v_x) = \tilde{\varphi}(w_x)$ whenever $(v_x, w_x) \in \mathcal{E}(x)$. As the graph $(\mathcal{V}(x), \mathcal{E}(x))$ is connected, $\tilde{\varphi}(v_x) = \tilde{\varphi}(w_x)$ for all $v_x, w_x \in \mathcal{V}(x)$, and therefore $\tilde{\varphi}$ only depends on $x$, as claimed. Hence, since $\pi_{SM}(\mathcal{V}) = M^{\intr}$ by (ii), there is $\varphi : M^{\intr} \to \mathrm{U}(n)$ such that $\tilde{\varphi} = \varphi \circ \pi_{SM}$. Moreover, for every $x \in M^{\intr}$, by smooth dependence of ODEs on initial conditions, the function $\tilde{\varphi}$ is smooth on the embedded submanifold $U_x$ as in (ii). Since $\tilde{\varphi} = \varphi \circ \pi_{SM}$ and $x \in (\pi_{SM}(U_x))^{\intr}$, it follows that $\varphi$ is smooth at $x$ and hence smooth on $M^{\intr}$.

Let $e = e(A_1, A_2) \in \Omega^1(M, \End(\C^{n \times n}))$ be given by
\begin{equation}
	e(A_1, A_2)Q = A_1Q - QA_2, \quad Q \in \C^{n \times n}.
\end{equation}
We can see $e$ as a connection on the trivial bundle $M \times \C^{n \times n}$. For any $Q \in \C^{n \times n}$ and any smooth curve $\gamma$, we have $P^e_\gamma Q = P^{A_1}_\gamma Q (P^{A_2}_\gamma)^{-1}$ so that $\varphi(x) = (P^e_{\gamma_{x,v}})^{-1} I$ for any $v_x \in \mathcal{V}(x)$. We let $\eta_{x,v} : [0,\tau(x,v)] \to M$ be the curve $\gamma_{x,v}$ travelled backwards, that is, $\eta_{x,v}(t) = \gamma_{x,v}(\tau(x,v) - t)$. Then, $\varphi(x) = P^e_{\eta_{x,v}} I$.

Let $v_i \in \mathcal{V}(x)$ be as in (iii). By definition of the parallel transport, $(d_e P^e_{\eta_{x,v_i}})_x(-v_i) = 0$ where $d_e = d + e(A_1, A_2)$, and so $(d_e \varphi)_x(v_i) = 0$ since $\phi_t(x, v_i) \in \mathcal{V}$ for small $t$. As the vectors $v_i$ span $T_x M$, we must have $(d_e \varphi)_x = 0$. It follows that $d_e \varphi = 0$ on $M^{\intr}$. By expanding
\begin{equation}
	0 = d_e \varphi = d\varphi + A_1 \varphi - \varphi A_2,
\end{equation}
we see that this is equivalent to $A_2 = A_1 \triangleleft \varphi$ on $M^{\intr}$. Moreover, since $d\varphi = \varphi A_2 - A_1 \varphi$, $\varphi$ is bounded as it takes values in $\mathrm{U}(n)$, and both connections are smooth on $M$, we see that $\varphi$ is uniformly continuous on $M^{\intr}$ and can therefore be extended uniquely to $\del M$ with $A_2 = A_1 \triangleleft \varphi$ on $M$.

It remains to show $\varphi \vert_{\del M} = I$. For $x \in M^{\intr}$ and a geodesic $\gamma_{x,v} : [0, \tau] \to M$ with $v_x \in \mathcal{V}(x)$, we have
\begin{equation}
	\varphi(x) = \lim_{t \nearrow \tau} (P^{A_1}_{\gamma_{x,v}})^{-1} P^{A_2}_{\gamma_{x,v}[0,t]} = (P^{A_1}_{\gamma_{x,v}})^{-1} \lim_{t \nearrow \tau} P^{A_1 \triangleleft \varphi}_{\gamma_{x,v}[0,t]} = (P^{A_1}_{\gamma_{x,v}})^{-1} \left( \lim_{t \nearrow \tau} \varphi^{-1}(\gamma(t)) P^{A_1}_{\gamma_{x,v}[0,t]} \right)\varphi(x).
\end{equation}
Taking the limit and simplifying, we get $P^{A_1}_{\gamma_{x,v}} = \varphi^{-1}(\gamma(\tau)) P^{A_1}_{\gamma_{x,v}}$ from which it follows that $\varphi(\gamma(\tau)) = I$. As $\alpha(\mathcal{V})$ is dense in $\del M$ and $\varphi$ is continuous, $\varphi \vert_{\del M} = I$.
\end{proof}

For example, if $\mathcal{V}(x)$ is the set of all nontangential geodesics starting at $x$ and $\mathcal{E}(x)$ is the set of all distinct pairs of such geodesics, then $\alpha(\mathcal{V}) = \del M$ and $(\mathcal{V}(x), \mathcal{E}(x))$ is connected since it is the complete directed graph on $\mathcal{V}(x)$. Moreover, we can see that $\pi_{SM}(\mathcal{V}^{\intr}) = M^{\intr}$ by considering the shortest geodesic $\gamma_{x, v}$ from any point $x \in M^{\intr}$ to $\del M$. It will always intersect $\del M$ normally, and hence, using the implicit function theorem, there is an open neighbourhood of $(x,v) \in SM$ that is contained in $\mathcal{V}$. Therefore, the broken non-abelian X-ray transform is injective up to gauge when considered on all pairs of broken nontangential geodesics, which shows Theorem \ref{thm:injectivity_broken_intro}.

\subsection{Injectivity up to a sign}

We will need a version of Theorem \ref{thm:injectivity_broken} for Hermitian connections when $S^{A_1}$ and $S^{A_2}$ agree up to a sign. Consider $\mathrm{U}_{\pm}(n) = \mathrm{U}(n)/\{\pm I\}$ the group obtained through the identification $u \sim -u$ for all $u \in \mathrm{U}(n)$. We write the equivalence class of $u \in \mathrm{U}(n)$ in $\mathrm{U}_{\pm}(n)$ as $[u]_\pm$. As the connections $A \triangleleft \varphi$ and $A \triangleleft (-\varphi)$ agree for all $\varphi : M \to \mathrm{U}(n)$, it makes sense to consider the connection $A \triangleleft \psi$ for $\psi : M \to \mathrm{U}_{\pm}(n)$.

\begin{coro}\label{coro:injectivity_sign}
Let $A_1$ and $A_2$ be smooth Hermitian connections on $E = M \times \C^n$. Let $(\mathcal{V}, \mathcal{E})$ be a complete geodesic graph structure on $(M, g)$. Then, $S^{A_1}(v_x, w_x) = \pm S^{A_2}(v_x, w_x)$ for all $(v_x, w_x) \in \mathcal{E}$ if and only if there is $\psi : M \to \mathrm{U}_{\pm}(n)$ such that $A_2 = A_1 \triangleleft \psi$ and $\psi \vert_{\del M} = [I]_\pm$.
\end{coro}

\begin{proof}
The proof follows the proof of Theorem \ref{thm:injectivity_broken} by considering the function $\tilde{\psi} = [\tilde{\varphi}]_\pm$ on $\mathcal{V}$. Since $[S^{A_1}]_\pm = [S^{A_2}]_\pm$ on $\mathcal{E}$, $\tilde{\psi}$ only depends on $x \in M^{\intr}$ and yields $\psi : M^{\intr} \to \mathrm{U}_{\pm}(n)$. The rest of the proof goes through by working locally and picking local lifts for $\psi$.
\end{proof}

We cannot directly deduce from Corollary \ref{coro:injectivity_sign} that $A_1$ and $A_2$ are gauge equivalent as we cannot guarantee the existence of a global lift $\varphi: M \to \mathrm{U}(n)$ such that $\psi = [\varphi]_\pm$. And even if such a lift exists, it does not guarantee that $\varphi \vert_{\del M} = I$. In what follows, we give conditions to remedy this issue.

Let $p \in \del M$ be any point on the boundary of $M$, and consider the fundamental group $\pi_1(M, p)$. As $M$ is compact, van Kampen's Theorem \cite[Theorem 1.20]{hatcher} implies $\pi_1(M, p)$ is finitely generated. Therefore, we can find loops $g_1, \dots, g_k : [0,1] \to M$ based at $p$ such that $[g_1], \dots [g_k]$ generate $\pi_1(M, p)$. Let $q \in \del M$ be a point distinct from $p$, and let $\eta_i : [0,1] \to M$ be paths such that $\eta_i(0) = g_i(\frac{1}{2})$ and $\eta_i(1) = q$. Consider the curves
\begin{equation}
g_i^+(t) = 
\begin{cases}
	g_i(t) & t \in [0, \frac{1}{2}],\\
	\eta_i(2t-1) & t \in [\frac{1}{2}, 1],
\end{cases}
\quad g_i^-(t) =
\begin{cases}
	\eta_i(1 - 2t) & t \in [0, \frac{1}{2}],\\
	g_i(t) & t \in [\frac{1}{2}, 1].
\end{cases}
\end{equation}
Denoting the concatenation of $g_i^+$ and $g_i^-$ as $g_i^+ \cdot g_i^-$ (traversing $g_i^+$ first), we see that $[g_i^+ \cdot g_i^-] = [g_i]$. The following Proposition shows we can decompose the curves $g_i^{\pm}$.

\begin{prop}\label{prop:homotopy}
Let $(M, g)$ be a smooth compact manifold with smooth boundary and let $p, q \in \del M$. Every homotopy class of paths from $p$ to $q$ contains a piecewise smooth curve composed of nontangential geodesics and paths contained in $\del M$.
\end{prop}

\begin{proof}
Let $M'$ be an extension of $M$ obtained by gluing $M$ and the cylinder $[0,1] \times \del M$ along $\del M$. Extend the metric $g$ on $M$ to $M'$ so that the boundary of $M'$ is strictly convex. By \cite[Lemma 7.1]{bohrstability}, there exists a complete extension $(N, g)$ of $M'$ with the property that geodesics that exit $N$ never come back. Fix a homotopy class of paths in $M$ from $p$ to $q$ and consider it as a homotopy class in $N$. By \cite[Proposition 6.25]{leeriemannian}, the homotopy class in $N$ contains a geodesic segment $\gamma \subset N$ that minimises length among all other piecewise smooth curves in that same homotopy class of paths. The restriction of $\gamma$ to $N \setminus M^{\intr}$ is contained in $\del M \times [0,1)$ by the no-return property of $N$ and hence every connected component of $\gamma \cap (N \setminus M^{\intr})$ is path-homotopic to a path on $\del M = \del M \times \{0\}$. Moreover, the restriction of $\gamma \cap M^{\intr}$ is composed of geodesic segments that can always be perturbed into path-homotopic nontangential geodesics. The claim readily follows.
\end{proof}

The extension $M'$ with a strictly convex boundary in the above proof can be obtained as follows. Choose a smooth extension $\tilde{g}$ of $g$ to $M'$ such that $\tilde{g} = ds^2 + \tilde{h}_s$ for some smooth family of metrics $\tilde{h}_s$ on $\del M$ and let $g = ds^2 + h_s$ with
\begin{equation}
	h_s = \chi(s) \tilde{h}_s + (1- \chi(s)) sk
\end{equation}
where $k$ is a metric on $\del M$ and $\chi : [0,1] \to [0,1]$ is smooth with $\chi(s) = 1$ for $s \in [0,1/4]$ and $\chi(s) = 0$ for $s \in [1/2,1]$. The second fundamental form $\mathbb{I}_s = 2 \del_s h_s$ is given by $k$ for $s = 1$ and is hence positive-definite so that the boundary of $M'$ is strictly convex.

Let $D_0, \dots, D_r$ be the boundary components of $M$. Without loss, suppose that $p \in D_0$. Let $\alpha_1, \dots, \alpha_r : [0,1] \to M$ be piecewise smooth curves composed of nontangential geodesics and paths contained in $\del M$ such that $\alpha_j(0) \in D_i$ and $\alpha_j(1) = p$ for $1 \leq j \leq r$. By Proposition \ref{prop:homotopy}, we can find a collection of nontangential geodesics and paths contained in $\del M$ such that their concatenation is path-homotopic to $g_i^+$ and $g_i^-$ for $0 \leq i \leq k$. Let $\gamma_1, \dots, \gamma_L$ be the set of all nontangential geodesics that are used in the decompositions of the curves $\alpha_j$ and the curves path-homotopic to $g^{\pm}_i$. We say that a piecewise smooth curve is $\gamma$-generated if it can be decomposed as a concatenation of the geodesics $\gamma_i$ and curves contained in $\del M$.

\begin{thm}\label{thm:injectivity_sign}
Let $A_1$ and $A_2$ be smooth Hermitian connections on $E = M \times \C^n$. Let $(\mathcal{V}, \mathcal{E})$ be a complete geodesic graph structure on $(M, g)$ and suppose that $P^{A_1}_{\gamma_i} = P^{A_2}_{\gamma_i}$ for all $0 \leq i \leq \ell$ with the curves $\gamma_i$ as described above. If $S^{A_1}(v_x, w_x) = \pm S^{A_2}(v_x, w_x)$ for all $(v_x, w_x) \in \mathcal{E}$, then $A_1$ and $A_2$ are gauge equivalent.
\end{thm}

\begin{proof}
First, by Corollary \ref{coro:injectivity_sign}, there is $\psi : M \to \mathrm{U}_{\pm}(n)$ such that $A_2 = A_1 \triangleleft \psi$. To get a lift for $\psi$, we show by way of contradiction that there cannot be $(v_x, w_x) \in \mathcal{E}$ such that $S^{A_1}(v_x, w_x) = -S^{A_2}(v_x, w_x)$.

Suppose there was such a pair $(v_x, w_x)$, and consider the broken geodesic $\Gamma : [0, \tau] \to M$ associated with this pair. Using the curves $\alpha_j$, we can construct a $\gamma$-generated curve $\beta_0 : [0,1] \to M$ such that $\beta_0(0) = p$ and $\beta_0(1) = \Gamma(0)$. Similarly, we can construct a $\gamma$-generated curve $\beta_\tau : [0,1] \to M$ such that $\beta_\tau(0) = \Gamma(\tau)$ and $\beta_\tau(1) = p$. Consider the curve $\beta_0 \cdot \Gamma \cdot \beta_\tau$. It is a loop based at $p$ and therefore, as the loops $g_i$ generate the group $\pi_1(M, p)$, there is a $\gamma$-generated curve $\zeta$ such that the loop $\ell = \beta_0 \cdot \Gamma \cdot \beta_\tau \cdot \zeta$ is contractible. The induced map in homotopy $\iota_* : \pi_1(\ell) \to \pi_1(M)$ is therefore trivial, where $\iota : \ell \to M$ is the inclusion map. The map $(\psi \circ \iota)_* : \pi_1(\ell) \to \pi_1(\mathrm{U}_{\pm}(n))$ is therefore also trivial. By the lifting criterion \cite[Proposition 1.33]{hatcher}, there exists $\varphi : \ell \to \mathrm{U}(n)$ such that $[\varphi] = \psi \circ \iota$.

As $\Gamma(0)$ and $\Gamma(\tau)$ lie on $\del M$ and $\psi\vert_{\del M} = [I]$, $\varphi(\Gamma(\tau))$ and $\varphi(\Gamma(0))$ must be equal to $I$ or $-I$. As $A_2 = A_1 \triangleleft \varphi$ on $\ell$ and $S^{A_1}_\Gamma = - S^{A_2}_\Gamma$, we have
\begin{equation}
	S^{A_1}_\Gamma = - S^{A_2}_\Gamma = - S^{A_1 \triangleleft \varphi}_\Gamma = - \varphi^{-1}(\Gamma(\tau)) S^{A_1}_\Gamma \varphi(\Gamma(0))
\end{equation}
and therefore we must have $\varphi(\Gamma(\tau)) = -\varphi(\Gamma(0))$. 

We can also consider the curve $\beta_0 \cdot \zeta \cdot \beta_\tau^{-1}$ where $\beta_\tau^{-1}$ is the curve $\beta_\tau$ traversed in the opposite direction. The curve $\beta_0 \cdot \zeta \cdot \beta_\tau^{-1}$ joins $\Gamma(0)$ to $\Gamma(\tau)$ and it is $\gamma$-generated since $\beta_0$, $\beta_\tau$ and $\zeta$ are $\gamma$-generated. For each $\gamma_i : [0, \tau_i] \to M$ that is contained in $\ell$, we have
\begin{equation}
	P^{A_1}_{\gamma_i} = P^{A_2}_{\gamma_i} = P^{A_1 \triangleleft \varphi}_{\gamma_i} = \varphi^{-1}(\gamma_i(\tau_i)) P^{A_1}_{\gamma_i} \varphi(\gamma_i(0)).
\end{equation}
Therefore, $\varphi(\gamma_i(0)) = \varphi(\gamma_i(\tau_i))$. Since $\varphi$ is constant on segments of $\ell$ that are contained on the boundary and $\beta_0 \cdot \zeta \cdot \beta_\tau^{-1}$ is $\gamma$-generated, we see that $\varphi(\Gamma(\tau)) = \varphi(\Gamma(0))$.

The last two paragraphs contradict each other. Therefore, $S^{A_1}(v_x, w_x) = S^{A_2}(v_x, w_x)$ for all $(v_x, w_x) \in \mathcal{E}$, and hence the connections are gauge equivalent by Theorem \ref{thm:injectivity_broken}.
\end{proof}

\section{Calderón problem with a cubic nonlinearity at high but fixed frequency}\label{sec:calderon_cubic}

\subsection{Third order linearisation}\label{sec:thirdorderlin}

Let $f_\varepsilon = \varepsilon_1 f_1 + \varepsilon_2 f_2 + \varepsilon_3 f_3$ for some $f_j \in C^\infty(\del M; \C^n)$ and small $\varepsilon_j \in \R$, $j = 1,2,3$. Let $u_\varepsilon$ be the solution of \eqref{eq:nonlinear_equations} corresponding to $f_\varepsilon$, that is, $u_\varepsilon$ satisfies
\begin{equation}\label{eq:extensioneps}
\begin{cases}
	(\Delta_A - \lambda^2) u_{\varepsilon} + \abs{u_\varepsilon}^2 u_{\varepsilon} = 0, \\
	u_\varepsilon \vert_{\del M} = f_\varepsilon.
\end{cases}
\end{equation}
Note that $u_\varepsilon = 0$ when $\varepsilon = 0$ (meaning $\varepsilon_1 = \varepsilon_2 = \varepsilon_3 = 0$). We let $u_j = \del_{\varepsilon_j}\vert_{\varepsilon = 0} u_{\varepsilon}$ for $j = 1, 2, 3$. By differentiating everything with respect to $\varepsilon_j$ in the previous equation and evaluating at $\varepsilon = 0$, we see that $u_j$ solves
\begin{equation}\label{eq:linearisation_equations}
\begin{cases}
	(\Delta_A - \lambda^2) u_j = 0, \\
	u_j \vert_{\del M} = f_j.
\end{cases}
\end{equation}
We are interested in the third-order linearisation of $\Lambda_A^{\lambda}$. We denote $w = \del_{\varepsilon_1} \del_{\varepsilon_2} \del_{\varepsilon_3} \vert_{\varepsilon = 0} u_\varepsilon$. Then,
\begin{equation}
	(D^3\Lambda_A^\lambda)_0(f_1, f_2, f_3) = \nabla_\nu w \vert_{\del M}.
\end{equation}

\begin{lem}\label{lem:thirdlin}
Let $f \in C^\infty(\del M, \C^n)$ and let $v_f$ solve $(\Delta_A - \lambda^2)v_f = 0$ in $M$ with $v_f \vert_{\del M} = f$. Then,
\begin{equation}\label{eq:thirdlinidentity}
	(\nabla_\nu w, f)_{L^2(\del M; \C^n)} = \sum_{\sigma \in S_3} \int_M \dual{u_{\sigma(1)}}{u_{\sigma(2)}} \dual{u_{\sigma(3)}}{v_f} \de V_g
\end{equation}
where the sum is over all the permutations $\sigma$ of $\{1, 2, 3\}$. 
\end{lem}

\begin{proof}
We start by applying $\del_{\varepsilon_1} \del_{\varepsilon_2} \del_{\varepsilon_3} \vert_{\varepsilon = 0}$ to \eqref{eq:extensioneps}. This yields
\begin{equation}
\begin{cases}
	\Delta_A w - \lambda^2 w + \tilde{w} = 0, \\
	w \vert_{\del M} = 0,
\end{cases}
\end{equation}
where $\tilde{w} = \del_{\varepsilon_1} \del_{\varepsilon_2} \del_{\varepsilon_3} \vert_{\varepsilon = 0}(\abs{u_\varepsilon}^2 u_\varepsilon)$. Let us compute $\tilde{w}$. Since $u_\varepsilon$ takes values in $\C^n$, we can write it as $u_\varepsilon = (u^1, \dots, u^n)$ where $u^j : M \to \C$ for $j = 1, \dots, n$. The $j$-th component of $\abs{u_\varepsilon}^2 u_\varepsilon$ is $\sum_{i=1}^n \overline{u^i} u^i u^j$. We have
\begin{equation}
	\del_{\varepsilon_1} \del_{\varepsilon_2} \del_{\varepsilon_3} \vert_{\varepsilon = 0}(\overline{u^i}u^i u^j) = \sum_{\sigma \in S_3} \overline{u^i_{\sigma(1)}} u^i_{\sigma(2)} u^j_{\sigma(3)}
\end{equation}
since $u^j = 0$ when $\varepsilon = 0$ and so most of the terms in the expansion of $\del_{\varepsilon_1} \del_{\varepsilon_2} \del_{\varepsilon_3} (\overline{u^i}u^i u_j)$ vanish when evaluating at $\varepsilon = 0$. Hence,
\begin{equation}
	\tilde{w} = \sum_{\sigma \in S_3} \sum_{j=1}^n \sum_{i=1}^n \overline{u^i_{\sigma(1)}} u^i_{\sigma(2)} u^j_{\sigma(3)} = \sum_{\sigma \in S_3} \dual{u_{\sigma(1)}}{u_{\sigma(2)}} u_{\sigma(3)}.
\end{equation}
Now, it follows from Stokes' formula \cite{cekic, kop} that
\begin{equation}
	(\nabla_\nu w, f)_{\del M} = -(\Delta_A w, v_f)_M + (\nabla w, \nabla v_f)_M
\end{equation}
where the inner products are with respect to $L^2(\del M; \C^n)$ and $L^2(M; \C^n)$ respectively. Expanding with the expression for $\Delta_A w$, we get
\begin{align*}
	(\nabla_\nu w, f)_{\del M} &= -\lambda^2 (w, v_f)_M + (\tilde{w}, v_f)_M + (\nabla w, \nabla v_f)_M \\
	&= -\lambda^2 (w, v_f)_M + (\tilde{w}, v_f)_M + (w, \Delta_A v_f)_M + (w, \nabla_\nu v_f)_{\del M} \\
	&= (\tilde{w}, v_f)_M
\end{align*}
since $\Delta_A v_f = \lambda^2 v_f$ and $w \vert_{\del M} = 0$. The result follows from expanding $(\tilde{w}, v_f)_M$ with the expression for $\tilde{w}$.
\end{proof}

\begin{lem}\label{lem:thirdlin_estimate}
Let $A$ be a smooth Hermitian connection on $E = M \times \C^n$ and let $k \in \N$, $k \geq \frac{m}{8}$. For $i = 1, \dots, 4$, let $f_i, g_i \in H^{2k}(\del M)$ be functions whose $H^{2k}(\del M)$ norms are bounded from above by $C_0 > 0$. There is $C = C(M, g, A, C_0, k)$ such that
\begin{equation}\label{eq:distance_thirdlin}
	\abs{((D^3 \Lambda^A_\lambda)_0(f_1, f_2, f_3), f_4) - ((D^3 \Lambda^A_\lambda)_0(g_1, g_2, g_3), g_4)} \leq C \lambda^{4(m + 2k + 3)} \max_{i =1, \dots, 4} \norm{f_i - g_i}_{H^{2k}(\del M)}.
\end{equation}
for $\lambda \in [1, \infty)\setminus J$.
\end{lem}

\begin{proof}
From Lemma \ref{lem:thirdlin} and the triangle inequality, we can bound the LHS of \eqref{eq:distance_thirdlin} by
\begin{equation}
	\sum_{\sigma \in S_3} \int_M \abs{ \dual{u_{\sigma(1)}}{u_{\sigma(2)}} \dual{u_{\sigma(3)}}{u_4} - \dual{v_{\sigma(1)}}{v_{\sigma(2)}} \dual{v_{\sigma(3)}}{v_4}} \de V_g
\end{equation}
with $u_i = P_\lambda f_i$ and $v_i = P_\lambda g_i$, where $P_\lambda$ is as in Corollary \ref{coro:P_lambda}. We will bound the term corresponding to the trivial permutation in $S_3$. The other terms will have similar bounds. We can expand
\begin{align*}
	\dual{u_1}{u_2} \dual{u_3}{u_4} - &\dual{v_1}{v_2} \dual{v_3}{v_4} = \\
	&\dual{u_1}{u_2}(\dual{u_3 - v_3}{u_4} + \dual{v_3}{u_4 - v_4}) + \dual{v_3}{v_4}(\dual{u_1 - v_1}{v_2} + \dual{u_1}{u_2 - v_2}).
\end{align*}
Using Cauchy-Schwarz twice, the first term in the expansion yields
\begin{equation}
	\int_M \abs{\dual{u_1}{u_2} \dual{u_3 - v_3}{u_4}} \de V_g \leq \norm{u_1}_{L^4(M)}\norm{u_2}_{L^4(M)} \norm{u_3 - v_3}_{L^4(M)} \norm{u_4}_{L^4(M)}
\end{equation}
Since $k \geq \frac{m}{8}$, we have the Sobolev embedding $H^{2k}(M) \hookrightarrow L^4(M)$. Therefore,
\begin{equation}
	\norm{v_1}_{L^4(M)} = \norm{P_\lambda g_1}_{L^4(M)} \leq C \norm{P_\lambda g_1}_{H^{2k}(M)} \leq C\lambda^{m + 2k + 3}\norm{g_1}_{H^{2k}(\del M)}
\end{equation}
by Corollary \ref{coro:P_lambda} and since $H^{2k}(\del M) \subset H^{2k - 1/2}(\del M)$. Similar expressions hold for the other terms and $\norm{u_3 - v_3}_{L^4(M)} = \norm{P_\lambda (f_3 - g_3)}_{L^4(M)}$. Combining all the different terms and using that the $H^{2k}(\del M)$ norms of $f_i$ and $g_i$ are bounded by $C_0$ yields \eqref{eq:distance_thirdlin}.
\end{proof}

\subsection{Technical conditions for Theorem \ref{thm:calderon_cubic}}\label{sec:technical_conditions}

For $T > 0$, $\theta_0 \in (0, \pi/2)$, and $\varepsilon > 0$, let $\mathcal{G}_{T, \theta_0, \varepsilon}$ be the subset of nontangential geodesics as in Section \ref{sec:GBestimates}. For $x \in M$, we let $\mathcal{G}_{T, \theta_0, \varepsilon}(x)$ be the set of $v_x \in S_x M$ such that $\gamma_{x,v} \in \mathcal{G}_{T, \theta_0, \varepsilon}$. For $\theta_1 \in (0, \pi/2)$, $0 < r < \mathrm{Inj}(M)/2$, and $c_0 > 0$, we let $\mathcal{H}_{\theta_1, r, c_0}(x) \subset \mathcal{F}_x$ be the set of pairs of geodesics $(v_x, w_x)$ such that
\begin{enumerate}
	\item[(i)] The full geodesics $\gamma_{x,v}$ and $\gamma_{x,w}$ are in $\mathcal{G}_{T, \theta_0, \varepsilon}(x)$ with $\gamma_{x,v}(0) = \gamma_{x,w}(0) = x$;
	\item[(ii)] The geodesics $\gamma_{x,v}$ and $\gamma_{x,w}$ intersect at $x$ at an angle greater or equal to $\theta_1$;
	\item[(iii)] If $d(\gamma_{x,v}(t), \gamma_{x,w}(s)) = d < r$, then $\abs{t} \leq c_0 d$ and $\abs{s} \leq c_0 d$.
\end{enumerate}
Note that $v_x \in \mathcal{G}_{T, \theta_0, \varepsilon}(x)$ if and only if $-v_x \in \mathcal{G}_{T, \theta_0, \varepsilon}(x)$ and $(v_x, w_x) \in \mathcal{H}_{\theta_1, r, c_0}(x)$ if and only if $(w_x, v_w) \in \mathcal{H}_{\theta_1, r, c_0}(x)$. We denote $\mathcal{H} = \mathcal{H}_{\theta_1, r, c_0} = \cup_{x \in M} \mathcal{H}_{\theta_1, r, c_0}(x) \subset \mathcal{F}$. In Theorem \ref{thm:calderon_cubic}, we assume the following about $(M, g)$.
\begin{itemize}
	\item[(H)] There exist constants $T > 0$, $\theta_0 \in (0, \pi/2)$, $\varepsilon > 0$, $\theta_1 \in (0, \pi/2)$, $0 < r < \mathrm{Inj}(M)/2$, and $c_0 > 0$ such that there is a complete geodesic graph structure $(\mathcal{V}, \mathcal{E})$ with $\mathcal{V}(x) \subset \mathcal{G}_{T, \theta_0, \varepsilon}(x)$ and $\mathcal{E}(x) \subset \mathcal{H}_{\theta_1, r, c_0}(x)$ for all $x \in M^{\intr}$.
\end{itemize}

\subsection{Recovery of the broken non-abelian X-ray transform}

We show the following.

\begin{thm}\label{thm:recovery_broken}
Let $T > 0$, $\theta_0 \in (0, \pi/2)$, $\varepsilon > 0$, $\theta_1 \in (0, \pi/2)$, $0 < r < \mathrm{Inj}(M)/2$, and $c_0 > 0$. Let $A_1$ and $A_2$ be smooth Hermitian connections on $(M, g)$. For $\delta > 0$, let $J$ be as in \eqref{eq:resolvent_intro} with $\abs{J} < \delta$. Unless
\begin{equation}
	S^{A_1}(v_x, w_x) = \pm S^{A_2}(v_x, w_x)
\end{equation}
for all pairs $(v_x, w_x) \in \mathcal{H}_{\theta_1, r, c_0}$, there exists a constant $\lambda_0 = \lambda_0(M, g, T, \theta_0, \varepsilon, \theta_1, r, c_0, A_1, A_2, \delta)$ such that $\Lambda^{A_1}_\lambda \neq \Lambda^{A_2}_\lambda$ for all $\lambda \in (\lambda_0, \infty) \setminus J$.
\end{thm}

To prove Theorem \ref{thm:recovery_broken}, we need to fix some notation. Let $x \in M$, and let $\gamma : [0, \tau] \to M$ and $\eta : [0, \sigma] \to M$ be two nontangential geodesics that intersect exactly once at $x$ with $\gamma(t_0) = \gamma(s_0) = x$ corresponding to $(v_x, w_x) = (\dot{\gamma}(t_0), \dot{\eta}(s_0)) \in \mathcal{H}_{r, \theta_1, c_0}(x)$. We let $u$ and $\tilde{u}$ be Gaussian beams of order $K$ along $\gamma$ with respect to $A$ that are normalised in $L^4$ such that $u$ is obtained from solving ODEs in forward time and $\tilde{u}$ is obtained from solving ODEs in backward time. Similarly, we let $v$ and $\tilde{v}$ be Gaussian beams of order $K$ obtained analogously along the geodesic $\eta$.

As in Section \ref{sec:GBestimates}, the Gaussian beams $u$ and $\tilde{u}$ are supported on a $\delta'$-neighbourhhod of $\gamma$; the Gaussian beams $v$ and $\tilde{v}$ on a $\delta'$-neighbourhood of $\eta$. We chose $\delta'$ in terms of $N$, $g$, and $T$ so that $\im(\phi(t, y)) \geq c\abs{y}^2$ for all $\abs{y} < \delta'$. Without loss, we can choose $\delta'$ in terms of $r$ and $c_0$ as well to get uniform estimates in $\mathcal{H}_{r, \theta_1, c_0}$. Let
\begin{equation}
	\rho = \min\left(\frac{r}{2}, \frac{\mathrm{Inj}(M)}{2c_0}\right).
\end{equation}
By the same arguments as in \cite[Section 4.2]{katya}, if $\delta' < \rho$, then $\pi_\gamma : N_\gamma \to M$ and $\pi_\eta : N_\eta \to M$ both have at most one pre-image for every point in the ball $B_{\rho}(x) \subset M$, and hence all our Gaussian beams take a simple form on that ball, that is, $u = \lambda^{\frac{m-1}{8}} e^{\I\lambda \phi} a$, $\tilde{u} = \lambda^{\frac{m-1}{8}} e^{\I\lambda \tilde{\phi}} \tilde{a}$, $v = \lambda^{\frac{m-1}{8}} e^{\I\lambda \psi} b$, $\tilde{v} = \lambda^{\frac{m-1}{8}} e^{\I\lambda \tilde{\psi}} \tilde{b}$. Moreover, if $\delta' < \frac{\rho}{(1 + 2c_0)}$, then $\dual{u}{v}$ is supported on $B_{\rho}(x)$.

The zeroth order terms of $a$ and $\tilde{a}$ at $x$ are given respectively by
\begin{align}
	a_0(x) &= \exp\left(-\frac{1}{2} \int_{0}^{t_0} \tr H_\phi(s) \de s\right) P^A_{\gamma[0,t_0]} f_0, \\ \tilde{a}_0(x) &= \exp\left(-\frac{1}{2} \int_{0}^{\tau-t_0} \tr H_{\tilde{\phi}}(s) \de s\right) (P^A_{\gamma[t_0,\tau]})^{-1} f_\tau,
\end{align}
where the phases are given by $\phi(t, y) = t + \frac{1}{2} H_\phi(t) y \cdot y + O(\abs{y}^3)$ and $\tilde{\phi}(t, y) = - t + \frac{1}{2} H_{\tilde{\phi}}(t) y \cdot y + O(\abs{y}^3)$ when seen in $N_\gamma$, and $f_0$ and $f_\tau$ are the initial conditions for $a_0$ and $\tilde{a}_0$ at $\gamma(0)$ and $\gamma(\tau)$ respectively. In what follows, we will assume $\abs{f_0} = \abs{f_\tau} = 1$. Similarly, the zeroth order terms of $b$ and $\tilde{b}$ are given at $x$ by
\begin{align}
	b_0(x) &= \exp\left(-\frac{1}{2} \int_{0}^{s_0} \tr H_\psi(s) \de s\right) P^A_{\eta[0,s_0]} g_0, \\ \tilde{b}_0(x) &= \exp\left(-\frac{1}{2} \int_{0}^{\sigma-s_0} \tr H_{\tilde{\psi}}(s) \de s\right) (P^A_{\eta[s_0,\sigma]})^{-1} g_\sigma,
\end{align}
with $\psi(s,z) = r + \frac{1}{2}H_\psi(s) z \cdot z + O(\abs{z}^3)$ and $\tilde{\psi}(s,z) = - s + \frac{1}{2} H_{\tilde{\psi}}(s) z \cdot z + O(\abs{z}^3)$ when seen in $N_\eta$, and $g_0$ and $g_\sigma$ are the initial conditions for $b_0$ and $\tilde{b}_0$ at $\eta(0)$ and $\eta(\sigma)$ respectively with $\abs{g_0} = \abs{g_\sigma} = 1$.

\begin{lem}\label{lem:broken_stationary}
Let $T > 0$, $\theta_0 \in (0, \pi/2)$, $\varepsilon > 0$, $\theta_1 \in (0, \pi/2)$, $0 < r < \mathrm{Inj}(M)/2$, and $c_0 > 0$. Let $u$, $\tilde{u}$, $v$, $\tilde{v}$ be Gaussian beams of order $K$ with respect to a smooth Hermitian connection $A$ along the nontangential geodesics $\gamma, \eta \in \mathcal{G}_{T, \theta_0, \varepsilon}$ as above such that $(\dot{\gamma}(t_0), \dot{\eta}(s_0)) \in \mathcal{H}_{r, \theta_1, c_0}(x)$. There are positive constants $c_{\min} = c_{\min}(M, g, T, K)$, $c_{\max} = c_{\max}(M, g, T, K, \theta_1)$, and $C = C(M, g, T, r, c_0, A, K)$ such that
\begin{equation}
	\abs{\lambda^{\frac{1}{2}}((D^3 \Lambda_\lambda^A)_0(u, \tilde{u}, v), \tilde{v})_{L^2(\del M)} - c_{\gamma, \eta}(\dual{a_0}{b_0}_x\dual{\tilde{a}_0}{\tilde{b}_0}_x + \dual{\tilde{a}_0}{b_0}_x \dual{a_0}{\tilde{b}_0}_x)} \leq C \lambda^{-1} 
\end{equation}
for some constant $c_{\gamma, \eta}$ such that $c_{\min} \leq \abs{c_{\gamma, \eta}} \leq c_{\max}$.
\end{lem}

\begin{proof}
Take $k \geq \frac{m}{8}$ so that the embedding $H^{2k}(M) \hookrightarrow L^4(M)$ holds. By Theorem \ref{thm:beams_manifold}, we can take $K$ large enough so that
\begin{equation}
	\norm{(\Delta - \lambda^2)u}_{H^{2k-2}(M)} \leq C \lambda^{-m -2k - 3}.
\end{equation}
By Corollary \ref{coro:remainders}, we can therefore find $r \in H^{2k}(M) \cap H_0^1(M)$ such that $(\Delta - \lambda^2)(u + r) = 0$ and
\begin{equation}
	\norm{r}_{L^4(M)} \leq C \norm{r}_{H^{2k}(M)} \leq C \lambda^{m + 2k + 1} \norm{(\Delta - \lambda^2)u}_{H^{2k-2}(M)} \leq C\lambda^{-2}.
\end{equation}
We can find similar remainders for the other Gaussian beams. Since the quasimodes are bounded in $L^4(M)$, Lemma \ref{lem:thirdlin} yields 
\begin{align}\label{eq:6terms}
	\lambda^{\frac{1}{2}}((D^3 \Lambda^A_\lambda)_0 (u, \tilde{u}, v), \tilde{v})_{L^2(\del M)} &= \lambda^{\frac{1}{2}}\int_M \dual{u}{v} \dual{\tilde{u}}{\tilde{v}} + \dual{u}{\tilde{u}} \dual{v}{\tilde{v}} + \dual{v}{u} \dual{\tilde{u}}{\tilde{v}} + \dual{v}{\tilde{u}} \dual{u}{\tilde{v}} \\
	&\qquad + \dual{\tilde{u}}{u} \dual{v}{\tilde{v}} + \dual{\tilde{u}}{v} \dual{u}{\tilde{v}} \de V_g + O(\lambda^{-\frac{3}{2}}).
\end{align}
Each term in the integral is supported on $B_{\rho}(x)$ and we can therefore expand them locally. Each expansion yields an oscillatory integral. The imaginary part of their phase is always the sum of the imaginary parts of $\phi$, $\tilde{\phi}$, $\psi$ and $\tilde{\psi}$. It is therefore always nonnegative and vanishes only where the geodesics $\gamma$ and $\eta$ cross, which is at $x$. Only the first and last oscillatory integrals have a phase with a critical point at $x$. By non-stationary phase \cite[Theorem 7.7.1]{hormander1}, the other integrals are of order $O(\lambda^{-j})$ for any $j$ uniformly in $\mathcal{H}_{r, \theta_1, c_0}$.

On the ball $B_{\rho}(x)$, let $(t, y)$ be Fermi coordinates adapted to the geodesic $\gamma$ with $x = (t_0, 0)$ and let $(s,z)$ be Fermi coordinates adapted to $\eta$ with $x = (s_0, 0)$. We can write the coordinates $(s, z)$ in terms of $(t,y)$ as $(s(t,y), z(t,y))$ such that $s(t_0,0) = s_0$ and $z(t_0, 0) = 0$. This allows us to view $\psi$ and $\tilde{\psi}$ as a function of $(t,y)$.

The first term in \eqref{eq:6terms} can be expanded as
\begin{align}
	&\lambda^{\frac{1}{2}}\int_M \dual{u}{v} \dual{\tilde{u}}{\tilde{v}} \de V_g = \lambda^{\frac{m}{2}} \int_{\R}\int_{\R^{m-1}} e^{\I\lambda \Phi} (\dual{a_0}{b_0}\dual{\tilde{a}_0}{\tilde{b}_0} + O(\lambda^{-1})) \chi(y/\delta')^2 \chi(z/\delta')^2 \abs{g}^{\frac{1}{2}} \de y \de t
\end{align}
where $\Phi = \psi + \tilde{\psi} - \overline{\phi} - \overline{\tilde{\phi}}$. The integrand is smooth and compactly supported on $B_{\rho}(x)$. From Theorem \ref{thm:beams_tube}, there is $c > 0$ such that
\begin{equation}
	\im(\phi(t,y)) \geq c \abs{y}^2, \quad \im(\tilde{\phi}(t,y)) \geq c \abs{y}^2, \quad \im(\psi(s,z)) \geq c \abs{z}^2, \quad \im(\tilde{\psi}(s,z)) \geq c \abs{z}^2.
\end{equation}
Therefore, by the same arguments as in \cite[Section 4.2]{katya}, we have $\im(\Phi''(x)) \geq (c \sin^2 \theta_1) I$. The phase $\Phi$ has a single critical point at $x$. It is non-degenerate and $\Phi(x) = 0$. Theorem \ref{thm:stationaryphase} then yields
\begin{equation}
	\lambda^{\frac{1}{2}}\int_M \dual{u}{v} \dual{\tilde{u}}{\tilde{v}} \de V_g = \det(\Phi''(x)/(2\pi i))^{-\frac{1}{2}} \dual{a_0}{b_0}_x \dual{\tilde{a}_0}{\tilde{b}_0}_x + O(\lambda^{-1})
\end{equation}
where the implicit constants in the big-O notation are uniform among pairs of geodesics in $\mathcal{H}_{r, \theta_1, c_0}$ by the estimates in Theorems \ref{thm:beams_tube} and \ref{thm:stationaryphase}. A similar expansion holds for the last term in \eqref{eq:6terms} with the same phase $\Phi$, and hence $c_{\gamma, \eta} = \det(\Phi''(x)/(2\pi i))^{-\frac{1}{2}}$.

It only remains to give bounds on $c_{\gamma, \eta}$, which is equivalent to bounds on $\abs{\det(\Phi''(x))}$. Theorem \ref{thm:beams_tube} yields $C^2$-bounds on the phases in the Gaussian beams, and hence $\abs{\det(\Phi''(x))}$ is uniformly bounded from above in $\mathcal{H}_{r, \theta_1, c_0}$, which translates to $\abs{c_{\gamma, \eta}} \geq c_{\min}$. In the language of \cite{note_matrices}, $\Phi''(x)$ is a strictly dissipative matrix since its imaginary part is positive-definite and hence there exists a nonsingular matrix $Q$ as well as some real numbers $\alpha_1, \dots, \alpha_m$ such that
\begin{equation}
	\Phi''(x) = Q \mathrm{diag}(i + \alpha_1, \dots, i + \alpha_m) Q^{\dagger}
\end{equation}
and $\im(\nabla_x^2 \Phi) = Q Q^{\dagger}$. Therefore,
\begin{equation}
	\abs{\det(\nabla_x^2 \Phi)} = \abs{\det Q}^2 \abs{\prod_{j=1}^m (i + \alpha_j)} \geq \abs{\det Q}^2 = \det(\im(\nabla_x^2 \Phi)) \geq (c \sin^2 \theta_1)^m
\end{equation}
where the last inequality follows from $\im(\nabla_x^2 \Phi) \geq (c \sin^2 \theta_1)I$, and so $\abs{c_{\gamma, \eta}} \leq c_{\max}$.
\end{proof}

From the expressions for $a_0$ and $b_0$, we can expand
\begin{equation}
	\dual{a_0}{b_0}_x = \exp\left(-\frac{1}{2}\int_0^{t_0} \tr \overline{H_\phi(s)} \de s - \frac{1}{2} \int_0^{s_0} \tr H_\psi (s) \de s\right) \dual{P^A_{\gamma[0,t_0]} f_0}{P^A_{\eta[0,s_0]} g_0}
\end{equation}
As the connection is Hermitian, $(P^A_{\eta[0,s_0]})^\dagger = (P^A_{\eta[0,s_0]})^{-1}$, and hence
\begin{equation}
	\dual{P^A_{\gamma[0,t_0]} f_0}{P^A_{\eta[0,s_0]} g_0} = \dual{(P^A_{\eta[0,s_0]})^{-1} P^A_{\gamma[0,t_0]} f_0}{g_0} = \dual{S^A_{\Gamma_{00}} f_0}{g_0}
\end{equation}
where $\Gamma_{00}$ is the geodesic given by following $\gamma$ first from $\gamma(0)$ to $x = \gamma(t_0)$ and coming back to $\del M$ by following $\eta$ from $x = \eta(s_0)$ to $\eta(0)$. By expanding the other terms similarly, we get
\begin{equation}
	\dual{a_0}{b_0}_x\dual{\tilde{a}_0}{\tilde{b}_0}_x + \dual{\tilde{a}_0}{b_0}_x \dual{a_0}{\tilde{b}_0}_x = c_{\phi, \psi}(\dual{S^A_{\Gamma_{00}} f_0}{g_0} \dual{S^A_{\Gamma_{\tau \sigma}} f_\tau}{g_\sigma} + \dual{S^A_{\Gamma_{\tau 0}} f_\tau}{g_0} \dual{S^A_{\Gamma_{0 \sigma}} f_0}{g_\sigma})
\end{equation}
for some constant $c_{\phi, \psi}$ that depends only on the phases (and is therefore independent of $A$), and where the broken geodesics $\Gamma_{00}$, $\Gamma_{0\sigma}$, $\Gamma_{\tau 0}$, and $\Gamma_{\tau \sigma}$ are as in Figure \ref{fig:brokengeodesics}. We can rewrite all the broken X-rays in terms of $S^A_{\Gamma_{00}}$, $P^A_{\gamma} = P^A_{\gamma[0, \tau]}$, and $P^A_{\eta} = P^A_{\eta[0, \sigma]}$ as
\begin{equation}
	S^A_{\Gamma_{\tau \sigma}} = P^A_{\eta} S^A_{\Gamma_{00}} (P^A_\gamma)^{-1}, \qquad S^A_{\Gamma_{\tau 0}} = S^A_{\Gamma_{00}} (P^A_{\gamma})^{-1}, \qquad S^A_{\Gamma_{0 \sigma}} = P^A_{\eta} S^A_{\Gamma_{00}}.
\end{equation}
By taking $z_1 = f_0$, $z_2 = g_0$, $z_3 = (P^A_{\gamma})^{-1} f_\tau$, and $z_4 = (P^A_{\eta})^{-1} g_\sigma$, we can rewrite
\begin{equation}\label{eq:broken_Gamma}
	\dual{a_0}{b_0}_x\dual{\tilde{a}_0}{\tilde{b}_0}_x + \dual{\tilde{a}_0}{b_0}_x \dual{a_0}{\tilde{b}_0}_x = c_{\phi, \psi}(\dual{S^A_{\Gamma_{00}} z_1}{z_2} \dual{S^A_{\Gamma_{00}} z_3}{z_4} + \dual{S^A_{\Gamma_{00}} z_3}{z_2} \dual{S^A_{\Gamma_{00}} z_1}{z_4}).
\end{equation}
Moreover, we can bound $\abs{c_{\phi, \psi}}$ from below using Theorem \ref{thm:beams_tube} since a $C^2$ bound on $\phi$ yields an upper bound on $\tr H_\phi$.

\begin{figure}
\begin{center}
\scalebox{.8}{
\begin{tikzpicture}
		\node [label = left:$\gamma(0)$] (0) at (-3, -3) {};
		\node [label = right:$\gamma(\tau)$] (1) at (3, 3) {};
		\node [label = left:$\eta(0)$] (2) at (-3, 3) {};
		\node [label = right:$\eta(\sigma)$] (3) at (3, -3) {};
		\node (4) at (-1.5, -2.5) {};
		\node (5) at (0, -1) {};
		\node (6) at (1.5, -2.5) {};
		\node (7) at (-2.5, -1.5) {};
		\node (8) at (-1, 0) {};
		\node (9) at (-2.5, 1.5) {};
		\node (10) at (-1.5, 2.5) {};
		\node (11) at (0, 1) {};
		\node (12) at (1.5, 2.5) {};
		\node (13) at (2.5, 1.5) {};
		\node (14) at (1, 0) {};
		\node (15) at (2.5, -1.5) {};
		\node (16) at (-2, -2) {};
		\node (17) at (-2, 2) {};
		\node (18) at (-1, -2) {};
		\node (19) at (-2, -1) {};
		\node (20) at (1, 2) {};
		\node (21) at (2, 1) {};
		\node [label = below:$\Gamma_{0 \sigma}$] at (0, -1.5) {};
		\node [label = left:$\Gamma_{00}$] at (-1.5, 0) {};
		\node [label = above:$\Gamma_{\tau 0}$] at (0, 1.5) {};
		\node [label = right:$\Gamma_{\tau \sigma}$] at (1.5, 0) {};
		\node [label = $x$] at (0,0) {};
		
		\filldraw[black] (0,0) circle (1.5pt);
		\filldraw[black] (0) circle (1.5pt);
		\filldraw[black] (1) circle (1.5pt);
		\filldraw[black] (2) circle (1.5pt);
		\filldraw[black] (3) circle (1.5pt);
		
		\draw [->, thick] (0.center) to (16.center);
		\draw [thick] (16.center) to (1.center);
		
		\draw [->, thick] (2.center) to (17.center);
		\draw [thick] (17.center) to (3.center);
		
		\draw [dashed, ->] (4.center) to (18.center);
		\draw [dashed] (18.center) to (5.center);
		\draw [dashed] (5.center) to (6.center);
		
		\draw [dashed, ->] (7.center) to (19.center);
		\draw [dashed] (19.center) to (8.center);
		\draw [dashed] (8.center) to (9.center);
		
		\draw [dashed, ->] (12.center) to (20.center);
		\draw [dashed] (20.center) to (11.center);
		\draw [dashed] (11.center) to (10.center);
		
		\draw [dashed, ->] (13.center) to (21.center);
		\draw [dashed] (21.center) to (14.center);
		\draw [dashed] (14.center) to (15.center);
\end{tikzpicture}
}
\end{center}
\caption{Broken geodesics obtained from $\gamma$ and $\eta$.}
\label{fig:brokengeodesics}
\end{figure}
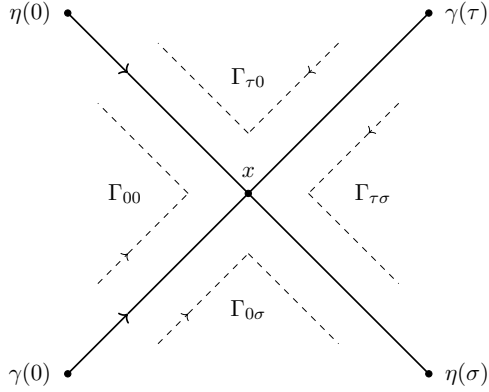

\begin{lem}\label{lem:Fpm}
For a complex matrix $Q \in \C^{n \times n}$, let $F[Q] : (\C^n)^4 \to \C$ be the function
\begin{equation}
	F[Q](z_1, z_2, z_3, z_4) = \dual{Qz_1}{z_2}\dual{Qz_3}{z_4} + \dual{Qz_3}{z_2}\dual{Qz_1}{z_4}.
\end{equation}
Then, $F[Q_1] = F[Q_2]$ if and only if $Q_1 = \pm Q_2$.
\end{lem}

\begin{proof}
Let $Q = (q_{ij})$ and let $e_1, \dots, e_n$ be the canonical basis of $\C^n$. Unless $Q = 0$, there is some nonzero $q_{ij}$. Without loss, suppose $q_{11} \neq 0$. We have $	F[Q](e_1, e_1, e_1, e_1) = 2 \overline{q}_{11}^2$ and therefore $F_Q$ determines $q_{11}$ up to a sign that we fix. Then, $F[Q](e_1, e_1, e_1, e_i) = 2 \overline{q}_{11} \overline{q}_{i1}$ determines uniquely $q_{i1}$ for $1 \leq i \leq n$, and therefore $F[Q](e_1, e_i, e_i, e_i) = 2 \overline{q}_{i1} \overline{q}_{ii}$ determines uniquely $q_{ii}$ for $1 \leq i \leq n$. Finally, $	F[Q](e_j, e_i, e_i, e_i) = 2 \overline{q}_{ij} \overline{q}_{ii}$ determines uniquely $q_{ij}$ for all $i, j$. Up to the initial choice of sign for $q_{11}$, the function $F[Q]$ determines $Q$ and hence $F[Q_1] = F[Q_2]$ if and only if $Q_1 = \pm Q_2$. 
\end{proof}

In the notation of Lemma \ref{lem:Fpm}, we can therefore rewrite the result of Lemma \ref{lem:broken_stationary} as
\begin{equation}\label{eq:broken_stationary}
	\abs{\lambda^{\frac{1}{2}}((D^3 \Lambda_\lambda^A)_0(u, \tilde{u}, v), \tilde{v})_{L^2(\del M)} - c_{\gamma, \eta} c_{\phi, \psi} F[S^A(v_x, w_x)](z_1, z_2, z_3, z_4)} \leq C \lambda^{-1}
\end{equation}
for $z_i$ as above, $\abs{z_i} = 1$. We now have all the ingredients to prove Theorem \ref{thm:recovery_broken}.

\begin{proof}[Proof of Theorem \ref{thm:recovery_broken}]
By Theorem \ref{thm:GBtraces}, either $\mathcal{P}_{\gamma}^{A_1, K} = \mathcal{P}_{\gamma}^{A_2, K}$ for all $\gamma \in \mathcal{G}_{T, \theta_0, \varepsilon}$ or there exists $\lambda_0$ such that $\DN^{A_1}_\lambda \neq \DN^{A_2}_\lambda$ for all $\lambda \in (\lambda_0, \infty) \setminus J$. If the transport maps of the Gaussian beams don't agree, then we are done since $\Lambda^{A_1}_\lambda \neq \Lambda^{A_2}_\lambda$ whenever $\DN^{A_1}_\lambda \neq \DN^{A_2}_\lambda$ as $\DN^A_\lambda$ is the first order linearisation of $\Lambda^{A}_\lambda$. Hence, let us suppose that $\Lambda^{A_1}_\lambda = \Lambda^{A_2}_\lambda$ and $\mathcal{P}_{\gamma}^{A_1, K} = \mathcal{P}_{\gamma}^{A_2, K}$ for all $\gamma \in \mathcal{G}_{T, \theta_0, \varepsilon}$.

Let $\gamma$ and $\eta$ be nontangential geodesics corresponding to $(v_x, w_x) \in \mathcal{H}_{\theta_1, r, c_0}$. Let $u_i$, $\tilde{u}_i$, $v_i$, and $\tilde{v}_i$ be Gaussian beams of order $K$ constructed along $\gamma$ and $\eta$ as above with respect to the connection $A_i$ for $i = 1, 2$. We choose $u_1 = \lambda^{\frac{m-1}{8}} e^{\I\lambda \phi} a^{(1)}$ and $u_2 = \lambda^{\frac{m-1}{8}} e^{\I\lambda \phi} a^{(2)}$ such that their initial conditions at $\gamma(0)$ are in $\mathcal{A}_{\del M, \gamma(0)}^K(C_0)$ and agree to all orders. Since $\mathcal{P}_{\gamma}^{A_1, K} = \mathcal{P}_{\gamma}^{A_2, K}$ and the boundary values of both $u_1$ and $u_2$ are supported on small neighbourhoods of $\gamma(0)$ and $\gamma(\tau)$, Lemma \ref{lem:local_estimate} yields
\begin{equation}
	\norm{u_1 - u_2}_{H^{2k}(\del M)} \leq \sum_{j = 0}^K \lambda^{-j} \lambda^{-\frac{m-1}{8}}\norm{\lambda^{\frac{m-1}{4}}e^{i\lambda\phi}(a_j^{(1)} - a_j^{(2)})}_{H^{2k}(\del M)} \leq C\lambda^{-\frac{3K}{2} - \frac{m-1}{8} + k}
\end{equation}
since $a_j^{(1)} - a_j^{(2)}$ vanishes to order $3K - 2j$ at $\gamma(0)$ and $\gamma(\tau)$. Importantly, the constant $C$ is uniform with respect to $C_0$ and $\gamma \in \mathcal{G}_{T, \theta_0, \varepsilon}$. We choose the other pairs of Gaussian beams with initial conditions that agree in $\mathcal{A}_{\del M, \gamma(0)}^K(C_0)$ as well. Similar estimates hold for them as well. Therefore, by Lemma \ref{lem:thirdlin_estimate}, if we take $k \geq \frac{m}{8}$, we can choose $K$ large enough so that
\begin{equation}
	\abs{\lambda^{\frac{1}{2}}((D^3 \Lambda_\lambda^{A_1})_0(u_1, \tilde{u}_1, v_1), \tilde{v}_1)_{L^2(\del M)} - \lambda^{\frac{1}{2}}((D^3 \Lambda_\lambda^{A_2})_0(u_2, \tilde{u}_2, v_2), \tilde{v}_2)_{L^2(\del M)}} \leq C\lambda^{-1}
\end{equation}
for some $C = C(M, g, T, \theta_0, \varepsilon, A_1, A_2, K, k, \delta)$. Note that to use Lemma \ref{lem:thirdlin_estimate}, we needed the remainders of the Gaussian beam quasimodes to vanish on the boundary. Since $P^{A_1}_\gamma = P^{A_2}_\gamma$ and $P^{A_1}_\eta = P^{A_2}_\eta$, the values for $z_3$ and $z_4$ in \eqref{eq:broken_stationary} agree for both connections. By the triangle inequality and equation \eqref{eq:broken_stationary}, we then have
\begin{equation}
	\abs{c_{\gamma, \eta} c_{\phi, \psi} (F[S^{A_1}(v_x, w_x)](z_1, z_2, z_3, z_4) - F[S^{A_2}(v_x, w_x)](z_1, z_2, z_3, z_4))} \leq C\lambda^{-1}
\end{equation}
for all $\abs{z_i} = 1$. The magnitudes of $c_{\gamma, \eta}$ and $c_{\phi, \psi}$ can be bounded from below uniformly in $\mathcal{H}$ by the discussion above. Therefore, there exists a constant $C > 0$ such that
\begin{equation}
	S_{\max} := \sup_{(v_x, w_x) \in \mathcal{H}} \sup_{\abs{z_i} = 1} \abs{F[S^{A_1}(v_x, w_x)](z_1, z_2, z_3, z_4) - F[S^{A_2}(v_x, w_x)](z_1, z_2, z_3, z_4)} \leq C\lambda^{-1}.
\end{equation}
If $S_{\max} = 0$, then $S^{A_1}(v_x, w_x) = \pm S^{A_2}(v_x, w_x)$ for all $(v_x, w_x) \in \mathcal{H}_{\theta_1, r, c_0}$ by Lemma \ref{lem:Fpm}. And if $S_{\max} \neq 0$, then we get a contradiction if $\lambda > C/S_{\max}$, implying $\Lambda^{A_1}_\lambda \neq \Lambda^{A_2}_\lambda$.
\end{proof}

\subsection{Proof of Theorem \ref{thm:calderon_cubic}}

\begin{proof}[Proof of Theorem \ref{thm:calderon_cubic}]
Let $\gamma_1, \dots, \gamma_L$ be the nontangential geodesics as in Theorem \ref{thm:injectivity_sign}. As there is only a finite number of them and their endpoints are distinct, we can find constants $T$, $\theta_0$, $\varepsilon$ such that all the curves $\gamma_i$ are in $\mathcal{G}_{T, \theta_0, \varepsilon}$. Moreover, by (H), we can also assume that we can find constants $\theta_1$, $r$, $c_0$ such that there is a complete geodesic graph structure $(\mathcal{V}, \mathcal{E})$ with $\mathcal{V}(x) \subset \mathcal{G}_{T, \theta_0, \varepsilon}(x)$ and $\mathcal{E}(x) \subset \mathcal{H}_{\theta_1, r, c_0}(x)$ for all $x \in M^{\intr}$.

Since $\DN_\lambda^A$ is the first order linearisation of $\Lambda^A_\lambda$, we know from Theorems \ref{thm:GBtraces} and \ref{thm:recovery_broken} that unless $P^{A_1}_\gamma = P^{A_2}_\gamma$ for all $\gamma \in \mathcal{G}_{T, \theta_0, \varepsilon}$ and $S^{A_1}(v_x, w_x) = \pm S^{A_2}(v_x, w_x)$ for all pairs $(v_x, w_x) \in \mathcal{H}_{\theta_1, r, c_0}$, there exists a constant $\lambda_0$ such that $\Lambda^{A_1}_\lambda \neq \Lambda^{A_2}_\lambda$ for all $\lambda \in (\lambda_0, \infty)\setminus J$. But since the curves $\gamma_i$ are all in $\mathcal{G}_{T, \theta_0, \varepsilon}$ and there is a complete geodesic graph structure $(\mathcal{V}, \mathcal{E})$ with $\mathcal{V} \subset \mathcal{G}_{T, \theta_0, \varepsilon}$ and $\mathcal{E} \subset \mathcal{H}_{\theta_1, r, c_0}$, Theorem \ref{thm:injectivity_sign} guarantees that $A_1$ and $A_2$ are gauge equivalent if $P^{A_1}_\gamma = P^{A_2}_\gamma$ for all $\gamma \in \mathcal{G}_{T, \theta_0, \varepsilon}$ and $S^{A_1}(v_x, w_x) = \pm S^{A_2}(v_x, w_x)$ for all pairs $(v_x, w_x) \in \mathcal{H}_{\theta_1, r, c_0}$.
\end{proof}

\bibliographystyle{alpha}

\bibliography{refBrokenDN}

\end{document}